\documentclass{amsart}
\usepackage{times,mathrsfs,array,amssymb,pb-diagram}
\usepackage[toc,page,title,titletoc,header]{appendix} 
\usepackage{multirow,float}

\newcounter{lemma}
\newtheorem{Theorem}{Theorem}
\newtheorem{Lemma}[lemma]{Lemma}
\newtheorem{Corollary}[lemma]{Corollary}
\newtheorem{Proposition}[lemma]{Proposition}
\newtheorem{theorem}{Theorem}

\theoremstyle{definition}
\newtheorem{Example}[lemma]{Example}
\newtheorem{Definition}[lemma]{Definition}

\newtheorem{Remark}[lemma]{Remark}
\newtheorem{Assumption}[lemma]{Assumption}

\def\H{\mathfrak H}
\def\P{\mathbb P}

\def\tr{\operatorname{tr}}
\def\nm{\operatorname{nm}}

\def\C{\mathbb C}
\def\O{\mathcal O}
\def\P{\mathbb P}
\def\Q{\mathbb Q}
\def\R{\mathbb R}
\def\Z{\mathbb Z}

\def\mod{\  \mathrm{mod}\ }

\def\sgn{\operatorname{sgn}}

\def\gen#1{\langle #1\rangle}
\def\JS#1#2{\left(\frac{#1}{#2}\right)}
\def\gauss#1{\left\lfloor #1\right\rfloor}
\def\SL{\mathrm{SL}}
\def\CM{\mathrm{CM}}
\def\wt{\widetilde}
\def\M#1#2#3#4{\begin{pmatrix}#1&#2\\#3&#4\end{pmatrix}}
\def\SM#1#2#3#4{\left(\begin{smallmatrix}#1&#2\\#3&#4\end{smallmatrix}
  \right)}
\def\div{\operatorname{div}}
\def\Div{\operatorname{Div}}
\def\Pic{\operatorname{Pic}}
\def\Jac{\operatorname{Jac}}
\def\spin{\operatorname{spin}}

\begin{document}

\title{Equations of hyperelliptic Shimura curves}

\author{Jia-Wei Guo}
\address{Department of Mathematics, National Tsing Hua University,
  Hsinchu, Taiwan 300}


\author{Yifan Yang}
\address{Department of Applied Mathematics, National Chiao Tung
  University, Hsinchu, Taiwan 300}  
\email{yfyang@math.nctu.edu.tw}
\date{\today}
\thanks{The authors were partially supported by Grant
  102-2115-M-009-001-MY4 of the Ministry of Science and Technology,
  Taiwan (R.O.C.).}
\thanks{
  The AMPL code for solving integer programming problems appearing in
  this work was written by Mr. Yi-Hsuan Lin. The authors would like to
  thank Mr. Lin for his masterfully written code. It would have been
  impossible to construct required Borcherds forms without the code.} 

\begin{abstract}
By constructing suitable Borcherds forms on Shimura curves and using
Schofer's formula for norms of values of Borcherds forms at CM-points, we
determine all the equations of hyperelliptic Shimura curves
$X_0^D(N)$. As a byproduct, we also address the problem of whether a
modular form on Shimura curves $X_0^D(N)/W_{D,N}$ with a divisor
supported on CM-divisors can be realized as a Borcherds form, where
$X_0^D(N)/W_{D,N}$ denotes the quotient of $X_0^D(N)$ by all the
Atkin-Lehner involutions.
 The construction of Borcherds forms is done by solving
certain integer programming problems.
\end{abstract}
\maketitle

\begin{section}{Introduction}

For an indefinite quaternion algebra $B$ of discriminant $D$ over $\Q$
and a positive integer $N$ with $(D,N)=1$, we let $X_0^D(N)$ be
the Shimura curve associated to an Eichler order $\O$ of level $N$ in
$B$. When $D=1$, the Shimura curve $X_0^D(N)$ is simply the classical
modular curve $X_0(N)$, which is the coarse moduli space of elliptic
curves together with a cyclic subgroup of order $N$ and has been
extensively studied in literature. When $D>1$, the curve $X_0^D(N)$ is
the coarse moduli space of principally polarized abelian surfaces with
multiplication by $\O$. The arithmetic of such a Shimura curve is
similar to those of classical modular curves, but the lack of cusps makes the
Diophantine geometry and explicit calculation of such a Shimura curve
more interesting and challenging than those of classical modular
curves. The primary purpose of the present paper is to address the
problem of determining equations of Shimura curves.

In the classical modular case, which have been extensively studied and
well-known for admitting Fourier expansions around the cusps, there
are many constructions of modular forms and modular functions,
such as Eisenstein series, the Dedekind function, theta series, and
etc., and there are formulas for their Fourier expansions. Thus, it is
often easy to determine equations of modular curves. We refer the
reader to Galbraith {~\cite{Galbraith}}, Yang \cite{Yang-AIM}, and the
references contained therein for more informations about equations of modular
curves.

On the other hand, when $D\neq 1$, the absence of cusps has been an
obstacle for explicit approaches to Shimura curves since modular forms
or modular functions on Shimura curves do not have Fourier expansions
and as a result,  most of the methods for classical modular curves
cannot possibly be extended to the case of general Shimura curves.
Up to now, only a few equations of Shimura curves are known. Ihara
\cite{Ihara} was perhaps the first one to give defining equations of Shimura
curves. For example, he found an equation for the curve $X^6_0(1)$ of
genus $0$. Kurihara \cite{Kurihara} extended Ihara's method and
determined equations of $X^{10}_0(1)$ and $X^{22}_0(1)$ of genus $0$
and $X^{14}_0(1)$, $X^{21}_0(1)$, and $X^{46}_0(1)$ of genus $1$.
Jordan \cite{Jordan} computed equations of two Shimura curves
$X^{15}_0(1)$ and $X^{33}_0(1)$ of genus $1$. Later on, Gonz\'alez and Rotger
\cite{{Victor-genus-two},{Victor-genus-one}} completed the list of
equations of Shimura curves $X_0^D(N)$ of genus $1$ and $2$.
For Shimura curves $X_0^D(N)$ of higher genus, Elkies \cite{Elkies-K3}
found equations of Shimura curves $X^{57}_0(1)$ and  $X^{206}_0(1)$
using the fact that some families of $K3$ surfaces are parameterized
by Shimura curves. More recently, Molina \cite{Molina-hyperelliptic} 
found equations of $X^{39}_0(1)$ and $X^{55}_0(1)$ and Atkin-Lehner
quotients of some Shimura curves. Also, Gonz\'alez and Molina
\cite{Molina-isogeny} determine equations of all Shimura curves
$X_0^D(1)$ of genus $3$. (Note that it happens that all these curves
are hyperelliptic.) We remark that all the methods in the above-mentioned
works other than those in \cite{Elkies-K3} are 
strongly based on the Cerednik-Drinfeld theory of $p$-adic
uniformization of Shimura curves \cite{Boutot-Carayol}, $p|D$, and arithmetic properties
of CM-points. In addition, other than \cite{Elkies-K3}, their
methods do not allow us to locate general CM-points on the curves.

In this paper, we will adopt a very different approach, using the
theory of Borcherds forms and explicit formulas for values of
Borcherds forms at CM-points to obtain equations of Shimura
curves. (See Section \ref{section: Borcherds forms} for a quick
introduction to Borcherds forms.) The main result of this paper is a
complete list of equations of all hyperelliptic Shimura curves $X_0^D(N)$.

\begin{Theorem} \label{theorem: main}
The table in Appendix A gives a complete list of equations of
hyperelliptic Shimura curves $X_0^D(N)$, $D>1$.
\end{Theorem}

The idea of realizing modular forms on Shimura curves as Borcherds 
forms is not new. For example, as a corollary to his formula for
average values of Borcherds forms at CM-points, Schofer \cite{Schofer}
proved a weak analogue of Gross and Zagier's result
\cite{Gross-Zagier} on the prime factorization of the norm of the
difference of two singular moduli on the classical modular curve
$X_0(1)$ for the case of Shimura curves. Later on, Errthum
\cite{Errthum} applied Schofer’s formula to
compute singular moduli on $X^6_0(1)/W_{6,1}$ and
$X^{10}_0(1)/W_{10,1}$, verifying Elkies’ numerical computation
\cite{Elkies-computation}, where $W_{D,N}$ denotes the full
Atkin-Lehner group on $X_0^D(N)$. However, applications of
Borcherds forms to theory of Shimura curves were not explored any
further in literature. One possible reason is that in order to
successfully use Borcherds forms to do computation on Shimura curves,
one needs a systematic method to construct them in the first place,
but such a method is not developed in literature yet. Thus, our first
task here is to develop a systematic method to construct Borcherds
forms. We will see that the problem of constructing Borcherds forms
reduces to that of solving certain integer programming problems,
which we solve by using the AMPL modeling language
(http://www.ampl.com) and the Gurobi solver (http://www.gurobi.com).

Note that our method works for any Shimura curve $X_0^D(N)$ such that
$X_0^D(N)/W_{D,N}$ has genus $0$, but because there are too many of
them, here we consider only the hyperelliptic cases. (There are more
than $110$ non-hyperelliptic Shimura curves $X_0^D(N)$ whose
Atkin-Lehner quotient $X_0^D(N)/W_{D,N}$ has genus $0$.) In addition,
under a certain technical assumption (Assumption below), it is also possible to determine
equations of $X_0^D(N)/W_{D,N}$ even if it is not of genus $0$. In
Section \ref{subsection: additional}, we give two such
examples. However, the method is less systematic and it is not clear
whether it will always work in general.

In principle, our list of equations should also be
obtainable using Elkies' approach \cite{Elkies-K3}, but our approach
via Borcherds forms have potential applications to other problems
about Shimura curves beyond the scope of the present paper. To
illustrate our point, in Section \ref{section: height}, we will show
how our construction of Borcherds forms lead to a method to compute
heights of CM-points on Shimura curves, again, under Assumption
\ref{assumption: 1}. Note that both Elkies' and our approaches have an
advantage over other methods in that we can determine the coordinates
of CM-points on Shimura curves.

The rest of the paper is organized as follows. In Section
\ref{section: Borcherds forms}, we give a quick overview of the theory of
Borcherds forms and explain the idea of realizing modular forms on
Shimura curves in terms of Borcherds forms. The exposition of
this section follows \cite{Yang-CM}. In Section
\ref{section: construction}, we discuss how to
construct Borcherds forms by solving certain integer programming
problems. For our purpose, the case of odd $D$ needs special
attention. As a byproduct, we find that for $(D,N)$ in Theorem
\ref{theorem: main} with even $D$, all meromorphic modular forms with
divisors supported on CM-divisors (Definition \ref{definition:
  CM-divisor}) can be realized as Borcherds forms. (We believe that
this is also true for odd $D$, but since it is not the main problem we
are concerned with, we will not prove this assertion here.) In Section
\ref{section: equations}, we will give several examples illustrating
how to obtain equations of Shimura curves using Borcherds forms we 
constructed in Section \ref{section: construction} and Schofer's
formula for values of  Borcherds forms at CM-points. In
Section \ref{subsection: additional}, we give additional 
examples where the genus of $X_0^D(N)/W_{D,N}$ is not
$0$. Specifically, we determine equations of $X_0^{142}(1)/W_{142,1}$
and $X_0^{302}(1)/W_{302,1}$, under Assumption \ref{assumption: 1}.
Finally, in Section \ref{section: height} we demonstrate how to
explicitly compute heights of CM-points on Shimura curves using our
construction of Borcherds forms.
\end{section}

\begin{section}{Borcherds forms}
\label{section: Borcherds forms}
\begin{subsection}{Basic theory} We give a quick introduction to
  Borcherds forms. For details, see {\cite{Borcherds-Invent,
  Borcherds-Duke, Bruinier}} for the classical setting,
see {\cite{Errthum, Kudla-Integral, Schofer}} for the adelic setting.

Let $L$ be an even lattice with symmetric bilinear form $\left\langle
  \cdot,\cdot \right\rangle$ of signature $(n,2)$ and $L^{\vee}$ be
the dual lattice of $L$. We assume $L$ is nondegenerate and
denote by
$$
  \{e_{\eta}:\eta\in L^{\vee}/L\}
$$
the standard basis for the group algebra $\mathbb{C}[L^{\vee}/L]$.
Associated to the lattice $L$, we have a unitary \emph{Weil
  representation} $\rho_L$ of the metaplectic group 
$$\widetilde{{\SL}}(2,\Z)=\left\lbrace 
\left(\begin{pmatrix}
a & b\\
c & d
\end{pmatrix},\sqrt{c\tau+d}\right):\begin{pmatrix}
a & b\\
c & d
\end{pmatrix}\in {\SL}(2,\Z)
\right\rbrace 
$$ 
on the group algebra $\C[L^{\vee}/L]$ defined by
\[\begin{split}
\rho_{L}(T)e_{\eta}=&e^{-2\pi i\left\langle \eta,\eta \right\rangle/2 }e_{\eta},\\
\rho_{L}(S)e_{\eta}
=&\frac{e^{2\pi i(n-2)/8}}{\sqrt{|L^{\vee}/L|}}\displaystyle \sum_{\delta \in L^{\vee}/L}e^{2\pi i\left\langle \eta,\delta\right\rangle }e_{\delta},
\end{split}\]
where
$$
S=\left(\begin{pmatrix}
0 & -1\\
1 & 0
\end{pmatrix},\sqrt{\tau}\right)\quad \text{and}\quad
T=\left(\begin{pmatrix}
1 & 1\\
0 & 1
\end{pmatrix},1\right),
$$
which generate $\widetilde\SL(2,\Z)$.

\begin{Definition} A holomorphic function $F:\H\rightarrow
  \C[L^{\vee}/L]$ is called a weakly holomorphic vector-valued modular
  form of weight $k\in\frac{1}{2}\Z$ and {type} $\rho_L$ on
  $\widetilde{{\SL}}(2,\Z)$ if it satisfies
$$
  F\left(\frac{a\tau+b}{c\tau+d}\right)=(c\tau+d)^k\rho
  \left(\begin{pmatrix} a & b\\ c &
      d \end{pmatrix},\sqrt{c\tau+d}\right)
  F(\tau)
$$
for all $\tau\in\H$ and all $\SM abcd\in\SL(2,\Z)$ and $F$ is
meromorphic at the cusp $\infty$. The last condition means that the
Fourier expansion of $F$ is of the form
$$
  F(\tau)=\sum_{\eta\in L^{\vee}/L}
  \sum_{m\in\Z+\left\langle \eta,\eta\right\rangle /2, ~m>-m_0}
  c_{\eta}(m)q^{m}{e_{\eta}}, \quad q=e^{2\pi\tau},
$$
for some rational number $m_0$.
\end{Definition}

For $k=\Q,\ \R$ or $\C$, let $V(k)=L\otimes k$ and extend the
definition of $\left\langle \cdot,\cdot \right\rangle$ to $V(k)$ by
linearity. Define $O_V(\R)$ to be the orthogonal group of the bilinear
form $\left\langle \cdot,\cdot \right\rangle$ and its subgroup
$$O^+_V(\R):=\{\sigma\in O_V(\R):\spin\sigma= \sgn\det\sigma \}, $$
where if $\sigma$ is equal to the product of $n$ reflections with
respect to the vectors $v_1,\ldots,v_n$, then its spinor norm is
defined by $\spin\sigma=(-1)^n\prod^n_{i=1}\sgn\left\langle v_i,v_i
\right\rangle$. We also define
$$
  O^+_L:=\{\sigma\in O^+_V(\R): \sigma(L)= L \}
$$
to be the orthogonal group of the lattice $L$. As the
orthogonal group $O^+_L$ acts on the dual lattice $L^{\vee}$, there is
an induced operation on $\C[L^{\vee}/L]$ given by
$$
  \sum_{\eta\in L^{\vee}/L}c_{\eta}e_{\eta}\longmapsto\sum_{\eta\in
  L^{\vee}/L}c_{\eta}e_{\sigma\eta}, \qquad \sigma\in O^+_L.
$$

\begin{Definition}
Suppose that $F=\sum_{\eta\in L^{\vee}/L}F_\eta e_\eta$ is a vector-valued
modular form. We define the {automorphism group} $O^+_{L,F}$ of $F$ by
\begin{equation*}
O^+_{L,F}=\{\sigma \in O^+_L:F_{\sigma\eta}=F_{\eta}
  \text{ for all }\eta \text{ in } L^{\vee}/L\}.
\end{equation*}

\end{Definition}

Consider the subset
$$
  K=\left\{ \left[ z\right] \in \mathbb{P}\left( V\left( \mathbb{C}
  \right) \right) :\left\langle z,z \right\rangle =0,
  \ \left\langle z,\bar{z} \right\rangle <0\right\}
$$
of the projective space $\mathbb{P}(V(\C))$. This set $K$ consists of
two connected components and the orthogonal group $O^+_V(\R)$
preserves the components. Pick one of them to be $K^+$. Then it can be
checked that $O^+_V(\R)$ acts transitively on $K^+$.

\begin{Definition}
Suppose $\widetilde{K}^+=\left\{ z \in V\left( 
\mathbb{C}\right) :\left[ z\right] \in K^+ \right\}$.
For each subgroup of $\Gamma$ of finite index of $O^+_L$, we call a
meromorphic function $\Psi:\widetilde{K}^+\rightarrow \mathbb{P}(\C)$
a {modular form of weight} {$k$} and character $\chi$ on $\Gamma$
if $\Psi$ satisfies 
\begin{enumerate}
\item[(i)] ${\Psi}(cz)=c^{-k}\widetilde{\Psi}(z)$ for all $c\in
    {\C}^\ast$ and $z\in\widetilde{K}$,
\item[(ii)] ${\Psi}(hz)=\chi (h)\widetilde{\Psi}(z)$ for all $h\in \Gamma$
  and $z\in\widetilde{K}$.
\end{enumerate}
\end{Definition}

\begin{theorem}[{\cite[Theorem 13.3]{Borcherds-Invent}}]\label{B}
Let $L$ be an even lattice of signature $(n,2)$ and $F(\tau)$ be a
weakly holomorphic vector-valued modular forms of weigh $1-{n}/{2}$
and type $\rho_L$ with Fourier expansion $ F(\tau)=\sum_{\eta}\left(
  \sum_{n}c_{\eta}(n)q^n\right){e_{\eta}}$. Suppose that
$c_{\eta}(n)\in\Z$ for any $\eta\in L^{\vee}/L$ and $n\leq 0$. Then
there corresponds a meromorphic function $\Psi_F(z),\ z\in
\widetilde{K}^+$ with the following properties.
\begin{enumerate}
\item[(i)] $\Psi_F(z)$ is a meromorphic modular forms of weight $c_0(0)/2$
  for the group ${O}^+_{L,F}$ with respect to some unitary character
  $\chi$ of ${O}^+_{L,F}$.
\item[(ii)] The only zeros or poles of $\Psi_F(z)$ lie on the rational
  quadratic divisor 
$$
  \lambda^{\perp}=\left\{z\in\widetilde{K}^+:\left\langle z,\lambda
  \right\rangle=0\right\}
$$
for $\lambda$ in $L$, $\left\langle \lambda,\lambda\right\rangle>0$,
and are of order
$$
  \sum_{0<r\in \Q,r\lambda\in
    L^{\vee}}c_{r\lambda}\left(-r^2\left\langle
  \lambda,\lambda \right\rangle/2\right).
$$
\end{enumerate}
\end{theorem}
\begin{Definition}
We call the function $\Psi_F(z)$ the \emph{Borcherds form} associated to $F$.
\end{Definition}

\end{subsection}

\begin{subsection}{Borcherds forms on Shimura curves}
\label{subsection: Borcherds on Shimura}
We now explain how to realize modular forms on Shimura curves as
Borcherds forms. We follow the exposition in \cite{Yang-CM}. See also
\cite{Errthum}.

Let $B$ be an indefinite quaternion algebra of discriminant $D$ over
$\Q$. Consider the vector space
$$
  V=V(\Q)=\{x\in B:\tr(x)=0\}
$$
over $\Q$ with the natural bilinear form $\left\langle
  x,y\right\rangle=\tr(x\overline y)=-\tr(xy)$. Then $V$ has signature
$(1,2)$ and the associated quadratic form is $\nm(x)=-x^2$. Given an
Eichler order $\O$ of level $N$ in $B$, we let $L$ be the lattice
$$
  L=\O\cap V=\{x\in \O:\tr(x)=0\}.
$$
For an invertible element $\beta$ in $B\otimes\R$, define
$\sigma_\beta:V(\R)\to V(\R)$ by
$\sigma_\beta(\gamma)=\beta\gamma\beta^{-1}$. Then, we can show that
$$
  O^+_V(\R)=\{\sigma_{\beta}:\beta \in (B\otimes \R)^\ast/{\R}^\ast,
  \nm(\beta)> 0\}\times \{\pm 1\}
$$
and 
$$
  O^+_L=\{\sigma_{\beta}:\beta \in N^+_B(\O )/{\Q}^\ast\}\times \{\pm
  1\}.
$$
If we assume that the quaternion algebra is represented by
$B=(\frac{a,b}{\Q})$ with $a>0$ and $b>0$, that is, $B=\Q+\Q i+\Q j+\Q
ij$ with $i^2=a,\ j^2=b,$ and $ij=-ji$, and fix an embedding
$\iota:B\hookrightarrow M(2,\R)$ by
$$
  \iota: i\longmapsto\M0{\sqrt a}{\sqrt a}0, \qquad
  j\mapsto\M{\sqrt b}00{-\sqrt b},
$$
then each class in $K=\{z\in
\mathbb{P}(V(\C)):\left\langle z,z\right\rangle=0, \left\langle
  z,\bar{z}\right\rangle<0\}$ contains a unique representative of the
form  
$$
  z(\tau)=\frac{1-{\tau}^2}{2\sqrt{a}}i+\frac{\tau}{\sqrt{b}}j
  +\frac{1+{\tau}^2}{2\sqrt{ab}}ij
$$
for some $\tau\in \H^{\pm} $, the union of upper and lower
half-plane. The mapping $\tau \mapsto z(\tau)\mod \C^\ast$ is a
bijection of between $\H^{\pm}$ and $K$.

Let $K^+$ be the image of $\H^+=\H$ under the mapping. Then we get
{compatible} actions of $N_B^+(\O)/\Q^\ast$ on $K^\ast$ and $\H$
with the action on $K^+$ by {conjugation} and the action on $\H$ by
{linear fraction transformation}. More precisely, this means that for
$\alpha\in N_B^+(\O)$, if we write $\iota(\alpha)=\SM{c_1}{c_2}{c_3}{c_4}$, then
\begin{equation}\label{equation: compatibility}
\begin{split}
  \alpha z(\tau)\alpha^{-1}=\frac{(c_3\tau+c_4)^2}{\nm(\alpha)}
  z\left(\frac{c_1\tau+c_2}{c_3\tau+c_4}\right)
  \equiv z(\iota(\alpha)\tau) \mod \C^\ast.
\end{split}
\end{equation}

\begin{Lemma}[{~\cite[Lemma 4]{Yang-CM}}] 
Let $ F(\tau)=\sum_{\eta}\left(
  \sum_{n}c_{\eta}(n)q^{n}\right){e_{\eta}}$ be a weakly holomorphic
vector-valued modular form of weight ${1}/{2}$ and type $\rho_L$ such
that $O^+_{L,F}=O^+_L$ and $c_{\eta}(n)\in \Z$ whenever $\eta\in
L^\vee/L$ and $n\leq 0$. Then the function $\psi_F(\tau)$ defined by
$\psi_F(\tau)=\Psi_F(z(\tau))$ is a meromorphic modular forms of
weight $c_0(0)$ with certain unitary character $\chi$ on the Shimura
curve $X^D_0(N)/W_{D,N}$.
\end{Lemma}


\begin{Definition} \label{definition: Borcherds on Shimura}
With assumptions given as in the lemma, the function $\psi_F(\tau)$
defined by $$\psi_F(\tau)=\Psi_F(z(\tau))$$ is called the
\emph{Borcherds forms} on the Shimura curve $X^D_0(N)/W_{D,N}$
associated to $F$.
\end{Definition}

The next lemma gives us the criterion when the character of a
Borcherds form $\psi_F(\tau)$ is trivial, under the assumption that
the genus of $N^+_B(\O)\backslash\H$ is zero.

\begin{Lemma}[{~\cite[Lemma 6]{Yang-CM}}] \label{lemma: genus 0}
Assume that the genus of $X=N^+_B(\mathcal{O})\backslash\H$ is zero. Let
$\tau_1,\cdots,\tau_r$ be the elliptic points of $X$ and assume that
their orders are $b_1,\ldots,b_r,$ respectively. Assume further that,
as $\mathrm{CM}$-points, the discriminant of $\tau_1,\cdots,\tau_r$
are $d_1,\cdots,d_r,$ respectively. Let $ F(\tau)=\sum_{\eta}\left(
  \sum_{m}c_{\eta}(n)q^m\right){e_{\eta}}$ be a weakly holomorphic
vector-valued modular form of weight ${1}/{2}$ and type $\rho_L$ such
that ${O}^+_{L,F}=O^+_L$ and $c_\eta(m)\in\Z$ whenever $\eta\in
L^\vee/L$ and $m\leq 0$. Assume that $c_0(0)$ is even. Then the
Borcherds form $\psi_F(\tau)$ is a modular form with trivial character
on $X$ if and only if for $j$ such that $b_j\neq 3,$ the order of
$\Psi_F(z)$ at $z(\tau_j)$ has the same parity as $c_0(0)/2$. 
\end{Lemma}


We now state Schofer's formula {~\cite[Corollaries 1.2 and
  3.5]{Schofer}} in the setting of Shimura curves as follows.

\begin{theorem}[{\cite[Corollaries 1.2 and 3.5]{Schofer}}]
\label{Schofers formula} Let 
$F(\tau)=\sum_{\eta}\left( \sum_{m}c_{\eta}(n)q^{m}\right){e_{\eta}}$
be a weakly holomorphic vector-valued modular form of weight $1/2$ and
type $\rho_L$ for $\widetilde{\SL}(2,\Z)$ such that $O^+_{L,F}=O^+_L$,
$c_0(0)=0$, and $c_{\eta}(m)\in \Z$ whenever $\eta\in L^\vee/L$ and
$n\leq 0$. Let $d<0$ be a fundamental discriminant such that the set
$\mathrm{CM}(d)$ of $\mathrm{CM}$-points of discriminant $d$ on
$N_B^+(\O)\backslash\H$ is not empty and that the support of
$\div\psi(\tau)$ does not intersect $\mathrm{CM}(d)$. Then we have
$$
  \sum_{\tau \in\mathrm{CM}(d)}\mathrm{log}|\psi_F(\tau)|
=-\frac{|\mathrm{CM}(d)|}{4}\displaystyle\sum_{\gamma\in
  L^{\vee}/L} \sum_{m\geq 0}c_{\gamma}(-m)\kappa_{\gamma}(m),
$$ 
where $\kappa_{\gamma}(m)$ are certain sums involving derivatives of
Fourier coefficients of some incoherent Eisenstein series.
\end{theorem}

We refer the reader to \cite{Errthum,Yang-CM} for strategies to
compute $\kappa_\gamma(m)$ explicitly.
\end{subsection}
\end{section}

\begin{section}{Construction of Borcherds forms}
\label{section: construction}

\begin{subsection}{Errthum's method}
  In this section, we will review Errthum's method \cite{Errthum} for
  constructing vector-valued modular forms out of scalar-valued
  modular forms. Here the notations $D$, $N$, $\O$, $L$, and etc. have
  the same meanings as those in Section \ref{subsection: Borcherds on
    Shimura}. The level $N$ is always assumed to be squarefree.

  Let us first describe the structure of the lattice $L$.

\begin{Lemma} \label{lemma: basis for O}
  Assume that $N$ is squarefree.
  Let $q$ be a prime number such that $q\equiv 1\mod 4$ and
  \begin{equation} \label{equation: condition on q}
    \JS qp=\begin{cases}
    -1, &\text{if }p|D, \\
     1, &\text{if }p|N. \end{cases}
  \end{equation}
  Then $B=\JS{DN,q}\Q$ is a quaternion algebra of discriminant $D$ over
  $\Q$. Moreover, let $a$ be an integer such that $a^2DN\equiv 1\mod
  q$. Then the $\Z$-module $\O$ generated by
  \begin{equation} \label{equation: ei}
    e_1=1, \quad e_2=\frac{1+j}2, \quad e_3=\frac{i+ij}2, \quad
    e_4=\frac{aDNj+ij}q
  \end{equation}
  is an Eichler order of level $N$ in $B$. Also, let $L$ be the set of
  elements of trace zero in $\O$ and let
  \begin{equation} \label{equation: ell i}
    \ell_1=j, \qquad \ell_2=\frac{i+ij}2, \qquad
    \ell_3=\frac{aDNj+ij}q.
  \end{equation}
  Then
  $$
    L=\Z\ell_1+\Z\ell_2+\Z\ell_3, \qquad
    L^\vee=\Z\frac{\ell_1}2+\Z\frac{\ell_2}{DN}+\Z\frac{\ell_3}{DN}.
  $$
\end{Lemma}

\begin{proof} The conditions in \eqref{equation: condition on q} imply
  that $B$ is ramified at prime divisors of $D$ and unramified at
  prime divisors of $N$. Also, by the quadratic reciprocity law, we
  have $\JS{DN}q=1$. Thus, the discriminant of $B$ is $D$.

  We check that
  \begin{equation*}
  \begin{split}
    e_2^2&=\frac{q-1}4e_1+e_2, \\
    e_2e_3&=\frac{aDN(q-1)}4e_1+\frac{aDN(1-q)}2e_2+\frac{1-q}2e_3
    +\frac{q(q-1)}4e_4, \\
    e_2e_4&=aDNe_1-aDNe_2-e_3+\frac{q+1}2e_4, \\
    e_3^2&=\frac{DN(1-q)}4e_1, \\
    e_3e_4&=-\frac{DN(a^2DN(q-1)+q+1)}{2q}e_1
     +\frac{DN(a^2DN(q-1)+1)}qe_2 \\
    &\qquad\quad+aDNe_3+\frac{aDN(1-q)}2e_4,
  \end{split}
  \end{equation*}
  so that $\Z e_1+\Z e_2+\Z e_3+\Z e_4$ is an order in $B$.
  Also, the Gram matrix
$$
  (\tr(e_i\overline e_j))=\begin{pmatrix}
  2 & 1 & 0 & 0\\ 1 & (q-1)/2 & 0 & -aDN \\
  0 & 0 & DN(q-1)/2 & DN \\ 0 & -aDN & DN & 2DN(1-a^2DN)/q \end{pmatrix}
$$
has determinant $(DN)^2$. Thus, it is an Eichler order of level $N$.

Moreover, it is clear that $\ell_1$, $\ell_2$, and $\ell_3$ span
$L$. Also, the Gram matrix of $L$ with respect to this basis is
\begin{equation} \label{equation: Gram L}
  \begin{pmatrix}
  -2q & 0 & -2aDN \\ 0 & DN(q-1)/2 & DN \\ -2aDN & DN
  & 2DN(1-a^2DN)/q \end{pmatrix},
\end{equation}
and its determinant is $2D^2N^2$. From the Gram matrix of $L$, it is
easy to check that $L^\vee$ is spanned byt $\ell_1/2$, $\ell_2/DN$ and
$\ell_3/DN$. This proves the lemma.
\end{proof}

\begin{Corollary}\label{corollary: level}
Assume that $N$ is squarefree.
The discriminant of the lattice $L$ is $$|L^{\vee}/L|=2(DN)^2$$ and
the level of $L$ is
\begin{equation*}
\begin{cases}4DN,& \text{ if }DN\text{ is odd}, \\
2DN,&\text{if }DN\text{ is even}.
\end{cases}
\end{equation*}
\end{Corollary}

\begin{proof}
The result follows directly from the proof in above lemma since the
determinant of the Gram matrix in \eqref{equation: Gram L} is $2(DN)^2$
and $L^\vee/L\simeq(\Z/2)\times(\Z/DN)^2$.
\end{proof}

We now recall Errthum's method \cite{Errthum} for constructing weakly
holomorphic vector-valued modular forms. Let
  $\chi_\theta$ denote the character associated to the Jacobi theta
  function $\theta(\tau)=\sum_{n\in\Z}q^{n^2}$. That is, $\chi_\theta$
  is defined by
  $$
    \theta(\gamma\tau)=\chi_\theta(\gamma)(c\tau+d)^{1/2}\theta(\tau)
  $$
  for all $\gamma=\SM abcd\in\Gamma_0(4)$ and all $\tau\in\H$.

\begin{Lemma}[{~\cite[Theorem 4.2.9]{Bar}}]\label{lemma: Barnard}
  Let $M$ be the level of the lattice $L$. Suppose that $f(\tau)$ is a
  weakly holomorphic scalar-valued modular form of weight $1/2$ such
  that
  $$
    f(\gamma\tau)=\chi_\theta(\gamma)(c\tau+d)^{1/2}f(\tau)
  $$
  for all $\gamma=\SM abcd\in\Gamma_0(M)$. Then the function
  $F_f(\tau)$ defined by
  \begin{equation} \label{equation: Ff}
    F_f(\tau)=\sum_{\gamma\in\wt{\Gamma}_0(M)\backslash
    \wt\SL(2,\Z)}f\big|_\gamma(\tau)\rho_L(\gamma^{-1})e_0
  \end{equation}
  is a weakly holomorphic vector-valued modular form of weight $1/2$
  and type $\rho_L$.
\end{Lemma}

\begin{Lemma}[{\cite[Theorem 5.8]{Errthum}}]
  \label{lemma: OLf=OL}
  Let $f(\tau)$ and $F_f(\tau)$ be given as in previous lemma.
  Then for $\eta$ and $\eta'\in L^{\vee}/L$ with
  $\left\langle \eta,\eta\right\rangle=\left\langle
    \eta',\eta'\right\rangle$, the $e_{\eta}$ component and
  $e_{\eta'}$ component of $F_f(\tau)$ are equal. Consequently, we
  have $O^+_{L,F_f}=O^+_L$.
\end{Lemma}

\begin{Lemma}[{\cite[Theorem 6.2]{Borcherds-Duke}}]
  \label{lemma: eta}
  Let $M$ be the level of the lattice $L$. Suppose that $r_d$, $d|M$,
  are integers satisfying the conditions
  \begin{enumerate}
  \item[(i)] $\sum_{d|M} r_d=1$,
  \item[(ii)] $|L^\vee/L|\prod_{d|M} d^{r_d}$ is a square in $\Q^*$,
  \item[(iii)] $\sum_{d|M} dr_d\equiv 0\mod 24$, and
  \item[(iv)] $\sum_{d|M} (M/d)r_d\equiv 0\mod 24$.
  \end{enumerate}
  Then $\prod_{d|M}\eta(d\tau)^{r_d}$ is a weakly holomorphic
  scalar-valued modular form satisfying the condition for $f(\tau)$ in
  Lemma \ref{lemma: Barnard}.
\end{Lemma}

\begin{Definition} \label{definition: admissible eta}
  If an eta-product satisfies the conditions in
  Lemma \ref{lemma: eta}, then we say it is \emph{admissible}.
\end{Definition}

To have a better control over the divisors of Borcherds forms
constructed, we will use certain special admissible eta products.

\begin{Definition} Let $M$ be the level of the lattice $L$ and let $S$ be a
  subset of the cusps of $\Gamma_0(M)$. If $f$ is a weakly holomorphic
  modular form of weight $1/2$ on $\Gamma_0(M)$ whose only poles are at the cusps
  in $S$, then we say $f$ is a $S$-weakly holomorphic scalar-valued
  modular form of weight $1/2$ on $\Gamma_0(M)$.
\end{Definition}

Later on, we will use $\{\infty\}$-weakly holomorphic
modular forms to construct Borcherds forms for the case of even $D$
and $\{\infty,0\}$-weakly holomorphic modular forms for the case of
odd $D$. Therefore, let us introduce the following definitions.

\begin{Definition}
  Let $D_0$ be the odd part of $DN$. We let $M^!(4D_0)$ denote the
  space of all $\{\infty\}$-weakly holomorphic modular forms of weight
  $1/2$ on $\Gamma_0(4D_0)$. Also, for a nonnegative integer $n$, let
  $M^!_n(4D_0)$ be the subspace of $M^!(4D_0)$ consisting of modular
  forms with a pole of order $\le n$ at $\infty$.
  If $j$ is a positive integer such that there does not exist a
  modular form in $M^!(4D_0)$ with a pole of order $j$ at $\infty$,
  then we say $j$ is a \emph{gap} of $M^!(4D_0)$.

  Similarly, we let $M^{!,!}(4D_0)$ be the space of all
  $\{\infty,0\}$-weakly holomorphic modular forms of weight $1/2$ on
  $\Gamma_0(4D_0)$. For nonnegative integers $m$ and $n$, let
  $M^{!,!}_{m,n}(4D_0)$ be the subspace of $M^{!,!}(4D_0)$ consisting
  of modular forms with a pole of order $\le m$ at $\infty$ and a pole
  of order $\le n$ at $0$.
\end{Definition}


\begin{Remark}
  Notice that the space $M^!_0(4D_0)$ is simply the space of holomorphic
  modular forms of weight $1/2$ on $\Gamma_0(4D_0)$. Since $D_0$ is
  assumed to be squarefree, by Theorem A of \cite{Serre-Stark}, the
  space $M^!_0(4D_0)$ is one-dimensional and spanned by $\theta(\tau)$.
\end{Remark}

\end{subsection}

\begin{subsection}{Case of even $D$}
\label{subsection: D even}

In this section, we assume that $D$ is even and $N$ is squarefree.
In Proposition \ref{proposition: admissible span}, we will
see how the problem of constructing Borcherds forms becomes the
problem of solving certain integer programming problem. Ultimately, in
Proposition \ref{proposition: all Borcherds}, we will show that for
$(D,N)$ in Theorem \ref{theorem: main} with $2|D$, every meromorphic
modular form of even weight on $X_0^D(N)/W_{D,N}$ with divisor
supported on CM-divisors (see Definition \ref{definition: CM-divisor})
can be realized as a Borcherds form. Note that Bruinier
\cite{Bruinier-converse} and Heim and Murase \cite{Heim-Murase}
studied when a modular form on an orthogonal group $O(n,2)$ can be
realized as a Borcherds form, but as the integer $n$ is assumed to be
at least $2$, their results do not apply to the case of Shimura
curves. In fact, it is pointed out in
\cite[Section 1]{Bruinier-converse} that counterexamples exist in the
case $n=1$ (see also \cite[Section 8.3]{Bruinier-Ono}).
It will be a very interesting problem to characterize those
modular forms on Shimura curves 
  $X_0^D(N)/W_{D,N}$ that can be realized as Borcherds forms.

Let $D_0$ be the odd part of $DN$. Then according to Corollary
\ref{corollary: level}, the level of the lattice under consideration
is $4D_0$.
 Let us first determine the dimensions of $M^!_n(4D_0)$. 

\begin{Lemma} \label{lemma: dimension}
  Let $D_0$ be the odd part of $DN$ and $g$ be the genus
  of the modular curve $X_0(4D_0)$. Then for a nonnegative integer
  $n$ with
  $$
    n\ge 2g-2-\sum_{d|D_0}\gauss{d/4},
  $$
  we have
  $$
    \dim_\C M^!_n(4D_0)=n+\sum_{d|D_0}\gauss{d/4}+1-g.
  $$
  Moreover, the number of gaps of $M^!(4D_0)$ is $g-\sum_{d|D_0}\gauss{d/4}$.
\end{Lemma}

\begin{proof} Let $\theta(\tau)=\sum_nq^{n^2}$ be the Jacobi theta
  function. For a divisor $d$ of $4D_0$, let $C_d$ represent the cusp
  $1/d$. As a modular form on $\Gamma_0(4D_0)$, we have
  $$
    \div\theta=\sum_{d|D_0}\frac{d}4(C_{2d}).
  $$
  A modular form $f$ is contained in $M^{!}_{n}(4D_0)$ if and only if
  the modular function $g=f/\theta$ on $\Gamma_0(4D_0)$ satisfies
  $$
    \div g\ge -n(\infty)-\sum_{d|D_0}\gauss{d/4}(C_{2d}).
  $$
  Then by the Riemann-Roch theorem, when $n$ is a nonnegative integer
  such that $n\ge 2g-2-\sum_{d|D_0}\gauss{d/4}$, the dimension of the
  space $M^{!}_{n}(4D_0)$ is
  $$
    n+\sum_{d|D_0}\gauss{d/4}+1-g.
  $$
  Now since $D_0$ is squarefree, by Theorem A of \cite{Serre-Stark},
  the space $M_0^!(4D_0)$ is one-dimensional and spanned by
  $\theta$, which implies that there is no modular form in
  $M^!(4D_0)$ having a zero at $\infty$. Therefore, from the dimension
  formula for $M^!_n(4D_0)$, we see that the number of gaps is $n+1-\dim
  M_n^!(4D_0)=g-\sum_{d|D_0}\gauss{d/4}$.
\end{proof}

\begin{Proposition} \label{proposition: admissible span}
  For $(D,N)$ in Theorem \ref{theorem: main} with even $D$, the space
  $M^!(4D_0)$ is spanned by admissible eta-products. Moreover,
  there exists a positive integer $m$ such that, for each positive
  integer $j\ge m$, there exists a modular form $f_j$ in
  $M^!(4D_0)\cap\Z((q))$ whose order of pole at $\infty$ is $j$ and whose
  leading coefficient is $1$.
\end{Proposition}

\begin{proof}
  Let $g$ be the genus of the modular curve $X_0(4D_0)$ and set
  $$
    n_0=\max(2g-2-\sum_{d|D_0}\gauss{d/4},0).
  $$
  According to Lemma \ref{lemma: dimension}, if $n$ is an integer such
  that $n\ge n_0$, then there exists a modular form in $M^!(4D_0)$
  with a pole of order $n$ at $\infty$. Now suppose that we can find
  an eta-product $t(\tau)$ such that $t$ is a modular function on
  $\Gamma_0(4D_0)$ with a unique pole at $\infty$. Let $k$ be the
  order of the pole of $t$ at $\infty$. Now Lemma \ref{lemma:
    dimension} implies that for each integer $j\ge n_0$, there is a
  modular form in $M^!(4D_0)$ with a pole of order $j$ at $\infty$.
  Thus, for all $n\ge n_0$, we have
  $$
    M^!_{n+k}(4D_0)=M^!_n(4D_0)+t M^!_n(4D_0).
  $$
  Therefore, to prove the assertion about $M^!(4D_0)$, it suffices to
  find such a modular function $t$ and show that the space
  $M^!_{n_0+k}(4D_0)$ can be 
  spanned by eta-products and that there exists a positive integer
  $m\ge n_0$ such that for each integer $j$ with $m\le j\le m+k-1$,
  there exists a modular form in $M^!_{j}(4D_0)\cap\Z((q))$ whose
  order of pole at $\infty$ is $j$ and whose leading coefficient is $1$.

  Consider the case of a maximal order first. Assume that $D=2p$ for
  some odd prime $p$. By Lemma \ref{lemma: eta}, for an eta-product
  $\prod_{d|4p}\eta(d\tau)^{r_d}$ to be admissible, the integers $r_d$
  must satisfy
  \begin{equation} \label{equation: integer programming 1}
  \begin{array}{rcrcrcrcrcrl}
 r_1& + &r_2& + &r_4& + &r_p& + &r_{2p}& + &r_{4p}&=1 \\
    &   &r_2&   &  &   &   & + &r_{2p} &   &     &=1+2\delta_2 \\
    &   &   &   &  &  & r_p & + &r_{2p}& + &r_{4p}&=2\delta_p \\
 r_1&+&2r_2&+&4r_4&+&pr_p&+&2pr_{2p}&+&4pr_{4p}&=24\epsilon_1 \\
 4pr_1&+&2pr_2&+&pr_4&+&4r_p&+&2r_{2p}&+&r_{4p}&=24\epsilon_2 \\ 
  \end{array}
  \end{equation}
  for some integers $\delta_2$, $\delta_p$, $\epsilon_1$, and $\epsilon_2$.
  Moreover, the congruence subgroup $\Gamma_0(4p)$ has $6$ cusps,
  represented by $1/c$ with $c|4p$. The orders of the eta function
  $\eta(d\tau)$ at these cusps, multiplied by $24$, are given by
  $$ \extrarowheight3pt
  \begin{array}{c|cccccccccc} \hline\hline
  & 1 & 1/2 & 1/4 & 1/p & 1/2p &  1/4p  \\ \hline
  \eta(\tau) & 4p & p & p & 4& 1& 1\\ 
  \eta(2\tau) & 2p & 2p & 2p & 2& 2& 2 \\
  \eta(4\tau)  & p & p & 4p & 1& 1& 4 \\
  \eta(p\tau) & 4 & 1 & 1 & 4p & p & p \\
  \eta(2p\tau)  & 2 & 2 & 2 & 2p & 2p & 2p \\ 
  \eta(4p\tau)  & 1 & 1 & 4 & p & p & 4p \\
  \hline\hline
  \end{array}
  $$
  Thus, in order for an eta-product to be in $M^!_n(4p)$, the
  exponents $r_d$ should satisfy
  \begin{equation} \label{equation: integer programming 2}
  \begin{array}{rcrcrcrcrcrl}
  r_1&+&2r_2&+&4r_4&+&pr_p&+&2pr_{2p}&+&4pr_{4p}&\ge -24n \\
  r_1&+&2r_2&+&r_4&+&pr_p&+&2pr_{2p}&+&pr_{4p}&\ge 0 \\
  4r_1&+&2r_2&+&r_4&+&4pr_p&+&2pr_{2p}&+&pr_{4p}&\ge 0\\
  pr_1&+&2pr_2&+&4pr_4&+&r_p&+&2r_{2p}&+&4r_{4p}&\ge 0\\
  pr_1&+&2pr_2&+&pr_4&+&r_p&+&2r_{2p}&+&r_{4p}&\ge 0\\
  4pr_1&+&2pr_2&+&pr_4&+&4r_p&+&2r_{2p}&+&r_{4p}&\ge 0\\
  \end{array}
  \end{equation}
  In literature, problems of solving a set of equalities and
  inequalities in integers are called \emph{integer programming
    problems}. Solving \eqref{equation: integer programming 1} and
  \eqref{equation: integer programming 2} using the AMPL modeling
  language (http://www.ampl.com) and the gurobi solver
  (http://www.gurobi.com), we can produce many admissible eta-products.

  To find $t$, we replace the first two equations in \eqref{equation:
    integer programming 1} by $r_1+r_2+r_4+r_p+r_{2p}+r_{4p}=0$ and
  $r_2+r_{2p}=2\delta_2$ and solve the integer programming problem. We
  find that we can choose
  $$
    t(\tau)=\frac{\eta(4\tau)^4\eta(2p\tau)^2}{\eta(2\tau)^2\eta(4p\tau)^4}
  $$
  with $k=(p-1)/2$.

  In the other cases when $N>1$, $N$ is always a prime. The modular
  curve $X_0(4D_0)$ has $12$ cusps and there are more inequalities and
  equalities in the integer programming problem. Nevertheless, we can
  easily find $t$ and many admissible eta-products by solving the
  integer programming problem.

  Having found $t(\tau)$ and many admissible
  eta-products, we check case by case that eta-products do span
  $M^!_{n_0+k}(4p)$ and that there exists positive integer $m\ge n_0$
  such that for each integer $j$ with $m\le j\le m+k-1$,
  there exists a modular form $f_j\in M^!_{n_0+k}(4D_0)\cap\Z((q))$
  whose order of pole at $\infty$ is $j$ and whose leading coefficient
  is $1$. (Sometimes, $f_j$ will be a linear combination of
  eta-products with rational coefficients. To show that all Fourier
  coefficients are integers, we use Sturm's theorem.) Here we
  omit the details, providing only one example as below.
\end{proof}

\begin{Example} \label{example: 26-1}
Consider the case $D=26$ and $N=1$. The modular curve $X_0(52)$ has
genus $5$. Thus, by Lemma \ref{lemma: dimension}, the number of gaps
of $M^!(52)$ is $5-\sum_{d|13}\lfloor d/4\rfloor=2$. The modular function
$$
  t(\tau)=\frac{\eta(4\tau)^4\eta(26\tau)^2}{\eta(2\tau)^2\eta(52\tau)^4}
$$
has a unique pole of order $6$ at $\infty$. According to the proof of
Proposition \ref{proposition: admissible span}, we need to show that
the space $M^!_{11}(52)$ can be spanned by eta-products. Using the gurobi
solver, we find the following solutions
$(r_1,r_2,r_4,r_{13},r_{26},r_{52})$ to the integer
programming problem in \eqref{equation: integer programming 1} and
\eqref{equation: integer programming 2} with $n=11$
\begin{equation*}
\begin{split}
 &(-3 , 6 , 0 , -3 , 11 , -10), \ (-1 , 3 , 1 , 3 , 2 , -7),
  (3 , -3 , 3 , -1 , 8 , -9), \ (1 , 1 , 1 , 1 , 4 , -7), \\
 &(-1 , 1 , 1 , 3 , 4 , -7), \ (0 , -3 , 6 , -2 , 8 , -8),
  (3 , -1 , 1 , -1 , 6 , -7), \ (-5 , 12 , -4 , -1 , 5 , -6), \\
 &(-1 , 2 , 0 , -5 , 15 , -10), \ (1 , -1 , 1 , 1 , 6 , -7),
  (1 , 3 , -1 , 1 , 2 , -5), \ (-1 , 3 , -1 , 3 , 2 , -5), \\
 &(3 , 1 , -1 , -1 , 4 , -5), \ (-2 , 3 , 2 , 0 , 2 , -4),
  (0 , -1 , 2 , -2 , 6 , -4), \ (-2 , 5 , -2 , 0 , 0 , 0).
\end{split}
\end{equation*}
Suitable linear combinations of these eta-products
$\prod_{d|52}\eta(d\tau)^{r_d}$ yield a basis consisting of
\begin{equation*}
\begin{split}
  f_0=1+2q+2q^4+2q^9+\cdots, \qquad
  &f_3=q^{-3}+q^{-1}+q^3+q^9+\cdots \\
  f_4=q^{-4}-q^{-1}-q+q^3+\cdots, \qquad
  &f_5=q^{-5}+q^{-2}-2q+q^2+\cdots, \\
  f_6=q^{-6}+q^{-2}-2q+2q^2+\cdots, \qquad
  &f_7=q^{-7}-q^{-2}+2q-q^2+\cdots, \\
  f_8=q^{-8}+q^{-2}+q^2+2q^5+\cdots, \qquad
  &f_9=q^{-9}+2q^{-1}+3q+2q^3+\cdots, \\
  f_{10}=q^{-10}+3q^{-1}+q-q^3+\cdots, \qquad
  &f_{11}=q^{-11}+2q^2+q^5+4q^7+\cdots,
\end{split}
\end{equation*}
for the space $M^!_{11}(52)$. In fact, since all these modular forms
have integral coefficients, multiplying these $f_j$ by powers of
$t$, we find that for each non-gap positive integer $j$, there exists
a modular form $f_j$ in $M^!(52)\cap\Z((q))$ with a pole of order $j$
at $\infty$ and a leading coefficient $1$.
\end{Example}

\begin{Remark}
Quite curiously, our computation shows that whenever $N=1$, i.e.,
whenever $D_0=p$ is an odd prime, the space $M^!(4D_0)$ has the
property that for each non-gap positive integer $j$, there exists a
modular form $f$ in $M^!(4D_0)\cap\Z((q))$ such that $f$ has a pole of
order $j$ at $\infty$ with leading coefficient $1$.

The smallest $D_0$ such that $M^!(4D_0)$ does not have this property
is $D_0=51$. We can show that the gaps of $M^!(204)$ are
$1,\ldots,14$, and $20$ and there exists a modular form $f$ in
$M^!(204)\cap\Z((q))$ with a Fourier expansion
$2q^{-22}-q^{-20}-2q^{-14}+2q^{-12}+\cdots$. As $20$ is a gap, there
cannot exist $g\in M^!(204)\cap\Z((q))$ with a Fourier expansion $q^{-22}+\cdots$.
\end{Remark}

We now show that for $(D,N)$ in Theorem \ref{theorem: main} with even
$D$, all meromorphic modular forms of even
weights on $X_0^D(N)/W_{D,N}$ with divisors supported on CM-divisors,
which we define below, can be realized as Borcherds forms.

\begin{Definition} \label{definition: CM-divisor}
For a negative discriminant $d$, we
let $\CM(d)$ denote the set of CM-points of discriminant $d$ on
$X_0^D(N)/W_{D,N}$, $h_d=|\CM(d)|$, and $P_d$ be the divisor
$$
  P_d=\sum_{\tau\in\CM(d)}\tau.
$$
(If $h_d=0$, then $P_d$ simply means $0$.) We call $P_d$ the
\emph{$\CM$-divisor} of discriminant $d$. Note that sometimes we
wish to keep track the degree of the divisor $P_d$. In such as a case,
we will write $P_d^{\times h_d}$ instead of $P_d$.
\end{Definition}


\begin{Lemma} \label{lemma: divisor at infinity}
  Let $f$ be an element in $M^!(4D_0)\cap\Z((q))$, $F_f$ be the
  vector-valued modular form constructed using $f$ as given by
  \eqref{equation: Ff}, and $\psi_{F_f}(\tau)$ be the Borcherds form
  on $X_0^D(N)/W_{D,N}$ corresponding to $F_f$ as defined in
  Definition \ref{definition: Borcherds on Shimura}. Suppose that the
  Fourier expansion of $f$ is $\sum_m c_mq^m$. Then
  $$
    \div\psi_{F_f}=\sum_{m<0}c_m
    \sum_{r\in\Z^+,4m/r^2\text{ is a discriminant}}\frac1{e_{4m/r^2}}
    P_{4m/r^2},
  $$
  where $e_d$ is the cardinality of the
  stabilizer subgroup of $\tau\in\CM(d)$ in $N_B^+(\O)/\Q^\ast$.
\end{Lemma}

\begin{proof} This follows from Proposition 5.4 of \cite{Errthum}
  and Lemma 7 of \cite{Yang-CM}.
\end{proof}

\begin{Proposition} \label{proposition: all Borcherds}
  For $(D,N)$ in Theorem \ref{theorem: main} with $2|D$,
  all meromorphic modular forms of even weights on $X_0^D(N)/W_{D,N}$
  with a divisor supported on CM-divisors can be realized as Borcherds
  forms.
\end{Proposition}

\begin{proof} We will prove only the case $(D,N)=(26,1)$. The proof of
  the other cases is similar.

  We claim that
  \begin{enumerate}
  \item[(i)] there is a Borcherds form $\psi$ of weight $2$ with a trivial
    character, and
  \item[(ii)] every modular function on $X_0^{26}(1)/W_{26,1}$ with
    divisor supported on CM-divisors can be realized as a Borcherds
    form.
  \end{enumerate}
  Then observe that if $\phi$ is a modular form of even weight $k$,
  then $\phi/\psi^{k/2}$ has weight $0$. The two claims imply that
  $\phi$ can be realized as a Borcherds form.

  The Shimura curve $X_0^{26}(1)/W_{26,1}$ has genus $0$ and precisely
  five elliptic points of order $2$. Among the five elliptic points,
  one is a CM-point of discriminant $-8$, one is a CM-point of
  discriminant $-52$, and the remaining three are CM-points of
  discriminant $-104$. Also, if $\psi$ is a meromorphic modular form
  of even weight $k$ on $X_0^{26}(1)/W_{26,1}$, then the degree of
  $\div\psi$ is $k/4$. Thus, by Lemmas \ref{lemma: genus 0} and \ref{lemma:
    divisor at infinity}, for $f=\sum_mc_mq^m\in M^!(52)\cap\Z((q))$,
  the Borcherds form $\psi_{F_f}$ has even weight $k$ and a trivial
  character if and only if
  \begin{equation} \label{equation: 26 condition 1}
    \sum_{m<0}c_m\sum_{r\in\Z^+,4m/r^2\text{ is a discriminant}}
    \frac1{e_{4m/r^2}}|\CM(4m/r^2)|=k/4
  \end{equation}
  and
  \begin{equation} \label{equation: 26 condition 2}
    \sum_{m=-2n^2}c_m\equiv\sum_{m=-13n^2}c_m\equiv
    \sum_{m=-26n^2}c_m\equiv k/2 \mod 2.
  \end{equation}

  Now from Example \ref{example: 26-1}, we know that for each $j\ge
  3$, there exists a unique element $f_j$ in $M^!(52)\cap\Z((q))$ such
  that its Fourier expansion is of the form
  $f_j=q^{-j}+c_{-2}q^{-2}+c_{-1}q^{-1}+\cdots$. In particular, we find
  \begin{equation*}
  \begin{split}
    f_7&=q^{-7}-q^{-2}+2q+\cdots,\\
    f_{13}&=q^{-13}-q^{-2}-2q^{-1}+q+\cdots, \\
    f_{26}&=q^{-26}+q^{-1}-q+\cdots.
  \end{split}
  \end{equation*}
  The modular form
  $$
    f=f_{26}-f_{13}+2f_7=q^{-26}-q^{-13}+2q^{-7}-q^{-2}+3q^{-1}+2q+\cdots,
  $$
  satisfies the conditions in \eqref{equation: 26 condition 1} and
  \eqref{equation: 26 condition 2} with $k=2$. (Note that there do not
  exist CM-points of discriminants $-4$ and $-7$ on the Shimura curve
  $X_0^{26}(1)$, so the presence of the terms $q^{-7}$ and $q^{-1}$
  will not contribute anything to the divisor of the Borcherds form.)
  This proves Claim (i).

  To prove Claim (ii), it suffices to show that for each discriminant
  $d<0$, there exists a modular form $f$ in $M^!(52)\cap\Z((q))$
  satisfying \eqref{equation: 26 condition 1} and \eqref{equation: 26
    condition 2} with $k=0$ such that $\div\psi_{F_f}=P_d^{\times
    h_d}-h_dP_{-8}$. For the special cases $d=-52$ and $d=-104$, we
  may choose $f$ to be $2f_{13}$ and $2f_{26}+6f_7$, respectively. If
  $d\neq-52,-104$ and $d$ is a fundamental discriminant, we choose
  $f$ to be $f_{|d|}+af_7$ with a proper integer $a$ such that the
  coefficient of $q^{-2}$ is $-2h_d$. (If $4|d$, we may choose
  $f_{|d|/4}+bf_7$ instead.) Now assume that $d$ is not a fundamental
  discriminant, say, $d=d_0n^2$ for some fundamental discriminant
  $d_0$. We let $a$ be the integer such that the coefficient of
  $q^{-2}$ in $f=\sum_{r|n}\mu(r)f_{|d|/r^2}+af_7$ is $-2h_d$, where
  $\mu(r)$ is the M\"obius function. Then $\div\psi_{F_f}=P_d^{\times
    h_d}-h_dP_{-8}$. (The case $d_0=-8$ needs a special treatment, but
  it is completely analogous.) This proves Claim (ii) and hence the
  proposition for the case $(D,N)=(26,1)$. 
\end{proof}

\end{subsection}

\begin{subsection}{Case of odd $D$}
\label{subsection: D odd}
The construction of Borcherds forms in the case of odd $D$ is a
little more complicated than the case of even $D$. The idea of
using $\{\infty\}$-weakly holomorphic modular forms to construct
Borcherds forms is no longer sufficient for our purpose. The reason is
that if the divisor of a Borcherds form arising from a
$\{\infty\}$-weakly holomorphic modular form is supported at a
CM-point of discriminant $d$, $d\equiv 1\mod 4$, then it also is
supported at CM-points of discriminant $4d$. However, in practice, we
are often required to construct Borcherds forms whose divisors are
supported at CM-points of discriminant $d$, but not at CM-points of
discriminant $4d$. Thus, in the case of odd $D$, we will need to
use $\{\infty,0\}$-weakly holomorphic modular forms to construct
desired Borcherds forms.

Assume that $D$ is odd and $N$ is squarefree. As usual, we let $\O$ be
an Eichler order of level $N$ in the quaternion algebra $B$ of
discriminant $D$, and $L$ be the lattice formed by elements of trace
$0$ in $\O$. For convenience, for a modular form $f$, we let $P(f)$
denote the principal part of $f$ at $\infty$, i.e., the sums of the terms with
negative exponents in the Fourier expansion of $f$. Similarly, for a
vector-valued modular form $F=\sum_{\eta\in L^\vee/L}F_\eta e_\eta$, we let
$$
  P(F)=\sum_\eta P(F_\eta)e_\eta.
$$

\begin{Lemma}Let $M$ be the level of $L$. Suppose $f$ is a 
  $\{\infty,0\}$-weakly holomorphic scalar-valued modular of weight
  $1/2$ on $\Gamma_0(M)$ and $F_f$ was given in Lemma \ref{lemma:
    Barnard}. Assume that $P(f|_{1/2}S)=\sum_{n>0}b_nq^{-n/M}$. Then
  $$
    P(F_f)=P(f)e_0+\frac{Me^{2\pi i/8}}{\sqrt{|L^\vee/L|}}
    \sum_{n>0}b_nq^{-n/M}\sum_{\eta\in L^\vee/L,\nm(\eta)\in n/M+\Z}
    e_\eta.
  $$
\end{Lemma}

\begin{proof} Since $f$ is of $\{\infty,0\}$-weakly holomorphic, if
  $\gamma$ is an element of $\SL(2,\Z)$ such that $\gamma\infty$ is
  not equivalent to the cusp $\infty$ or $0$, then we have
  $P(f|_{1/2}\gamma)=0$. Now $\gamma=I$ is the only right coset
  representative of $\Gamma_0(M)$ in $\SL(2,\Z)$ with
  $\gamma\infty\sim\infty$ and $\gamma=ST^j$, $j=0,\ldots,M-1$, are the
  only right coset representatives with $\gamma\infty\sim 0$. Thus,
  $$
    P(F_f)=P(f)e_0
   +\sum_{j=0}^{M-1}P(f\big|_{1/2}ST^j)\rho_L(T^{-j}S^{-1})e_0.
  $$
  Since
  $$
    \rho_L(S^{-1})e_\delta=\frac{e^{2\pi i/8}}{\sqrt{|L^\vee/L|}}
    \sum_{\eta\in L^\vee/L}e^{-\gen{\eta,\delta}}e_\eta,
  $$
  we find
  \begin{equation*}
  \begin{split}
    P(F_f)&=P(f)e_0+\frac{e^{2\pi i/8}}{\sqrt{|L^\vee/L|}}
    \sum_{j=0}^{M-1}P(f\big|_{1/2}ST^j)\rho_L(T^{-j})
    \sum_{\eta\in L^\vee/L}e_\eta \\
  &=P(f)e_0+\frac{e^{2\pi i/8}}{\sqrt{|L^\vee/L|}}
    \sum_{n>0}b_nq^{-n/M}\sum_{j=0}^{M-1}\sum_{\eta\in L^\vee/L}
    e^{2\pi ij(-n/M+\nm(\eta))}e_\eta \\
  &=P(f)e_0+\frac{Me^{2\pi i/8}}{\sqrt{|L^\vee/L|}}
    \sum_{n>0}b_nq^{-n/M}\sum_{\eta\in L^\vee/L,\nm(\eta)\in n/M+\Z}e_\eta.
  \end{split}
  \end{equation*}
  This proves the lemma.
\end{proof}

In general, the principal part
$e^{2\pi i/8}|L^\vee/L|^{-1/2}P(f|_{1/2}S)$ in the 
lemma lie in $\C[q^{-1/M}]$. For our purpose, we will only
consider those $f$ such that
$$
  P(f)\in\Z[q^{-1}], \qquad
  \frac{Me^{2\pi i/8}}{\sqrt{|L^\vee/L|}}P(f|_{1/2}S)
  \in\Z[q^{-1/4}].
$$

\begin{Lemma}\label{lemma: divisor infinity 0} Let $f$ be as in the
  lemma above. Suppose that $P(f)$ and $P(f|_{1/2}S)$ are of the form
  $$
    P(f)=\sum_{n>0,n\in\Z}a_nq^{-n}, \qquad
    \frac{Me^{2\pi i/8}}{\sqrt{|L^\vee/L|}}P(f|_{1/2}S)
   =\sum_{n>0,n\in\Z}b_nq^{-n/4}
  $$
  for some integers $a_n$ and $b_n$. Then
  \begin{equation*}
  \begin{split}
    \div\psi_{F_f}&=\sum_na_n\sum_{r\in\Z^+,-4n/r^2\text{ is a discriminant}}
    \frac1{e_{-4n/r^2}}P_{-4n/r^2} \\
  &\qquad+\sum_nb_n\sum_{r\in\Z^+,-N^2n/r^2\text{ is a discriminant}}
    \frac1{e_{-N^2n/r^2}}P_{-N^2n/r^2},
  \end{split}
  \end{equation*}
  where $e_d$ is the cardinality of the
  stabilizer subgroup of a CM-point of discriminant $d$ in $N_B^+(\O)/\Q^*$.
\end{Lemma}

\begin{proof} Let $q$ be a prime satisfying the condition in Lemma
  \ref{lemma: basis for O} so that $B=\JS{DN,q}\Q$ is a quaternion
  algebra of discriminant $D$. Let $\O$ be the Eichler order of level
  $N$ spanned by $e_1,\ldots,e_4$ given in \eqref{equation: ei} and
  $\{\ell_1,\ell_2,\ell_3\}$ be given as in \eqref{equation:
    ell i}. The contribution from $P(f)e_0$ to the divisor of
  $\psi_{F_f}$ is described in Lemma \ref{lemma: divisor at infinity}.
  Here we are mainly concerned
  with the contribution from $P(f|S)$.

  Consider the case of odd $N$ first. Let
  $\lambda$ be an element in
  $L^\vee=\Z\ell_1/2+\Z\ell_2/DN+\Z\ell_3/DN$ satisfying
  $\nm(\lambda)=n/4$ for some positive integer $n$. We need to
  determine the discriminant of the optimal embedding
  $\phi:\Q(\sqrt{-n})\hookrightarrow B$ that maps $\sqrt{-n}$ to
  $2\lambda$.

  Observe that $2DN\lambda\in\O$ and $\nm(2DN\lambda)=-D^2N^2n$.
  By {\cite[Proposition 1.53]{Bayer}}, we must have $2N\lambda\in\O$, i.e.,
  $\lambda=c_1\ell_1/2+c_2\ell_2/N+c_3\ell_3/N$ for som integers
  $c_1$, $c_2$, and $c_3$, and the discriminant of the optimal
  embedding $\phi$ is $-4N^2n/r^2$ for some integer $r$.

  From the Gram matrix in \eqref{equation: Gram L}, we have
  $$
    \nm(N\lambda)
  =-\frac{qN^2c_1^2}4+\frac{q-1}4DNc_2^2+\frac{1-a^2DN}qDNc_3^2
   -aDN^2c_1c_3+DNc_2c_3.
  $$
  As $q$ is congruent to $1$ modulo $4$, this shows that
  $\nm(2N\lambda)\equiv 0,3\mod 4$. Therefore, if $n\equiv 1,2\mod 4$,
  then there does not exist $\lambda\in L^\vee$ such that
  $\nm(\lambda)=n/4$. Also, if $n\equiv 3\mod 4$, then $c_1$ must be
  odd and
  $$
    \frac{1+N\lambda}2=\frac{1-Nc_1}2e_1+Nc_1e_2+c_2e_3+c_3e_4\in\O.
  $$
  In this case, the discriminant of the optimal embedding is
  $-N^2n/r^2$ for some $r$. If $n\equiv 0\mod 4$, then $c_1$
  is even. It follows that $N\lambda\in\O$ and the optimal embedding
  has discriminant $-N^2n/r^2$ for some $r$.

  Conversely, given a CM-point $\tau$ of discriminant $-N^2n/r^2$,
  there exists an element $\lambda=d_1\ell_1+d_2\ell_2+d_3\ell_3\in L$
  fixing $\tau$ and having norm
  $$
    \nm(\lambda)=\begin{cases}
    -N^2n/4, &\text{if }n\equiv0\mod 4, \\
    -N^2n, &\text{if }n\equiv 3\mod 4. \end{cases}
  $$
  Note that if $n$ is odd, then we must have $(1+\lambda)/2\in\O$.
  In other words, $d_2$ and $d_3$ are even and $d_1$ is odd.
  On the other hand,
  $$
    \nm(\lambda)
  =-qd_1^2+\frac{q-1}4DNd_2^2+\frac{1-a^2DN}qDNd_3^2
   -2aDNd_1d_3+DNd_2d_3.
  $$
  Since $N$ is squarefree, this implies that $N|d_1$. Setting
  $$
    \lambda'=\begin{cases}\lambda/N, &\text{if }n\equiv 0\mod 4, \\
    \lambda/(2N), &\text{if }n\equiv 3\mod 4, \end{cases}
  $$
  we find $\lambda'\in L^\vee$ with $\nm(\lambda')=n/4$. This proves
  the lemma for the case of odd $N$. The proof of the case of even $N$
  is similar and is omitted.
\end{proof}

\begin{Lemma}\label{lemma: eta at 0}
  Let $M$ be the level of the lattice $L$ and let
  $f(\tau)=\prod_{d|M}\eta(d\tau)^{r_d}$ be an admissible eta-product.
  (See Definition \ref{definition: admissible eta}.) Then we have
  $$
    \frac{e^{2\pi i/8}}{\sqrt{|L^\vee/L|}}(f\big|_{1/2}S)(\tau)
   =\frac1{\sqrt{|L^\vee/L|}}\prod_{d|M}
    \frac1{d^{r_d/2}}\eta(\tau/d)^{r_d}\in\Q((q^{1/M})).
  $$
\end{Lemma}

\begin{proof} The lemma follows immediately from the formula
  $\eta(-1/\tau)=e^{-2\pi i/8}\sqrt\tau\eta(\tau)$ and the assumptions
  that $\sum r_d=1$ and that $|L^\vee/L|\prod_{d|M}d^{r_d}$ is a
  square in $\Q^\ast$.
\end{proof}

\begin{Lemma} \label{lemma: dimension 2}
  Let $D_0$ be the odd part of $DN$ and $g$ be the genus of the
  modular curve $X_0(4D_0)$.
  \begin{enumerate}
  \item For nonnegative integers $m$ and $n$ with
  $$
    m+n\ge 2g-2-\sum_{d|D_0}\gauss{d/4},
  $$
  we have
  $$
    \dim_\C M_{m,n}^{!,!}(4D_0)=m+n+\sum_{d|D_0}\gauss{d/4}+1-g.
  $$
  \item Let $m$ be a nonnegative integers such that
  $m\ge 2g-2-\sum_{d|D_0}\gauss{d/4}$. Then for each positive
  integer $n$, there exists a modular form $f_n$ in $M^{!,!}_{m,n}(4D_0)$
  with a pole of order $n$ at $0$. Furthermore, the space
  $M^{!,!}(4D_0)$ is spanned by $M^!(4D_0)$ and $f_1,f_2,\ldots$.
  \end{enumerate}
\end{Lemma}

\begin{proof}
The proof of Part (1) is similar to that of Lemma \ref{lemma:
  dimension} and is omitted. To prove Part (2), we notice that Part
(1) implies that when $m\ge 2g-2-\sum_{d|D_0}\gauss{d/4}$, the space
$M^{!,!}_{m,0}(4D_0)$ has co-dimension $n$ in $M^{!,!}_{m,n}(4D_0)$.
It follows that for each integer $k$ with $1\le k\le n$, there exists
a modular form $f_k$ in $M^{!,!}_{m,n}(4D_0)$ with a pole of order $k$ at
$0$. Now if $f$ is a modular form in $M^{!,!}(4D_0)$, then for some
linear combination $\sum c_nf_n$, we have $f-\sum c_nf_n\in M^!(4D_0)$. This proves Part (2).
\end{proof}

\begin{Proposition}
  For $(D,N)$ in Theorem \ref{theorem: main} with odd $D$ and
  squarefree $N$, the space $M^{!,!}(4D_0)$ is spanned by admissible
  eta-products. Moreover, if $f(\tau)\in M^{!,!}(4D_0)\cap\Q((q))$, then
  $$
    \frac{e^{2\pi i/8}}{\sqrt{|L^\vee/L|}}
    (f|_{1/2}S)(\tau)\in\Q((q^{1/(4D_0)})).
  $$
\end{Proposition}

\begin{proof} Suppose that we can find an eta-product
  $t(\tau)$ such that $t(\tau)$ is a modular function on $X_0(4D_0)$
  with a unique pole at $\infty$. Let $k$ be the order of pole of
  $t(\tau)$ at $\infty$. Then $t(-1/(4D_0\tau))$ is a modular function
  on $X_0(4D_0)$ with a unique pole of order $k$ at $0$. Let $g$ be
  the genus of $X_0(4D_0)$ and $m$ be an integer with $m\ge
  2g-2-\sum_{d|D_0}\gauss{d/4}$. By Lemma \ref{lemma: dimension 2},
  for each positive integer $j$, there exist a modular form in
  $M^{!,!}_{m,j}(4D_0)$ with a pole of order $j$ at $0$. It follows that
  $$
    M_{m,n+k}^{!,!}(4D_0)=M_{m,n}^{!,!}(4D_0)+t(-1/4D_0\tau)
    M_{m,n}^{!,!}(4D_0)
  $$
  Thus, to prove the proposition, it suffices to show that
  \begin{enumerate}
  \item[(i)] there exists an eta-product $t(\tau)$ such that $t(\tau)$ is a modular
    function on $\Gamma_0(4D_0)$ with a unique pole of at $\infty$,
  \item[(ii)] admissible eta-products span $M^!(4D_0)$, and
  \item[(iii)] admissible eta-products span $M^{!,!}_{m,k}(4D_0)$, where $k$
    is the order of pole of $t(\tau)$ at $\infty$.
  \end{enumerate}
  For (i) and (ii), the integer programming problem involved in the
  construction of $t(\tau)$ and admissible eta-products is the same as
  that in Proposition \ref{proposition: admissible span}. For (iii),
  the integer programming problem is slightly different. For the case
  $D_0=p$ is a prime, instead of \eqref{equation: integer programming
  2}, we have
  \begin{equation*}
  \begin{array}{rcrcrcrcrcrl}
  r_1&+&2r_2&+&4r_4&+&pr_p&+&2pr_{2p}&+&4pr_{4p}&\ge -24m \\
  r_1&+&2r_2&+&r_4&+&pr_p&+&2pr_{2p}&+&pr_{4p}&\ge 0 \\
  4r_1&+&2r_2&+&r_4&+&4pr_p&+&2pr_{2p}&+&pr_{4p}&\ge 0\\
  pr_1&+&2pr_2&+&4pr_4&+&r_p&+&2r_{2p}&+&4r_{4p}&\ge 0\\
  pr_1&+&2pr_2&+&pr_4&+&r_p&+&2r_{2p}&+&r_{4p}&\ge 0\\
  4pr_1&+&2pr_2&+&pr_4&+&4r_p&+&2r_{2p}&+&r_{4p}&\ge -24k\\
  \end{array}
  \end{equation*}
  where the last inequality corresponds to the condition that the
  order of pole at $0$ is at most $k$. After setting up the integer
  programming problems, we check case by case that admissible
  eta-products do expand $M^{!,!}(4D_0)$.

  Since every modular form $f(\tau)$ in $M^{!,!}(4D_0)\cap\Q((q))$ is
  a $\Q$-linear combination of admissible eta-products, the assertion
  about rationality of Fourier coefficients of $f|_{1/2}S$ follows
  from Lemma \ref{lemma: eta at 0}.
\end{proof}

\begin{Example} \label{example: 15}
Consider the Shimura curve $X^{15}_0(1)/W_{15,1}$. We
  have $|L^\vee/L|=450$ and the level of the lattice $L$ is
  $60$. By solving the relevant integer programming problem, we find that
  \begin{equation*}
    t(\tau)=\frac{\eta(2\tau)\eta(12\tau)^6\eta(20\tau)^2\eta(30\tau)^3}
    {\eta(4\tau)^2\eta(6\tau)^3\eta(10\tau)\eta(60\tau)^6}
   =q^{-8}-q^{-6}+q^{-4}+q^{-2}+q^{4}+\cdots.
  \end{equation*}
  is a modular function on $\Gamma_0(60)$ with a unique pole of order
  $8$ at $\infty$. Also, the genus of $X_0(60)$ is $7$. By Lemma
  \ref{lemma: dimension}, the number of gaps of $M^!(60)$ is $3$, and
  for $n\ge 8$, we have $\dim M_n^!(60)=n-2$. According to the proof
  of Proposition \ref{proposition: admissible span}, we should find an
  integer $n_0$ such that $M^!_{n_0+8}(60)$ is spanned by eta-products
  and for each integer $j$ with $n_0<j\le n_0+8$, there exists a
  modular form in $M^!_{n_0+8}(60)$ with a pole of order $j$ at
  $\infty$. It turns out that we can choose $n_0=3$. (In other words,
  we will see that the gaps are $1,2,3$.)

  For convenience, we let
  $(r_1,r_2,r_3,r_4,r_5,r_6,r_{10},r_{12},r_{15},r_{20},r_{30},r_{60})$
  represents the eta-product $\prod_{d|60}\eta(d\tau)^{r_d}$. By
  solving the integer programming program, we find that there are at
  least $96$ eta-products in $M_{11}^!(60)$. Among them, we choose
  \begin{equation*} \small
  \begin{split}
  f_{11}=(0, 1, 0, -1, -1, -1, 0, 2, 5, 2, 1, -7), 
  &\qquad f_{10}=(0, 0, -1, 0, 2, 1, 0, 1, 1, 1, 2, -6), \\
  f_9=(0, 0, -1, 0, -1, 2, -1, 0, 2, 3, 4, -7),
  &\qquad f_8=(0, 0, -1, 1, -2, 0, 1, 1, 5, 1, 0, -5), \\
  f_7=(0, 1, 1, -1, 2, -1, -2, 1, -1, 3, 3, -5),
  &\qquad f_6=(0, 1, 0, -1, 0, -1, 0, 2, 2, 1, 1, -4), \\
  f_5=(0, 0, -1, 0, -1, 1, 2, 1, 2, 0, 0, -3),
  &\qquad f_4=(0, 0, -1, 0, 0, 2, -1, 0, -1, 2, 4, -4), \\
  f_0=(-2, 5, 0, -2, 0, 0, 0, 0, 0, 0, 0, 0). &
  \end{split}
  \end{equation*}
  They form a basis for $M^!_{11}(60)$. (The subscripts are the orders of
  poles at $\infty$.) Then multiplying those modular forms by suitable
  powers of $t(\tau)$, we get, for each a nongap integer $j>0$, a
  modular form in $M^!(60)\cap\Z((q))$ with a unique pole of order $j$
  at $\infty$ and a leading coefficient $1$.

  Furthermore, we find that there are at least $102$ eta-products in
  $M^{!,!}_{3,8}(60)$. Among them, we choose
  \begin{equation*} \small
  \begin{split}
    g_1(\tau)&=
    \frac{\eta(2\tau)\eta(3\tau)\eta(4\tau)\eta(5\tau)\eta(12\tau)\eta(30\tau)}
    {\eta(\tau)^2\eta(6\tau)\eta(60\tau)^2}
     =q^{-3}+2q^{-2}+4q^{-1}+\cdots, \\
    g_2(\tau)&=
    \frac{\eta(2\tau)^4\eta(3\tau)^2\eta(10\tau)^3\eta(12\tau)}
    {\eta(\tau)^3\eta(4\tau)\eta(5\tau)\eta(6\tau)^2\eta(20\tau)\eta(60\tau)}
     =q^{-2}+3q^{-1}+5+8q+\cdots, \\
    g_3(\tau)&=
    \frac{\eta(4\tau)^2\eta(6\tau)\eta(10\tau)^2}
    {\eta(\tau)^2\eta(20\tau)\eta(60\tau)}
   =q^{-2}+2q^{-1}+5+10q+18q^2+\cdots, \\
    g_4(\tau)&=
    \frac{\eta(2\tau)^3\eta(3\tau)^4\eta(5\tau)\eta(12\tau)^2\eta(30\tau)}
    {\eta(\tau)^4\eta(4\tau)\eta(6\tau)^3\eta(15\tau)\eta(60\tau)}
   =q^{-1}+4+11q+24q^2+\cdots, \\
    g_5(\tau)&=
    \frac{\eta(2\tau)^5\eta(3\tau)\eta(6\tau)\eta(10\tau)}
    {\eta(\tau)^5\eta(12\tau)\eta(60\tau)}
   =q^{-2}+5q^{-1}+15+39q+90q^2+\cdots, \\
    g_6(\tau)&=
    \frac{\eta(2\tau)^3\eta(3\tau)^2\eta(5\tau)\eta(6\tau)^2}
    {\eta(\tau)^5\eta(12\tau)\eta(60\tau)}
   =q^{-2}+5q^{-1}+17+48q+\cdots, \\
    g_7(\tau)&=
    \frac{\eta(2\tau)^2\eta(3\tau)\eta(4\tau)\eta(5\tau)^3\eta(6\tau)}
    {\eta(\tau)^5\eta(12\tau)\eta(15\tau)}
   =1+5q+18q^2+54q^3+\cdots, \\
    g_8(\tau)&=
    \frac{\eta(2\tau)^4\eta(3\tau)^2\eta(5\tau)^3\eta(12\tau)^2\eta(15\tau)}
    {\eta(\tau)^6\eta(4\tau)\eta(6\tau)^2\eta(10\tau)\eta(60\tau)}
   =q^{-1}+6+23q+72q^2+\cdots
  \end{split}
  \end{equation*}
  of weight $1/2$ on $\Gamma_0(60)$. By
  Lemma \ref{lemma: eta at 0},
  \begin{equation*} \small
  \begin{split}
  \frac{60e^{2\pi i/8}}{15\sqrt2}(g_1\big|S)(\tau)
  &=\frac2{3}(q^{-2/60}+2q^{-1/60}+4+8q^{1/60}+14q^{2/60}+\cdots), \\
  \frac{60e^{2\pi i/8}}{15\sqrt2}(g_2\big|S)(\tau)
  &=q^{-2/60}+q^{-1/60}+2+4q^{1/60}+6q^{2/60}+8q^{3/60}+\cdots, \\
  \frac{60e^{2\pi i/8}}{15\sqrt2}(g_3\big|S)(\tau)
  &=2(q^{-3/60}+q^{-2/60}+2q^{-1/60}+4+6q^{1/60}+\cdots), \\
  \frac{60e^{2\pi i/8}}{15\sqrt2}(g_4\big|S)(\tau)
  &=2(q^{-4/60}+q^{-3/60}+q^{-2/60}+2q^{-1/60}+3+\cdots), \\
  \frac{60e^{2\pi i/8}}{15\sqrt2}(g_5\big|S)(\tau)
  &=q^{-5/60}+q^{-4/60}+2q^{-3/60}+3q^{-2/60}+5q^{-1/60}+\cdots, \\
  \frac{60e^{2\pi i/8}}{15\sqrt2}(g_6\big|S)(\tau)
  &=\frac23(q^{-6/60}+q^{-5/60}+2q^{-4/60}+3q^{-3/60}+\cdots), \\
  \end{split}
  \end{equation*}

 \begin{equation*} \small
  \begin{split}
  \frac{60e^{2\pi i/8}}{15\sqrt2}(g_7\big|S)(\tau)
  &=\frac1{5}(q^{-7/60}+q^{-3/60}+q^{-2/60}+2q^{1/60}+\cdots), \\
  \frac{60e^{2\pi i/8}}{15\sqrt2}(g_8\big|S)(\tau)
  &=\frac2{15}(q^{-8/60}+q^{-7/60}+2q^{-6/60}+3q^{-5/60}+\cdots).
  \end{split}
  \end{equation*}

  Thus, letting
  $$
    h_1=3g_1/2-g_2, \quad h_2=g_2, \quad h_3=g_3/2, \quad h_4=g_4/2,
  $$
  and
  $$
    h_5=g_5, \quad h_6=3g_6/2, \quad h_7=5g_7, \quad h_8=15g_8/2,
  $$
  we get a sequence $h_j$, $j=1,\ldots,8$, of modular forms such that
  $$
    \frac{60e^{2\pi i/8}}{15\sqrt2}(h_j\big|S)(\tau)
   =q^{-j/60}+\cdots.
  $$
  Now we have
  $$
    t(-1/60\tau)=5\frac{\eta(2\tau)^3\eta(3\tau)^2\eta(5\tau)^6\eta(30\tau)}
    {\eta(\tau)^6\eta(6\tau)\eta(10\tau)^3\eta(15\tau)^2}
   =5+30q+120q^2+390q^3+\cdots,
  $$
  which is a modular function on $\Gamma_0(60)$ having a unique pole
  of order $8$ at the cusp $0$.
  Thus, by multiplying $h_j$ with suitable powers of $t(-1/60\tau)$, we
  get, for each positive integer $m$, an $\{\infty,0\}$-weakly
  holomorphic modular form $h_m$ whose order of pole at $\infty$ is
  bounded by $3$, while
  $$
    \frac{60e^{2\pi i/8}}{15\sqrt2}(h_m\big|S)(\tau)=q^{-m/60}+\cdots.
  $$
\end{Example}

\begin{Remark}
  We expect that, as in the case of even $D$, for $(D,N)$ in Theorem \ref{theorem: main} with
  odd $D$ and squarefree $N$, all meromorphic modular forms of even
  weights on $X_0^D(N)/W_{D,N}$ with a divisor supported on
  CM-divisors can be realized as a Borcherds form. However, a proof
  along the line of that of Proposition \ref{proposition: all
    Borcherds} will be a little complicated because the Fourier
  expansions at $0$ of a modular form in $M^{!,!}(4D_0)\cap\Z((q))$
  may not be integral.
\end{Remark}

\begin{Example} \label{example: 15 2}
  Here we give an example showing how to construct a
  Borcherds form with a desired divisor on $X_0^{15}(1)/W_{15,1}$
  using modular forms in $M^{!,!}(60)$.

  Suppose that we wish to construct a Borcherds form with a divisor
  $P_{-12}-P_{-3}$. For a positive integer $j$, we let $h_j$ be the
  modular form in $M^{!,!}(60)$ constructed in Example \ref{example:
    15} with the properties that its order of pole at
  $\infty$ is bounded by $3$ and
  $$
    \frac{60e^{2\pi i/8}}{15\sqrt2}(h_j\big|S)(\tau)=q^{-j/60}+\cdots.
  $$
  A suitable linear combination of these $h_m$ will yield a function
  $h$ with
  \begin{equation} \label{equation: 15 h}
    h(\tau)=q^{-2}+11+\cdots, \quad
    \frac{60e^{2\pi i/8}}{15\sqrt2}(h\big|S)(\tau)
   =2q^{-3/4}+4q^{1/60}+4q^{2/60}+\cdots.
  \end{equation}
  By Lemma \ref{lemma: divisor infinity 0},
  $$
    \div\psi_{F_h}=\frac13P_{-3}.
  $$
  Let
  \begin{equation*}
  \begin{split}
    f=f_8-f_5+f_4=q^{-8}+2q^{-3}+q^{-2}+2q^2+\cdots
  \end{split}
  \end{equation*}
  where $f_j$ are as given in Example \ref{example: 15}. By Lemma
  \ref{lemma: divisor infinity 0},
  $$
    \div\psi_{F_f}=P_{-12}+\frac13P_{-3}.
  $$
  Therefore, we find that $\psi_{F_{f-4h}}$  is a Borcherds form with
  a divisor $P_{-12}-P_{-3}$.
\end{Example}

\end{subsection}
\end{section}

\begin{section}{Equations of hyperelliptic Shimura curves}
\label{section: equations}
Recall that a compact Riemann surface $X$ of genus $\ge 2$ is
hyperelliptic if and only if there exists a double covering $\pi:
X\rightarrow \mathbb{P}(\mathbb{C}),$ or equivalently, if there exists
an involution $w:X\rightarrow X$ such that $X/w$ has genus zero.
The involution $w$ is unique and is called the hyperelliptic
involution.

\begin{theorem}[{\cite[Theorems 7 and 8]{Ogg}}]
Let $g(D,N)$ denote the genus of $X^D_0(N)$. The following table gives the
full list of hyperelliptic Shimura curves, $D>1$, and their hyperelliptic
involutions.

\renewcommand{\arraystretch}{1}
\tabcolsep=10pt                         
\begin{table}[htb]
\caption{List of hyperelliptic Shimura curves and their hyperelliptic involutions} \label{table: curves}
\newcommand{\tabincell}[2]{\begin{tabular}{@{}#1@{}}#2\end{tabular}}
  \centering
\begin{minipage}[b]{0.45\linewidth}\centering
\begin{tabular}{|c|c|c|c|}
\hline
$D$&$N$&$g(D,N)$&$w$\\
\hline
$26$&$1$&$2$&$w_{26}$\\
\hline
$35$&$1$&$3$&$w_{35}$\\
\hline
$38$&$1$&$2$&$w_{38}$\\
\hline
$39$&$1$&$3$&$w_{39}$\\
\hline
$51$&$1$&$3$&$w_{51}$\\
\hline
$55$&$1$&$3$&$w_{55}$\\
\hline
$57$&$1$&$3$&$w_{19}$\\
\hline
$58$&$1$&$2$&$w_{29}$\\
\hline
$62$&$1$&$3$&$w_{62}$\\
\hline
$69$&$1$&$3$&$w_{69}$\\
\hline
$74$&$1$&$4$&$w_{74}$\\
\hline
$82$&$1$&$3$&$w_{41}$\\
\hline
$86$&$1$&$4$&$w_{86}$\\
\hline
$87$&$1$&$5$&$w_{87}$\\
\hline
$93$&$1$&$5$&$w_{31}$\\
\hline

$94$&$1$&$3$&$w_{94}$\\
\hline

$95$&$1$&$7$&$w_{95}$\\
\hline

$111$&$1$&$7$&$w_{111}$\\
\hline

$119$&$1$&$9$&$w_{119}$\\
\hline
$134$&$1$&$6$&$w_{134}$\\
\hline
$146$&$1$&$7$&$w_{146}$\\
\hline
$159$&$1$&$9$&$w_{159}$\\
\hline

$194$&$1$&$9$&$w_{194}$\\
\hline
$206$&$1$&$9$&$w_{206}$\\
\hline
\end{tabular}
\end{minipage}
\hspace{0.5cm}
\begin{minipage}[b]{0.45\linewidth}
\centering
\begin{tabular}{|c|c|c|c|}
\hline
$D$&$N$&$g(D,N)$&$w$\\
\hline
$6$&$11$&$3$&$w_{66}$\\
\hline
$6$&$17$&$3$&$w_{34}$\\
\hline
$6$&$19$&$3$&$w_{114}$\\
\hline
$6$&$29$&$5$&$w_{174}$\\
\hline
$6$&$31$&$5$&$w_{186}$\\
\hline
$6$&$37$&$5$&$w_{222}$\\
\hline
$10$&$11$&$5$&$w_{110}$\\
\hline
$10$&$13$&$3$&$w_{65}$\\
\hline
$10$&$19$&$5$&$w_{38}$\\
\hline
$10$&$23$&$9$&$w_{230}$\\
\hline
$14$&$3$&$3$&$w_{14}$\\
\hline
$14$&$5$&$3$&$w_{14}$\\
\hline
$15$&$2$&$3$&$w_{15}$\\
\hline
$15$&$4$&$5$&$w_{15}$\\
\hline
$21$&$2$&$3$&$w_{7}$\\
\hline
$22$&$3$&$3$&$w_{66}$\\
\hline
$22$&$5$&$5$&$w_{110}$\\
\hline
$26$&$3$&$5$&$w_{26}$\\
\hline
$39$&$2$&$7$&$w_{39}$\\
\hline
\end{tabular}
\end{minipage}
\end{table}
\end{theorem}

\begin{subsection}{Method}
Let us briefly explain our method to compute equations of these hyperelliptic
Shimura curves. Before doing that, we remark that in addition to
Borcherds forms and Schofer's formula, arithmetic properties of
CM-points are also crucial in our computation. We refer the reader to
\cite[Section 5]{Victor-genus-one} for an explicit description of the
Shimura reciprocity law.

Let $X_0^D(N)$ be one of the curves in Ogg's list. Since the
hyperelliptic involution of $X_0^D(N)$ is an Atkin-Lehner involution,
the genus of $X_0^D(N)/W_{D,N}$ is necessarily $0$. Moreover, it turns
out that any of these $X_0^D(N)/W_{D,N}$ has at least three rational
CM-points $\tau_1$, $\tau_2$, and $\tau_3$ of discriminants $d_1$,
$d_2$, and $d_3$, respectively. Thus, there is a Hauptmodul $s(\tau)$
on $X_0^D(N)/W_{D,N}$ with $s(\tau_1)=\infty$, $s(\tau_2)=0$ and
$s(\tau_3)\in\Q$.

Let $W$ be a subgroup of index $2$ of $W_{D,N}$. Suppose that $w_m$ is
an element of $W_{D,N}$ not in $W$. Then $X_0^D(N)/W\to
X_0^D(N)/W_{D,N}$ is a double cover ramified at certain CM-points that
are fixed points of the Atkin-Lehner involutions $w_{mn/\gcd(m,n)^2}$,
$w_n\in W$. Thus, an equation of $X_0^D(N)/W$ is
\begin{equation} \label{equation: strategy 1}
  y^2=a\prod_{\tau\text{ ramified},s(\tau)\neq\infty}(s-s(\tau)),
\end{equation}
where $a$ is a rational number depending on the arithmetic of
$X_0^D(N)/W$. Specifically, $a$ must be a rational number such that
$(a\prod_{\tau\text{ ramified}}(-s(\tau)))^{1/2}$ is in the field of
definition of a CM-point of discriminant $d_2$ on $X_0^D(N)/W$. As an
additional check, note that when $\tau_1$ is not a ramified point, the
right-hand side of \eqref{equation: strategy 1} is a polynomial of even
degree and $a$ must be a rational number such that $\sqrt a$
is in the field of definition of a CM-point of discriminant $d_1$ on
$X_0^D(N)/W$.

To determine the coefficients of the polynomial on the right-hand side
of \eqref{equation: strategy 1}, we simply have to know the values of
$s$ and $y^2$ at sufficiently many points. For this purpose, we
observe that $s$ and $y^2$ are both modular functions on
$X_0^D(N)/W_{D,N}$ with divisors supported on CM-divisors. Thus, they
are both realizable as Borcherds forms. (This is proved in Proposition
\ref{proposition: all Borcherds} for the case of even $D$. We do not
try to give a proof for the case of odd $D$, but in practice, we are
always able to realize modular forms encountered as Borcherds forms.)
Then Schofer's formula gives us the absolute values of norms of values
of $s$ and $y^2$ at CM-points.

In order to obtain the actual values of $s$,
not just the absolute values, we let $\wt s$ be another Hauptmodul
with $\wt s(\tau_1)=\infty$, $\wt s(\tau_3)=0$, and $\wt
s(\tau_2)\in\Q$. We may also realize $\wt s$ as a Borcherds
form. Then the absolute values of $s(\tau_3)$ and $\wt s(\tau_2)$
obtained using Schofer's formula determine the relation $\wt s=bs+c$
between $s$ and $\wt s$. If $d$ is a discriminant such that there is
only one CM-point $\tau_d$ of discriminant $d$, then knowing the
values of $|s(\tau_d)|$ and $|\wt s(\tau_d)|=|bs(\tau_d)+c|$ from
Schofer's formula is enough to determine the value of $s(\tau_d)$. If
there are two CM-points $\tau_d$ and $\tau_d'$ of discriminant $d$,
then from the values of $|s(\tau_d)s(\tau_d')|$ and
$|(bs(\tau_d)+c)(bs(\tau_d')+c)|$ we get four possible candidates for
the minimal polynomial of $s(\tau_d)$. In almost all cases we
consider, there is precisely one of the four candidates that have
roots in the correct field. This gives us the values of $s(\tau_d)$
and $s(\tau_d')$. In practice, we do not need information from
discriminants with more than two CM-points.

The determination of values of $y^2$ from absolute values is easier.
For example, when $d$ is a discriminant such that there is only one
CM-point of discriminant $d$ on $X_0^D(N)/W_{D,N}$, $y(\tau_d)$ is
either $\sqrt{|y(\tau_d)^2|}$ or $\sqrt{-|y(\tau_d)^2|}$, but only one
of them is in the correct field.

Having determined values of $s$ and $y^2$ at sufficiently many
CM-points, it is straightforward to determine the equation of
$X_0^D(N)/W$. Then we will either work out equations of $X_0^D(N)/W'$
for various other subgroups $W'$ of $W_{D,N}$ of index $2$ or use
arithemtic properties of $X_0^D(N)$ to determine equations of
$X_0^D(N)$. We will give several examples in the next section.
\end{subsection}

\begin{subsection}{Examples}

\begin{Example} Consider $X_0^{15}(1)$. In \cite[Proposition
  3.2.1]{Jordan}, it is shown that an equation of $X_0^{15}(1)$ is
  $$
    3y^2+(x^2+3)(x^2+243)=0.
  $$
  In this example, we will use Borcherds forms and Schofer's formula
  to obtain this result.

  The curve $X=X_0^{15}(1)$ and its various Atkin-Lehner quotients have
  the following geometric information.
  $$ \extrarowheight3pt
  \begin{array}{c|cl} \hline\hline
  \text{curve} & \text{genus} & \text{elliptic points} \\ \hline
  X & 1 & \CM(-3)^{\times 2} \\
  X/w_3 & 0 & \CM(-3)^{\times 2}, \CM(-12)^{\times 2} \\
  X/w_5 & 1 & \CM(-3) \\
  X/w_{15} & 0 & \CM(-3), \CM(-15)^{\times 2},
                 \CM(-60)^{\times 2} \\
  X/W_{15,1} & 0 & \CM(-3), \CM(-12), \CM(-15), \CM(-60) \\
  \hline\hline
  \end{array}
  $$
  According to the method described in the previous section, we
  should first determine the equation of $X/W$ for some subgroup $W$
  of $W_{15,1}$ of index $2$. Here we choose $W=\gen{w_3}$. The
  double cover $X/w_3\to X/W_{15,1}$ is ramified at the CM-points
  $\tau_{-15}$ and $\tau_{-60}$ of discriminants $-15$ and $-60$. Let
  $s(\tau)$ be a Hauptmodul on $X/W_{15,1}$ taking values $0$
  and $\infty$ at CM-points $\tau_{-12}$ and $\tau_{-3}$ of
  discriminants $-12$ and $-3$, respectively, and satisfying
  $s(\tau_{-40})\in\Q$, where $\tau_{-40}$ is the unique CM-point
  of discriminant $-40$ on $X/W_{15,1}$. Then an equation of $X/w_3$ is
  $$
    y^2=a(s-s(\tau_{-15}))(s-s(\tau_{-60})),
  $$
  where $a=-3r^2$ for some $r\in\Q$ since a CM-point of discriminant
  $-3$ on $X/w_3$ is defined over $\Q(\sqrt{-3})$. The divisor of
  $y^2$, as a function on $X/W_{15,1}$, is $P_{-15}+P_{-60}-2P_{-3}$.
  Let also $\wt s$ be a Hauptmodul with $\wt s(\tau_{-15})=\infty$,
  $\wt s(\tau_{40})=0$, and $\wt s(\tau_{-60})\in\Q$.
  According to our method, we should construct Borcherds forms with
  divisors $P_{-12}-P_{-3}$, $P_{-40}-P_{-3}$, and
  $P_{-15}+P_{-60}-2P_{-3}$. A Borcherds form $P_{-12}-P_{-3}$ is
  constructed in Example \ref{example: 15 2}. Denote this Borcherds
  form by $\psi_1$. Here let us construct
  the other two Borcherds forms.

  Using the notations in Example \ref{example: 15} and letting $h$ be
  the modular form in \eqref{equation: 15 h}, we find that
  $$
    f_{10}-f_7+f_5-2f_4-3h=q^{-10}-3q^{-2}+q^{-1}-35+\cdots
  $$
  and
  $$
    \frac{60e^{2\pi i/8}}{15\sqrt 2}(f_{10}-f_7+f_5-2f_4-3h)|S
   =6q^{-3/4}+c_0+c_1q^{1/60}+\cdots
  $$
  for some $c_j$. Thus, by Lemma \ref{lemma: divisor infinity 0}, the
  Borcherds form $\psi_2$ associated to this modular form has a divisor
  $P_{-40}-P_{-3}$. Also, we have
  \begin{equation*}
  \begin{split}
  &2f_{15}+4f_{13}+2f_{12}-2f_{10}-4f_9-7f_8-10f_7+10f_6+3f_5-23f_4-6h \\
  &\qquad=2q^{-15}-q^{-8}-5q^{-2}-2q^{-1}-78+\cdots
  \end{split}
  \end{equation*}
  and
  \begin{equation*}
  \begin{split}
  &\frac{60e^{2\pi i/8}}{15\sqrt 2}
   (2f_{15}+4f_{13}+2f_{12}-2f_{10}-4f_9-7f_8-10f_7+10f_6+3f_5-23f_4-6h)|S \\
  &\qquad=12q^{-3/4}+c_0'+c_1'q^{1/60}+\cdots
  \end{split}
  \end{equation*}
  for some $c_j'$. Therefore, the Borcherds form $\psi_3$ associated
  to this modular form has a divisor $P_{-15}+P_{-60}-2P_{-3}$. An
  application of Schofer's formula yields the following values of
  Borcherds forms at CM-points.
  $$ \extrarowheight3pt
  \begin{array}{c|ccccccc} \hline\hline
    & -3 & -7 & -12 & -15 & -40 & -43 & -60 \\ \hline
  |\psi_1| & \infty & 1 & 0 & 3 & 1/2 & 1/16 & 1/27 \\
  5^{-3/2}|\psi_2| & \infty & 1/9 & 1/27 & 5/27 & 0 & 1/24 & 25/3^6 \\
  |\psi_3| & \infty & 35/3^6 & 5/2^43^5 & 0 & 5^4/2^63^6 &
           43^15^17^2/2^{12}3^6 & 0\\
  \hline \hline
  \end{array}
  $$
  Observe that multiplying $\psi_j$ by a scalar of absolute value $1$
  does not change the absolute value of its value at a CM-point. Thus,
  we may as well assume that $\psi(\tau_{-15})=-3$,
  $5^{-3/2}\psi_2(\tau_{-15})=5/27$, and
  $\psi_3(\tau_{-7})=-35/3^6$. Also, we choose $s$, $\wt s$, and $y$
  such that $s(\tau_{-15})=-243$, $\wt s(\tau_{-15})=5$, and
  $y(\tau_{-7})^2=-2^43^47$. Therefore,
  we have
  $$
    s=81\psi_1, \qquad\wt s=27\cdot5^{-3/2}\psi_2, \qquad
    y^2=\frac{2^43^{10}}5\psi_3.
  $$
  Then from the table above, we obtain
  $$
    |s(\tau_{-12})|=0, \quad |\wt s(\tau_{-12})|=1, \quad
    |s(\tau_{-40})|=81/2, \quad |\wt s(\tau_{-40})|=0,
  $$
  which implies that $\wt s$ is equal to one of $\pm 2s/81\pm 1$. As
  $s(\tau_{-15})=-243$ and $\wt s(\tau_{-15})=5$, we find that
  $\wt s=-2s/81-1$. Then the table above and the requirement that
  $y(\tau_d)$ must lie in the correct field yield
  $$ \extrarowheight3pt
  \begin{array}{c|ccccccc} \hline\hline
    & -3 & -7 & -12 & -15 & -40 & -43 & -60 \\ \hline
  s & \infty & 81 & 0 & -243 & -81/2 & 81/16 & -3 \\
  \wt s & \infty & -3 & -1 & 5 & 0 & -9/8 & -25/27 \\
  y^2 & \infty & -2^43^47 & -3^5 & 0 & -3^45^3/4 & -3^47^243/2^8 & 0 \\
  \hline\hline
  \end{array}
  $$
  It follows that an equation of $X/w_3$ is $3y^2+(s+243)(s+3)=0$.

  Furthermore, the double cover $X/w_{15}\to X/W_{15,1}$ is ramified
  at CM-points of discriminants $-3$ and $-12$. Thus, an equation of
  $X/w_{15}$ is $x^2=bs$ for some $b$. As CM-points of discriminant
  $-7$ are rational points on $X/w_{15}$, we find that $b$ must be a
  square, which we may assume to be $1$. That is, we have $s=x^2$.
  Therefore, we have $3y^2+(x^2+243)(x^2+3)=0$, which can be taken to
  be an equation of $X$, agreeing with Jordan's result.

  We remark that Elkies \cite{Elkies-computation} has used Schwarzian
  differential equations to compute numerically the values of $s$ at
  many CM-points. (His modular function differs from our $s$ by a factor of
  $-3$.) Using Borcherds forms, we verify that all the entries in
  Table 6 of \cite{Elkies-computation} are correct.
\end{Example}

\begin{Example} Consider the Shimura curve $X=X_0^{26}(1)$. In
  \cite{Victor-genus-two}, Gonz\'alez and Rotger proved that an
  equation of $X$ is
  $$
   y^2=-2x^6+19x^4-24x^2-169.
  $$
  In this example, we will obtain this result using Borcherds forms.

  We have the following informations about $X$ and its Atkin-Lehner
  quotients.
  $$ \extrarowheight3pt
  \begin{array}{c|cl} \hline\hline
  \text{curve} & \text{genus} & \text{elliptic points} \\ \hline
  X & 2 & \text{none} \\
  X/w_2 & 1 & \CM(-8)^{\times 2} \\
  X/w_{13} & 1 & \CM(-52)^{\times 2} \\
  X/w_{26} & 0 & \CM(-104)^{\times 6} \\
  X/W_{26,1} & 0 & \CM(-8),\CM(-52),\CM(-104)^{\times 3} \\
    \hline\hline
  \end{array}
  $$
  The double cover $X/w_{13}\to X/W_{26,1}$ is ramified at the
  CM-point of discriminant $-8$ and the three CM-points of
  discriminant $-104$. Let $s$ be a Hauptmodul on $X/W_{26,1}$
  with $s(\tau_{-8})=\infty$, $s(\tau_{-52})=0$, and
  $s(\tau_{-11})\in\Q$. Then an equation of $X/w_{13}$ is
  $$
    y^2=a\prod_{\tau:\CM\text{-points of discriminant }-104}(s-s(\tau))
  $$
  for some nonzero rational number $a$. As a modular function on
  $X/W_{26,1}$, we have $\div y^2=P_{-104}-3P_{-8}$. Let $\wt s$ be
  another Hauptmodul on $X/W_{26,1}$ with $\wt
  s(\tau_{-8})=\infty$, $\wt s(\tau_{-11})=0$, and $\wt
  s(\tau_{-52})\in\Q$. We now realize $s$, $\wt s$, and $y^2$ as
  Borcherds forms.

  Let $f_j$ be modular forms in $M^!(52)\cap\Z((q))$ with a pole of
  order $j$ at $\infty$ and a leading coefficient $1$ constructed in
  Example \ref{example: 26-1}. Using these $f_j$, we find three
  modular forms
  \begin{equation*}
  \begin{split}
    g_1&=2q^{-13}-2q^{-2}-4q^{-1}+2q-2q^2-2q^3+\cdots, \\
    g_2&=q^{-11}+2q^{-7}-2q^{-2}+4q+4q^4+\cdots, \\
    g_3&=2q^{-26}+6q^{-7}-6q^{-2}+2q^{-1}+10q-8q^2+\cdots
  \end{split}
  \end{equation*}
  in $M^!(52)$. Let $\psi_j$, $j=1,2,3$, be the Borcherds forms
  associated to $g_j$. By Lemma \ref{lemma: divisor at infinity},
  $$
    \div\psi_1=P_{-52}-P_{-8}, \qquad
    \div\psi_2=P_{-11}-P_{-8}, \qquad
    \div\psi_3=P_{-104}-3P_{-8}.
  $$
  Thus, $\psi_j$ are scalar multiples of $s$, $\wt s$, and $y^2$,
  respectively. Applying Schofer's formula, we get
  $$ \extrarowheight3pt
  \begin{array}{c|ccccccc} \hline\hline
    & -8 & -11 & -19 & -20 & -24 & -52 & -67 \\ \hline
  |\psi_1| & \infty & 1 & 9 & 5 & 3 & 0 & 81/25 \\
  |\psi_2| & \infty & 0 & 64 & 32 & 32 & 8 & 2^67/5^2 \\
  13^{-3}|\psi_3| & \infty & 2^{10}11 & 2^{10}19 & 2^{12} & 2^{13}
   & 2^613^5 & 2^{10}41^267/5^6 \\ \hline\hline
  \end{array}
  $$
  Since multiplying $\psi_j$ by a suitable factor of absolute value
  $1$ does not change the absolute value of its value at a CM-point,
  we may as well assume that $\psi_1(\tau_{-11})=1$,
  $\psi_2(\tau_{-52})=8$, and $\psi_3(\tau_{-11})=-2^{10}11^113^3$.
  Also, we choose $s$, $\wt s$, and $y$ in a way such that
  $s(\tau_{-11})=1$, $\wt s(\tau_{-52})=1$, and
  $y(\tau_{-11})^2=-2^411$, i.e., $s=\psi_1$, $\wt s=\psi_2/8$, and
  $y^2=\psi_3/2^613^3$. Then we have $\wt s=1-s$ and from the table
  above we get
  $$ \extrarowheight3pt
  \begin{array}{c|ccccccc} \hline\hline
    & -8 & -11 & -19 & -20 & -24 & -52 & -67 \\ \hline
  s & \infty & 1 & 9 & 5 & -3 & 0 & 81/25 \\
  \wt s=1-s & \infty & 0 & -8 & -4 & 4 & 1 & -56/25 \\
  y^2 & \infty & -2^411 & -2^419 & -2^6 & 2^7 & -13^2 & -2^441^267/5^6
    \\ \hline\hline
  \end{array}
  $$
  (The signs of $y(\tau_d)^2$ are determined by the Shimura reciprocity
  law.) From the data, we easily deduce that the relation between $y$
  and $s$ is
  $$
    y^2=-2s^3+19s^2-24s-169,
  $$
  which is an equation for $X_0^{26}(1)/w_{13}$.

  On the other hand, the cover $X_0^{26}(1)/w_{26}\to
  X_0^{26}(1)/W_{26,1}$ is ramified at the CM-points of discriminants
  $-8$ and $-52$. Thus, there is a modular function $x$ on
  $X_0^{26}(1)/w_{26}$ with $x^2=cs$ for some rational number
  $c$. Since CM-points of discriminant $-11$ are rational points on
  $X_0^{26}(1)/w_{26}$, we conclude that $c$ can be chosen to be $1$.
  Hence $y^2=-2x^6+19x^4-24x^2-169$ is an equation for $X_0^{26}(1)$ and the Atkin-Lehner involutions are given by
  $$
    w_{2}:(x,y)\mapsto\left(-x, -y\right), \qquad
    w_{26}:(x,y)\mapsto(x,-y).
  $$
\end{Example}


\begin{Example} Consider $X=X_0^{111}(1)$. We have the following
  informations.
  $$ \extrarowheight3pt
  \begin{array}{c|cl} \hline\hline
  \text{curve} & \text{genus} & \text{elliptic points} \\ \hline
   X & 7 & \text{none} \\
  X/w_3 & 4 & \text{none} \\
  X/w_{37} & 3 & \CM(-148)^{\times 4} \\
  X/w_{111} & 0 & \CM(-111)^{\times 8}, \CM(-444)^{\times 8} \\
  X/W_{111,1} & 0 & \CM(-148)^{\times 2}, \CM(-111)^{\times 4},
    \CM(-444)^{\times 4} \\ \hline\hline
  \end{array}
  $$
  Let $s$ and $\wt s$ be modular functions on $X/W_{111,1}$ such that
  $s(\tau_{-15})=\wt s(\tau_{-15})=\infty$, $s(\tau_{-60})=0$, $\wt
  s(\tau_{-24})=0$, $s(\tau_{-24})=1$, and $\wt s(\tau_{-60})=1$, so
  that $\wt s=1-s$. Then an equation for $X/w_{37}$ is
  \begin{equation} \label{equation: 111}
    y^2=a\prod_{\tau\in\CM(-111),\CM(-444)}(s-s(\tau)).
  \end{equation}
  As CM-points of discriminant $-60$ on $X/w_{37}$ lie in
  $\Q(\sqrt{-3})$, we choose $y$ such that $y(\tau_{-60})^2=-27$. Then
  realizing $s$, $\wt s$, and $y^2$ as Borcherds forms and using
  Schofer's formula, we deduce the following values of these modular
  functions at rational CM-points.
  $$ \tiny\extrarowheight3pt
  \begin{array}{c|ccccccccc} \hline\hline
    & -15 & -19 & -24 & -43 & -51 & -60 & -163 & -267 & -555 \\ \hline
  s & \infty & 3 & 1 & -3 & -1 & 0 & 3/5 & 1/3 & 5 \\
  y^2 & \infty & -2^83^219 & -2^83 & -2^83^243 & -2^83 & -27
      & -2^83^213^2163/5^8 & -2^813^2/3^7 & -2^83^137^2 \\
    \hline\hline
  \end{array}
  $$
  As the right-hand side of \eqref{equation: 111} is a polynomial of
  degree $8$, these CM-values are not sufficient to determine the
  equation and we will need values of $s$ and $y^2$ at some degree $2$
  CM-points.

  Let $\tau_{-39}$ and $\tau_{-39}'$ be the two CM-points of
  discriminant $-39$ on $X/W_{111,1}$. Schofer's formula yields
  $$
    |s(\tau_{-39})s(\tau_{-39}')|=3, \qquad
    |(1-s(\tau_{-39}))(1-s(\tau_{-39}'))|=4.
  $$
  From the Shimura reciprocity law, we know that
  $s(\tau_{-39})\in\Q(\sqrt{-3})$. Thus,
  $$
    s(\tau_{-39})s(\tau_{-39}')=3, \qquad
    (1-s(\tau_{-39}))(1-s(\tau_{-39}'))=4.
  $$
  From these, we deduce that $s(\tau_{-39})=\pm\sqrt{-3}$. Likewise,
  we find that the values of $s$ at the two CM-points
  $\tau_{-52},\tau_{-52}'$ of discriminants $-52$ are $1\pm
  2\sqrt{-1}$. Also, we have
  $$
    y(\tau_{-39})^2y(\tau_{-39}')^2=2^{16}3^213, \qquad
    y(\tau_{-52})^2y(\tau_{-52}')^2=2^{16}13^2.
  $$
  These data are enough to determine the equation of $X/w_{37}$. We
  find that it is
  \begin{equation} \label{equation: 111 2}
   y^2=-(3s^4-6s^3+28s^2-10s+1)(s^4-2s^3+4s^2+18s+27).
  \end{equation}

  Similarly, we can compute an equation for $X/w_{111}$ by observing
  that $X/w_{111}\to X/W_{111,1}$ is ramified at the two CM-points of
  discriminant $-148$, constructing a Borcherds form with divisor
  $P_{-148}-2P_{-15}$, and evaluating at various CM-points and obtain
  $$
    t^2=5s^2-18s+45.
  $$
  The conic has rational points $(s,t)=(3,\pm 6)$ corresponding the
  two CM-points of discriminant $-19$ on $X/w_{111}$, so it admits a
  rational parameterization. Specifically, let $x$ be a Hauptmodul on
  $X/w_{111}$ that has a pole and a zero at the two CM-points of
  discriminant $-19$, respectively, and takes rational values at
  CM-points of discriminant $-43$. (In terms of $(s,t)$, the
  coordinates are $(-3,\pm 12)$.) Then
  $$
    x=\frac{c(s-3)}{s-t+3}
  $$
  for some rational number $c$. Choose $c=2$ so that it takes values
  $\pm 1$ at the CM-points of discriminant $-43$. We have
  $$
    (s,t)=\left(\frac{3x^2-3x-3}{x^2+x-1},\frac{6x^2+6}{x^2+x-1}\right).
  $$
  Plugging in $s=(3x^2-3x-3)/(x^2+x-1)$ in \eqref{equation: 111 2} and
  making a slight change of variables, we find that an equation of
  $X_0^{111}(1)$ is
  \begin{equation*}
  \begin{split}
    z^2&=-(x^8-3x^5-x^4+3x^3+1) \\
    &\qquad\times(19x^8-44x^7-16x^6+55x^5+37x^4-55x^3-16x^2+44x+19)
  \end{split}
  \end{equation*}
  with the actions of the Atkin-Lehner involutions given by
  $$
    w_{37}:(x,z)\mapsto\left(-\frac1x,\frac z{x^8}\right), \qquad
    w_{111}:(x,z)\mapsto(x,-z).
  $$
\end{Example}

\begin{Example} Consider $X=X_0^{146}(1)$. Let $s$ be the Hauptmodul
  of $X/W_{146,1}$ such that $s(\tau_{-43})=0$,
  $s(\tau_{-11})=\infty$, and $s(\tau_{-20})=1$. Let $y$ be a modular
  function on $X/w_{73}$ such that $y^2$ is a modular function on
  $X/W_{146,1}$ with $\div y^2=P_{-584}-8P_{-11}$. Realizing $s$ and
  $y^2$ as Borcherds forms and suitably scaling $y^2$, we find that an
  equation for $X/w_{73}$ is
  \begin{equation} \label{equation: 146}
    y^2=-11s^8+82s^7-309s^6+788s^5-1413s^4+1858s^3-1803s^2+1240s-688.
  \end{equation}
  Similarly, we find that an equation for $X/w_{146}$ is $t^2=s^2+4$,
  where the roots of $s^2+4$ correspond the to CM-points of
  discriminant $-292$. We choose a rational parameterization of the
  conic to be
  $$
    (s,t)=\left(\frac{x^2-1}x,\frac{x^2+1}x\right),
  $$
  where $x$ is actually a modular function on $X/w_{146}$ that has a
  pole and a zero at the two CM-points of discriminant $-11$ and is
  equal to $\pm 1$ at the two CM-points of discriminant $-43$ on
  $X/w_{146}$. Substituting $s=(x^2-1)/x$ in \eqref{equation: 146} and
  making a change of variables, we find that an equation for $X$ is
  \begin{equation*}
  \begin{split}
    z^2&=-11x^{16}+82x^{15}-221x^{14}+214x^{13}+133x^{12}-360x^{11}-170x^{10}
       +676x^9 \\
   &\qquad-150x^8-676x^7-170x^6+360x^5+133x^4-214x^3-221x^2-82x-11,
  \end{split}
  \end{equation*}
  where the Atkin-Lehner involutions are given by
  $$
    w_{73}:(x,y)\mapsto\left(-\frac 1x,\frac y{x^8}\right), \qquad
    w_{146}:(x,y)\mapsto(x,-y).
  $$
\end{Example}

\begin{Example}
  Let $X=X_0^{14}(5)$. Let $s$ be the Hauptmodul of $X/W_{14,5}$ such
  that $s(\tau_{-4})=\infty$, $s(\tau_{-11})=1$, and
  $s(\tau_{-35})=0$. We find that an equation for $X/\gen{w_5,w_7}$ is
  $$
    y^2=-16s^3-347s^2+222s-35,
  $$
  which is isomorphic to the elliptic curve $E_{\text{14A5}}$ in Cremona's table
  \cite{Cremona}. (In fact, we can use Cerednik-Drinfeld theory of
  $p$-adic uniformization of Shimura curves \cite{Boutot-Carayol} to
  determine the singular fibers of $X/\langle w_5,w_7\rangle$ and conclude that
  it is isomorphic to $E_{\text{14A5}}$.) The double cover $X/\gen{w_5,w_{14}}\to
  X/W_{14,5}$ is ramified at the CM-point of discriminant $-4$ and the
  CM-point of discriminant $-35$, so that there is a Hauptmodul $t$ of
  $X/\gen{w_5,w_{14}}$ such that $t^2=cs$ for some rational number
  $c$. In addition, the CM-points of discriminant $-11$ on
  $X/\gen{w_5,w_{14}}$ are rational points. Thus, we may choose $c=1$
  and find that an equation for $X/w_5$ is
  \begin{equation} \label{equation: 14 5}
    y^2=-16t^6-347t^4+222t^2-35.
  \end{equation}

  We next determine an equation of $X/w_{14}$. The double cover
  $X/w_{14}\to X/\gen{w_5,w_{14}}$ is ramified at the two CM-points of
  discriminant $-280$. Using Schofer's formula, we find
  $s(\tau_{-280})=5/16$ and thus, an equation for $X/w_{14}$ is
  $u^2=d(16t^2-5)$ for some rational number. The point such that
  $t=0$ is the CM-point of discriminant $-35$. Therefore, we may
  choose $d=-1$ and find that an equation for $X/w_{14}$ is
  $$
    u^2+16t^2=5.
  $$
  This is a conic with rational points and a rational parameterization
  is
  $$
    (t,u)=\left(\frac{x^2-x-1}{2x^2+2},\frac{x^2+4x-1}{x^2+1}\right).
  $$
  Substituting $t=(x^2-x-1)/(2x^2+2)$ into \eqref{equation: 14 5} and
  making a change of variables, we conclude that an equation for
  $X_0^{14}(5)$ is
  $$
    z^2=-23x^8-180x^7-358x^6-168x^5-677x^4+168x^3-358x^2+180x-23,
  $$
  on which the actions of the Atkin-Lehner operators are given by
  $$
    w_2:(x,z)\mapsto\left(-\frac1x,\frac z{x^4}\right), \qquad
    w_{14}:(x,z)\mapsto(x,-z),
  $$
  and
  $$
    w_{35}:(x,z)\mapsto\left(\frac{x+2}{2x-1},\frac{25z}{(2x-1)^4}\right).
  $$

  Note that $X_0^{14}(5)/w_{14}$ is an example of Shimura curves of
  genus zero that is isomorphic to $\P^1$ over $\Q$ but none of the
  rational points is a CM-point.
\end{Example}

\begin{Example} Let $X=X_0^{10}(19)$. Let $s$ be the Hauptmodul of
  $X/W_{10,19}$ such that $s(\tau_{-8})=0$, $s(\tau_{-40})=\infty$,
  and $s(\tau_{-3})=1$. We find that an equation for
  $X/\gen{w_2,w_{95}}$ is
  $$
    y^2=-8s^3+57s^2-40s+16,
  $$
  which is isomorphic to the elliptic curve $E_{\text{190A1}}$ in Cremona's table
  \cite{Cremona}. Also, the double cover $X/\gen{w_5,w_{38}}\to
  X/W_{10,19}$ is ramified at the CM-point of discriminant $-8$ and
  the CM-point of discriminant $-40$. The CM-points of discriminant
  $-3$ are rational points on $X/\gen{w_5,w_{38}}$. Thus, arguing as
  before, we deduce that an equation for $X/w_{190}$ is
  $y^2=-8x^6+57x^4-40x^2+16$. Moreover, the double cover $X/w_{38}\to
  X/\gen{w_5,w_{38}}$ is ramified at the two CM-points of discriminant
  $-760$. Since $s(\tau_{760})=32/5$ and the point with $s=0$ is a
  CM-point of discriminant $-8$, we see that an equation for $X/w_{38}$ is
  $z^2=5x^2-32$. We conclude that an equation for $X$ is
  $$
    \begin{cases}
    y^2=-8x^6+57x^4-40x^2+16, \\
    z^2=5x^2-32, \end{cases}
  $$
  with the actions of the Atkin-Lehner involutions given by
  \begin{equation*}
  \begin{split}
    w_2&:(x,y,z)\mapsto(-x,y,z), \\
    w_5&:(x,y,z)\mapsto(x,-y,-z), \\
    w_{19}&:(x,y,z)\mapsto(-x,-y,z).
  \end{split}
  \end{equation*}
  Note that as the conic $z^2=5x^2-32$ has only real points, but no
  rational points, the Shimura curve $X$ is hyperelliptic over $\R$,
  but not over $\Q$.
\end{Example}

\begin{Remark}
  In \cite{Ogg}, Ogg mentioned that $X_0^{10}(19)$ and
  $X_0^{14}(5)$ are the only two hyperelliptic curves that he could not
  determine whether they are hyperelliptic over $\Q$. Our computation
  shows that $X_0^{14}(5)$ is hyperelliptic over $\Q$ because the
  curve $X_0^{14}(5)/w_{14}$ has rational points, but $X_0^{10}(19)$
  is not hyperelliptic over $\Q$.
\end{Remark}

\begin{Remark} \label{remark: Tu}
Note that there is a curve, namely, $X=X_0^{15}(4)$, whose equation is
not obtained using our method. This is because the normalizer of the
Eichler order in this case is larger than the Atkin-Lehner group.
For this special curve, we use the result of Tu \cite{Tu}. In Lemma 13
of \cite{Tu}, it is shown that there is a Hauptmodul $t_4$ on
$X/\gen{w_3,w_5}$ that takes values $\pm1/\sqrt{-3}$,
$\pm\sqrt{-15}/5$ and $(\pm1\pm\sqrt{-15})/8$ at CM-points of
discriminants $-12$, $-15$, and $-60$, respectively. Since the double
cover $X/w_3\to X/\gen{w_3,w_5}$ ramifies at CM-points of
discriminants $-15$ and $-60$, while the cover $X/w_{15}\to
X/\gen{w_3,w_5}$ ramifies at CM-points of discriminant $-12$, we find
that there are rational numbers $a$ and $b$ such that the equations of
$X/w_3$ and $X/w_{15}$ are
$$
  y^2=a(4t_4^2-t_4+1)(4t_4^2+t_4+1)(5t_4^2+3), \qquad
  z^2=b(3t_4^2+1),
$$
respectively. To determine the constants $a$ and $b$, we further
recall that Lemma 13 of \cite{Tu} shows that there is a Hauptmodul
$t_2$ on $X_0^{15}(2)/\gen{w_3,w_5}$ with
$$
  t_2=\frac{5t_4^2+2t_4+1}{7t_4^2-2t_4+3}.
$$
From this, the CM-values of $t_2$ obtained using Schofer's formula,
and arithmetic properties of CM-points, we see that we can choose
$a=b=-1$. Note that $X_0^{15}(4)$ is one of the hypereliiptic Shimura
curves that are not hyperelliptic over $\R$ (see \cite{Ogg}).
\end{Remark}
\end{subsection}

\begin{subsection}{Additional examples}
\label{subsection: additional}
In the previous section, we determine the equations of hyperelliptic
Shimura curves $X_0^D(N)$ whose Atkin-Lehner involutions act as
hyperelliptic involutions. In particular, the curves
$X_0^D(N)/W_{D,N}$ are of genus $0$, so that Lemma \ref{lemma: genus
  0} applies and we have a simple criterion for a Borcherds form to
have a trivial character. Throughout this section, we make the following
assumption.

\begin{Assumption} \label{assumption: 1}
  The criterion for a Borcherds form to have a trivial character is
  also valid for the case when $N_B^+(\O)\backslash\H$ has a positive genus.
\end{Assumption}

\begin{Remark} Recall that a Fuchsian group of the first kind is
  generated by some elements
  $\alpha_1,\ldots,\alpha_g,\beta_1,\ldots,\beta_g,\gamma_1,\ldots,\gamma_n$
  with defining relations
  $$
    [\alpha_1,\beta_1]\ldots[\alpha_g,\beta_g]\gamma_1\ldots\gamma_n=1,
    \qquad \gamma_i^{k_i}=1, \quad i=1,\ldots,n,
  $$
  where $\alpha_j,\beta_j$ are hyperbolic elements,
  $[\alpha_j,\beta_j]$ denotes the commutator, $g$ is the genus, and
  $k_i$ is an integer $\ge 2$ or $\infty$. (See, for instance,
  \cite{Katok}.) Let $\chi$ be the character of a Borcherds form on
  $N_B^+(\O)\backslash\H$. The proof of Lemma \ref{lemma: genus 0}
  given in \cite{Yang-CM} shows that $\chi(\gamma_i)=1$ for all $i$ 
  if and only if the condition in Lemma \ref{lemma: genus 0} holds.
  Thus, what we really assume in Assumption \ref{assumption: 1} is
  that for all hyperbolic elements $\alpha$, we have $\chi(\alpha)=1$.
\end{Remark}

It turns out that sometimes our methods can also be used to determine
equations of $X_0^D(N)/W_{D,N}$ even when they have positive
genera, under Assumption \ref{assumption: 1}. However, the method
becomes less  systematic and it is not clear whether our methods will
always work in general, so we will only give two examples in this section.

\begin{Example} \label{example: 142}
Let $X=X_0^{142}(1)/W_{142,1}$. It is of genus $1$ and has
rational points (for instance, the CM-point of discriminant $-3$).
Thus, $X$ is a rational elliptic curve. From the Jacquet-Langlands
correspondence, we know that it must lie in the isogeny class 142A in
Cremona's table \cite{Cremona}, whose corresponding cusp form on
$\Gamma_0(142)$ has eigenvalues $-1$ for the Atkin-Lehner involutions
$w_2$ and $w_{71}$. Since the isogeny class contains only 
one curve, we immediately conclude that the equation for
$X$ is $E_{\text{142A1}}:y^2+xy+y=x^3-x^2-12x+15$. Here we
will use our method to get the same conclusion. An advantage of our
method is that we can determine the coordinates of all CM-points on
the curve. In the next section, we will discuss the heights of these
CM-points and verify Zhang's formula \cite{Zhang} for heights of
CM-points in this particular case.

By finding many suitable eta-products, we construct $4$ modular forms
$f_1,f_2,f_3,f_4$ in $M^!(284)$ with Fourier expansions
\begin{equation*}
\begin{split}
  f_1&=-2q^{-87} - 2q^{-71} - 2q^{-48} - 2q^{-36} + 2q^{-16} - 2q^{-15} - 
    2q^{-12} + 2q^{-9} \\
  &\qquad\qquad  - 2q^{-7} - 2q^{-3} + 2q^{-2} +
  2q^{-1} - 4q\cdots, \\
  f_2&=2q^{-116} - q^{-87} - q^{-79} + 2q^{-71} + 2q^{-60} - 2q^{-48} + q^{-43} - 
    2q^{-29} \\
  &\qquad\qquad + q^{-19} - 4q^{-15} - 2q^{-12} + 2q^{-7} - 2q^{-3} -4q +\cdots,
    \\
  f_3&=q^{-87} - 2q^{-79} + q^{-76} - 2q^{-71} + 2q^{-48} - q^{-40} + 3q^{-32} - 
    2q^{-20} \\
  &\qquad\qquad - q^{-19} - 2q^{-12} + 2q^{-10} + q^{-8} - 2q^{-7} -
    2q^{-2} + 4q + \cdots, \\
  f_4&=-q^{-79} + q^{-76} + q^{-48} - q^{-40} + q^{-32} - q^{-20} - q^{-12} + q^{-10} +
    q^{-6} \\
  &\qquad\qquad - q^{-5} - q^{-2} - q^7 +\cdots.
\end{split}
\end{equation*}
Let $\psi_j$, $j=1,\ldots,4$, be the Borcherds form associated to
$f_j$. Under Assumption \ref{assumption: 1}, these Borcherds forms
have trivial characters. We have
\begin{equation*}
\begin{split}
  \div\psi_{1}=P_{-4}+P_{-8}-2P_{-3}, \quad
 &\div\psi_{2}=P_{-19}+P_{-43}-2P_{-3}, \\
  \div\psi_{3}=P_{-8}+P_{-40}-2P_{-20}, \quad
 &\div\psi_{4}=P_{-19}+P_{-24}-2P_{-20}.
\end{split}
\end{equation*}
It is easy to show that $\psi_{2}$ is a polynomial of degree $1$ in
$\psi_{1}$ and $\psi_{4}$ is a polynomial of degree $1$ in
$\psi_{3}$. Thus, there are modular functions $x$ and $y$ on
$X$ such that $x$ has a double pole at $\tau_{-3}$ with
$x(\tau_{-4})=x(\tau_{-8})=0$ and $x(\tau_{-19})=x(\tau_{-43})=1$ and $y$ has a
double pole at $\tau_{-20}$ with $y(\tau_{-8})=y(\tau_{-20})=0$ and
$y(\tau_{-19})=y(\tau_{-24})=1$. Computing singular moduli using Schofer's
formula and choosing proper scalars of modulus $1$ for $\psi_j$, we find
$$
  x=2^{-10}\psi_{1}, \quad 1-x=\psi_{2}, \quad
  y=\psi_{3}/2, \quad 1-y=\psi_{4}/2,
$$
and the values of $x$ and $y$ at various CM-points are
$$ \extrarowheight3pt
\begin{array}{c|cccccccccc} \hline\hline
  & -3 & -4 & -8 & -19 & -20 & -24 & -40 & -43 & -148 & -232 \\ \hline
x & \infty & 0 & 0 &1 & -1 & 1/2 & -1/2 & 1 & -1 & -1/2 \\
y & 2 & 1/2 & 0 & 1 & \infty & 1 & 0 & 3/2 & -2 & -5 \\ \hline\hline
\end{array}
$$
Since $y(\tau_{-4})\neq y(\tau_{-8})$, $y$ cannot lie in $\C(x)$. Therefore,
$x$ and $y$ generate the field of modular functions on $X$.
From the table above, we determine that the relation between $x$ and
$y$ is
$$
  2(x+1)^2y^2-(8x^2+11x+1)y+4x(2x+1)=0.
$$
Set
\begin{equation*}
\begin{split}
  x_1&=-\frac{2(x+1)^2y-5x^2-3x-1}{x^2}, \\
  y_1&=-\frac{(4x^3+6x^2-2)y-5x^3-6x^2+x+1}{x^3}.
\end{split}
\end{equation*}
We find $y_1^2+x_1y_1+y_1=x_1^3-x_1^2-12x_1+15$, which is indeed the
elliptic curve $E_{\text{142A1}}$. The coordinates of the CM-points
above on this model are
$$ \small\extrarowheight3pt
\begin{array}{cccccccccc} \hline\hline
  -3 & -4 & -8 & -19 & -20 & -24 & -40 & -43 & -148 & -232 \\ \hline
  -Q & -2Q & O & Q & 2Q & 3Q & 4Q & -3Q & -4Q & -6Q \\ \hline\hline
\end{array}
$$
where $Q=(1,1)$ generates the group of rational points on $E_{\text{142A1}}$. 
\end{Example}

\begin{Example} \label{example: 302}
  We next consider $X=X_0^{302}(1)/W_{302,1}$, which has genus
  $2$. We can construct four modular forms $f_1,\ldots,f_4$ in
  $M^!(604)$ whose associated Borcherds forms $\psi_1,\ldots,\psi_4$
  have divisors
\begin{equation*}
\begin{split}
  \div\psi_{1}&=P_{-43}+P_{-72}-P_{-19}-P_{-88}, \\
  \div\psi_{2}&=P_{-20}+P_{-36}-P_{-19}-P_{-88}, \\
  \div\psi_{3}&=P_{-8}+2P_{-40}-P_{-4}-2P_{-88}, \\
  \div\psi_{4}&=P_{-11}+P_{-19}+P_{-43}-P_{-4}-2P_{-88},
\end{split}
\end{equation*}
respectively. In addition, under Assumption \ref{assumption: 1}, they
have trivial characters. Thus, $\psi_{1}$ generates the unique
genus-zero subfield of degree $2$ of the hyperelliptic function field, and
$\psi_{2}$ is a polynomial of degree $1$ in $\psi_{1}$. Also,
$\psi_{4}$ must be a polynomial of degree $1$ in $\psi_{3}$. To
see this, we observe that there exists a suitable linear combination
$a\psi_{3}+b\psi_{4}$ such that it is a function of degree less
than or equal to $2$ on $X$ and hence is contained in
$\C(\psi_{1})$. If this linear combination is not a constant
function, then it must have a pole at $\tau_{-88}$; otherwise it will
have only a pole of order $1$ at $\tau_{-4}$, which is impossible.
It follows that $\tau_{-19}$ is also a pole of this linear combination.
However, $\tau_{-19}$ can never be a pole of this function. Therefore, we
conclude that this linear combination is a constant function.

Let $x$ be the unique function on $X$ with
$\div x=\div\psi_{1}$ and $x(\tau_{-20})=2$ and $y$ be the unique
function with $\div y=\div\psi_{3}$ and $y(\tau_{-11})=1$. Computing
using Schofer's formula, we find
$$ \extrarowheight3pt
\begin{array}{c|cccccccccc} \hline\hline
d & -4 & -8 & -11 & -19 & -20 & -40 & -43 & -88 & -148 & -232 \\ \hline
x & -1 & 3/2 & 1 & \infty & 2 & 1 & 0 & \infty & 5/3 & 5/3 \\ 
y & \infty & 0 & 1 & 1 & -1 & 0 & 1 & \infty & -1/9 & -1/2 \\
\hline\hline
\end{array}
$$
From the coordinates at $\tau_{-4}$, $\tau_{-8}$, $\tau_{-19}$, $\tau_{-40}$, and
$\tau_{-88}$, we see that the relation between $x$ and $y$ is
$$
  a(x+1)y^2+(-2x^3+bx^2+cx+d)y+(2x-3)(x-1)^2=0
$$
for some rational numbers $a$, $b$, $c$, and $d$. Then the information
at the other CM-points yields
$$
  a=1, \quad b=11, \quad c=-13, \quad d=2.
$$
Setting
$$
  x_0=\frac{3-x}{1-x}, \qquad
  y_0=\frac{4(2x^3-11x^2+13x-2-2xy-2y)}{(1-x)^3},
$$
we get a Weierstrass model
$$
  y_0^2=x_0^6-18x_0^4+113x_0^2-32
$$
for $X$. Then letting
$$
  x_1=x_0^2, \quad y_1=y_0, \quad
  x_2=-32/x_0^2, \quad y_2=32y_0/x_0^3,
$$
we obtain modular parameterization of two elliptic curves
$$
  y_1^2=x_1^3-18x_1^2+113x_1-32, \qquad y_2^2=x_2^3+113x_2^2+576x_2+1024.
$$
The minimal models of these two elliptic curves are
$E_{\text{302C1}}:Y^2+XY+Y=X^3-X^2+3$ and
$E_{\text{302A1}}:Y^2+XY+Y=X^3+X^2-230X+1251$,
respectively, in Cremona's table. The coordinates of the
CM-points on the two curves are
$$ \tiny\extrarowheight3pt
\begin{array}{c|cccccccccc} \hline\hline
  & -4 & -8 & -11 & -19 & -20 & -40 & -43 & -88 & -148 & -232 \\ \hline
E_{\text{302A1}} & 2P-Q & 3P-Q & 2P & 4P & P & 3P & 3P-Q & P & 3P+Q &
2P-Q \\
E_{\text{302C1}} & 5R & R & O & 2R & 2R & O & -R & -2R & 5R & -5R \\
\hline\hline
\end{array}
$$
where $P=(-32,256)$ generates the torsion subgroup of order $5$ and
$Q=(-96,320)$ generates the free part of $E_{\text{302A1}}(\Q)$, and
$R=(9,16)$ generates the group of rational points on
$E_{\text{302C1}}$. In the next section,
we will address the issue of heights of CM-points on the Jacobians
of these elliptic curves.
\end{Example}
\end{subsection}
\end{section}

\begin{section}{Heights of CM-divisors on Shimura curves}
\label{section: height}

In this section, we will show how to use Borcherds forms to
explicitly compute height of CM-points on certain Shimura curves,
under Assumption \ref{assumption: 1}. This enables us
to verify Zhang's formula \cite{Zhang} for heights of CM-points in certain
cases.

\begin{subsection}{Zhang's formula}
Zhang's formula \cite{Zhang} holds for Shimura curves over general
totally real number fields. For the case of Shimura curves over $\Q$,
his formula can be described as follows.

Let $X=X_0^D(N)$ be a Shimura curve. For a newform $f$ on
$X_0^D(M)$, $M|N$, let $a_f(p)$, $p\nmid DN$, denote the eigenvalue
of the Hecke operator $T_p$ for $f$. The Hecke algebra acts on the
Jacobian $\Jac(X)$. We let $J_f$ be the maximal abelian subvariety
of $\Jac(X)$ annihilated by all $T_p-a_f(p)$, $p\nmid DN$. Then
$\Jac(X)$ is isogenous to $\prod_{[f]}J_f$, where the product
runs over all Galois conjugacy classes of newforms on $X_0^D(M)$,
$M|N$.

Assume that $D\neq 1$ so that $X$ has no cusps. Define a canonical
divisor class
\begin{equation} \label{equation: canonical divisor}
  \xi=\frac1{\mathrm{Vol}(X)}\left(
  [\Omega_X^1]+\sum_{P\in X}\left(1-\frac1{e_P}\right)[P]\right)
\end{equation}
of degree $1$ in $\Pic(X)\otimes\Q$, where
$$
  \mathrm{Vol}(X)=\frac1{2\pi}\int_X\frac{dxdy}{y^2}
 =\frac{DN}6\prod_{p|D}\left(1-\frac1p\right)
  \prod_{p|N}\left(1+\frac1p\right)
$$
and $e_P$ is the cardinality of the stabilizer subgroup of $P$ in
$X_0^D(N)$ and let $\iota:X\to\Jac(X)\otimes\Q$ be defined by
$\iota(P)=P-\xi$. There exists a positive integer $n$ such that
$n\iota$ is defined over $\Q$.

For a discriminant $d<0$ such that a CM-point $x$ of
discriminant $d$ exists on $X$, let $H_d$ denote the field of
definition of $x$ and set
\begin{equation} \label{equation: CM divisor}
  z=\frac1{e_x}\sum_{\sigma\in\mathrm{Gal}(H_d/\Q(\sqrt d))}\iota(x^\sigma).
\end{equation}
We have $z\in\Jac(X)(\Q(\sqrt d))\otimes\Q$. Let $z_f$ be the
$f$-isotypical component of $z$ in $J\otimes\R$. Then Zhang's formula
gives the N\'eron-Tate height $\gen{z_f,z_f}$ of $z_f$ in terms of the
derivative of $L_d(f,s)=L(f,s)L(f,d,s)$ at $s=1$, where $L(f,d,s)$
denotes the $L$-function of the twist of $f$ by the Kronecker
character associated to $\Q(\sqrt d)$ over $\Q$.

\begin{theorem}[\cite{Zhang}] Under the setting above, we have $L_d(f,1)=0$
  and there exists an explicit nonzero constant $C_f$ depending only
  on $f$ such that
  $$
    \gen{z_f,z_f}=C_f\sqrt{|d|}L_d'(f,1)
  $$
  for all fundamental discriminants $d$ with $(d,DN)=1$.
\end{theorem}

For our purpose, we will consider Shimura curves of the form
$X_0^D(N)/W_{D,N}$.
Let $X=X_0^D(N)/W_{D,N}$ with $D\neq 1$. Define $\xi$ as in
\eqref{equation: canonical divisor}, but with the formula for
$\mathrm{Vol}(X)$ replaced by
$$
  \mathrm{Vol}(X)=\frac{DN}{6|W_{D,N}|}\prod_{p|D}
  \left(1-\frac1p\right)\prod_{p|N}\left(1+\frac1p\right).
$$
Define
$$
  z=\sum_{\sigma: H_d\hookrightarrow\C} x^\sigma
$$
and $z_f$ analogously as in \eqref{equation: CM divisor}. Note that
this time we have $z\in\Jac(X)(\Q)\otimes\Q$. Also, for a newform $f$
on $X_0^D(N)/W_{D,N}$, the sign of the functional equation for
$L(f,s)$ is necessarily $-1$. Therefore, Zhang's formula becomes
$$
  \gen{z_f,z_f}=C_f'\sqrt{|d|}L(f,d,1)
$$
for all fundamental discriminants $d$ with $(d,DN)$ for some nonzero
constant $C_f'$. In this section, we will discuss how to explicitly
compute $z_f$ and thus verify numerically Zhang's formula for a given
Shimura curve. We will consider $X_0^{142}(1)/W_{142,1}$ and
$X_0^{302}(1)/W_{302,1}$ as our primary examples.

Throughout the rest of this section, for a Shimura curve $X$ of the
form $X_0^D(N)/W_{D,N}$ and a discriminant $d<0$, we let $h_d$ be
the number of CM-points of discriminant $d$ on $X$. Also, let
$P_d^{\times h_d}$ denote the CM-divisor
$$
  P_d^{\times h_d}=\sum(\text{CM-points of discriminant }d)
$$
of discriminant $d$ (see Definition \ref{definition: CM-divisor}).
If $h_d=1$, we will simply write $P_d^{\times 1}$ as $P_d$. Let
$\Div_\CM(X)$ denote the subgroup of $\Div(X)$ generated
by all CM-divisors and $\Div_\CM^0(X)$ be the
subgroup of $\Div_\CM(X)$ of degree $0$. Set also
$\Pic_\CM(X)=\Div_\CM(X)/P_\CM(X)$ and
$J_\CM(X)=\Div_\CM^0(X)/P_\CM(X)$, where $P_\CM(X)$
denotes the group of principal divisors contained in $\Div_\CM(X)$.
Since a CM-divisor is defined over $\Q$, $J_\CM(X)$ is a subgroup of
$\Jac(X)(\Q)$.




\end{subsection}

\begin{subsection}{Example $X_0^{142}(1)/W_{142,1}$}

\begin{Lemma} 
  Assume that Assumption \ref{assumption: 1} holds for the Shimura
  curve $X=X_0^{142}(1)/W_{142,1}$. Then every divisor in
  $P_\CM(X)$ can be realized as the divisor of a Borcherds form.
  Moreover, the group $J_\CM(X)$ is a free abelian group of rank $1$
  generated by the divisor class of $P_{-3}-P_{-4}$.
\end{Lemma}

\begin{proof}
 The modular curve
  $X_0(284)$ has genus $34$. By  Lemma \ref{lemma: dimension}, when
  $n\ge 49$, we have 
  $$
    \dim M^!_n(284)=n-16.
  $$
  Let
  $$
    t=\frac{\eta(4\tau)^4\eta(142\tau)^2}{\eta(2\tau)^2\eta(284\tau)^4}
     =q^{-35}+2q^{-33}+q^{-31}+\cdots,
  $$
  which is a modular function on $\Gamma_0(284)$ having only a pole at
  $\infty$. If we can find a positive integer $n_0$ such that
  \begin{enumerate}
  \item[(i)] for each integer $k=0,\ldots,34$, there exists a linear
    combination of eta-products in $M^!(284)$ whose order of pole at
    $\infty$ is $n_0-k$, and
  \item[(ii)] the space $M^!_{n_0}(284)$ is spanned by eta-products,
  \end{enumerate}
  then by multiplying by $t$ suitably, we can show that all modular
  forms in $M^!(284)$ are linear combinations of eta-products. Indeed,
  our construction shows that $n_0$ can be as small as $55$. It is too
  complicated to exhibit a basis for $M^!_{55}(284)$ in terms of
  eta-products. Here we only remark that the $17$ gaps of $M^!(284)$ are
  $1,\ldots,12$, $15$, $16$, $18$, $19$ and $20$. Moreover, we find
  that for each non-gap integer $j$, there is a modular form $f_j$ in
  $M^!(284)\cap\Z((q))$ with a pole of order $j$ at $\infty$ and a
  leading coefficient $1$. Let $\psi_j$ be the Borcherds form
  associated to $f_j$. (Strictly speaking, because the vector-valued
  modular form constructed using $f_j$ may not have an integer
  constant term $c_0(0)$, we can only say that a suitable power of
  $\psi_j$ is a Borcherds form. Here we slightly abuse the
  terminology.)

  Now for a discriminant $d$ such that there are CM-points of
  discriminant $d$ on $X$, we let
  $$
    d_0=\begin{cases}
    d/4, &\text{if }d\equiv 0\mod 4,\\
    d, &\text{if }d\equiv 1\mod 4. \end{cases}
  $$
  When $|d_0|\ge 21$, $f_{|d_0|}$ exists and we have
  $$
    e_d\div\psi_{|d_0|}=P_d+\sum_{d'}c_{d'} P_{d'},
  $$
  where $e_d$ is the cardinality of the stabilizer subgroup of
  CM-points of discriminant $d$,  $d'$ runs over discriminants that
  are either odd integers with $|d'|<|d_0|$ or even integers with
  $|d'|<4|d_0|$. Therefore, recursively, we can show that every $P_d$
  is equivalent to a sum of $P_{d'}$, where $d'$ are either odd
  discriminants with $|d'|\le 21$ or even discriminants with $|d'|\le
  84$. Now for the finite list of discriminants with $|d'|\le 21$ for
  the case of odd $d'$ or $|d'|\le 84$ for the case of even $d'$, we
  check case by case that $P_{d'}\sim n_1P_{-4}+n_2P_{-3}$ for some
  integers $n_1$ and $n_2$. For instance, for $d'=-8$, we can
  construct a modular form in $M^!(284)$
  with Fourier expansion
  \begin{equation*}
  \begin{split}
   &-6q^{-71}-2q^{-48}+2q^{-40}-2q^{-36}+2q^{-24}-2q^{-20}+2q^{-16}
    -4q^{-15}-2q^{-12} \\ 
   &\qquad\qquad-2q^{-10}+2q^{-9}+2q^{-8}-6q^{-7}-2q^{-6}+2q^{-5}-2q^{-3}+2q^{-1}+\cdots,
  \end{split}
  \end{equation*}
  whose corresponding Borcherds form has a divisor $P_{-8}+P_{-4}-2P_{-3}$.
  This gives an algorithmic and recursive way to reduce a CM-divisor
  to a linear sum of $P_{-4}$ and $P_{-3}$. In particular, we find
  that $J_\CM(X)$ is generated by $P_{-4}-P_{-3}$.
  
  To see that $J_\CM(X)$ has rank $1$, i.e., $P_{-4}-P_{-3}$ is not a
  torsion, we recall that in Example \ref{example: 142}, we find that
  an equation for $X$ is $E_{\text{142A1}}:y^2+xy+y=x^3-x^2-12x+15$. Also, if we let
  $Q=(1,1)$ be the generator of $E_{\text{142A1}}(\Q)$, then
  $P_{-3}$ and $P_{-4}$ are $-Q$ and $-2Q$, respectively. As $n(-Q-2Q)$
  is not equal to $O$ for any nonzero integer $n$, there cannot exists a
  modular function on $X$ with a divisor equal to $n(P_{-3}-P_{-4})$ for some
  nonzero integer $n$. (Recall that a divisor $\sum n_i(P_i)$ of an
  elliptic curve is a principal divisor if and only if $\sum
  n_iP_i=O_E$.) This shows that our construction of Borcherds
  forms does generate all elements in $P_\CM(X)$ and that $J_\CM(X)$
  is a free abelian group of rank $1$ generated by the class of
  $P_{-3}-P_{-4}$.
\end{proof}

Of course, the fact that $J_\CM(X)$ is a free abelian
group of rank $1$ generated by some CM-divisor already follows from
Theorems A and C in \cite{Zhang}. Our approach gives an explicit way to
compute the heights of CM-divisors on $J(X)$. For instance,
for a discriminant $d<0$, we let $n$ be the integer such that
$P_d^{\times h_d}-h_dP_{-4}\sim n(P_{-3}-P_{-4})$, we have the following data.

$$ \extrarowheight3pt
\begin{array}{c|cccccccccc} \hline\hline
 d & -3 & -4 & -8 & -19 & -20 & -24 & -27 & -36 & -40 & -43 \\ \hline
h_d&  1 &  1 &  1 &  1  &  1  &  1  &  1  &  1  &  1  &  1 \\
 n & 1  & 0  & 2  &  3  &  4  &  5  &  8  &  7  &  6  & -1 \\
 \hline\hline
 d & -72 & -75 & -83 & -91 & -100 & -107 & -116 & -120 & -131 & -147
 \\ \hline
h_d&  1  &  2  &  3  &  2  &   1  &   3  &   3  &   2  &   5  &  2 \\ 
 n &  9  & 12  &  14 &  9  &  10  &  13  &   12  &   6  &  21  &  13 \\
 \hline\hline
 d & -148 & -152 & -171 & -179 & -180 & -187 & -196 & -200 & -216 & -219  
 \\ \hline
h_d&  1   &   3  &   4  &   5  &   2  &   2  &   2  &   3  &   3  &  4 \\
 n &  -2  &  11   &  19  &  19  &   8  &   2  &  11 &  16  &   8  &  15
 \\ \hline\hline
 d & -228 & -232 & -243 & -251 & -267 & -292 & -296 & -299 & \cdots
   & -568 \\ \hline
h_d&   2  &   1  &   3  &   7  &   2  &   2  &   5  &   8  & \cdots &
2 \\
 n &   5  &  -4  &   3  &  27  &   3  &  15  &  22  &  33  & \cdots & 20
 \\ \hline\hline
\end{array}
$$
Now recall that for a prime $p\neq 2,71$, the action of $T_p$ on
$\Div_\CM(X)$ is given by
$$
  T_p:\frac1{e_d}P_d\longmapsto\frac1{e_{p^2d}}P_{p^2d}+\frac1{e_d}
  \left(1+\JS dp\right)P_d,
$$
where $e_d$ is the cardinality of the stabilizer subgroup of a
CM-points of discriminant $d$ on $X$. We find $e_{-3}=3$, $e_{-4}=4$,
and
\begin{equation*}
\begin{split}
  T_3(P_{-3}-P_{-4})&=3P_{-27}+P_{-3}-4P_{-36}\sim-3(P_{-3}-P_{-4}), \\
  T_5(P_{-3}-P_{-4})&=3P_{-75}^{\times 2}-4P_{-100}^{\times 1}
   -2P_{-4}^{\times 1}\sim-4(P_{-3}-P_{-4}), \\
  T_7(P_{-3}-P_{-4})&=3P_{-147}^{\times 2}+2P_{-3}^{\times 1}
   -4P_{-196}^{\times 2}\sim-3(P_{-3}-P_{-4}).
\end{split}
\end{equation*}
We see that the eigenvalues for $T_3$, $T_5$, and $T_7$ are $-3$,
$-4$, and $-3$, respectively. This agrees with the Cremona's list of
Hecke eigenvalues for the cusp form corresponding to
$E_{\text{142A1}}$.

Now we compute heights of CM-divisors for several discriminants. We
first determine the canonical divisor class $\xi$ defined in
\eqref{equation: canonical divisor}. Since $X$ has genus
$1$, the class $[\Omega_X^1]$ is trivial. It has one elliptic point
$\tau_{-3}$ of order $3$, one elliptic point $\tau_{-4}$ of order $4$, and
three elliptic points $\tau_{-8}$, $\tau_{-568}$, and $\tau_{-568}'$ of
order $2$. Its volume is
$$
  \frac{142}{24}\left(1-\frac12\right)\left(1-\frac1{71}\right)
 =\frac{35}{12}.
$$
Thus,
$$
  \xi=\frac1{35}\left(8P_{-3}^{\times 1}+9P_{-4}^{\times 1}
  +6P_{-8}^{\times 1}+6P_{-568}^{\times 2}\right).
$$
According to the table above,
$$
  35\xi\sim 140P_{-3}-105P_{-4}.
$$
In other words, up to a $35$-torsion in
$J_\CM(X)\otimes\Q$, for a CM-divisor $P_d^{\times h_d}$,
we have
$$
  P_d^{\times h_d}-h_d\xi\sim(P_d^{\times h_d}-h_dP_{-4})-4h_d(P_{-3}-P_{-4}).
$$
Let $n$ be the integer such that the divisor on the right-hand side is
equal to $n(P_{-3}-P_{-4})$. The table above gives us
$$ \extrarowheight3pt
\begin{array}{c|ccccccc} \hline\hline
 d & -3 & -19 & -43 & -83 & -91 & -107 & -131 \\ \hline
 n & -3 &  -1 &  -5 &  2  &  1  &   1  &   1  \\ \hline\hline
 d & -179 & -187 & -219 & -251 & -267 & -299 & \\ \hline
 n &  -1  &  -6  &  -1  &  -1  &  -5  &   1  & \\ \hline\hline
\end{array}
$$
If we numerically approximate $L(E_{\mathrm{142A1}},d,1)$, we find
that
$$
  e_d^2\sqrt{|d|}L(E_{\mathrm{142A1}},d,1)\approx
  n^2\times 2.619470376720\ldots
$$
for the discriminants listed above, agreeing with Zhang's formula.
(The evaluation of the $L$-values is done by using Magma \cite{Magma}.)
\end{subsection}

\begin{subsection}{Example $X_0^{302}(1)/W_{302,1}$}
We next consider $X=X_0^{302}(1)/W_{302,1}$. We can show that admissible
eta-products span $M^!(604)$ as before. Also, assuming that Assumption
\ref{assumption: 1} holds for $X$, we find that every divisor in
$P_\CM(X)$ is the divisor of some Borcherds forms and that
$$
  J_\CM(X)\simeq(\Z/5)\times\Z\times\Z,
$$
where the torsion subgroup is generated by the divisor class of
$P_{-19}-P_{-20}$, and the free abelian subgroup is generated by the
divisor classes of $P_{-4}-P_{-20}$ and $P_{-11}-P_{-20}$. For those
who are interested, we remark that the Borcherds form with a divisor
$5(P_{-19}-P_{-20})$ is coming from the modular form in $M^!(604)$
with a Fourier expansion
\begin{equation*}
\begin{split}
  &-2q^{-144} - 3q^{-128} - 4q^{-124} - 4q^{-100} + 3q^{-84} - 2q^{-80} + 
    6q^{-76} - 6q^{-72} + 2q^{-68} \\
  &\qquad - 3q^{-64} + q^{-55} + 26q^{-52} - 
    4q^{-47} - q^{-44} + 40q^{-41} + 4q^{-40} + 6q^{-39} + q^{-36} \\
  &\qquad + 
    5q^{-32} + 6q^{-31} + 19q^{-30} + 12q^{-28} + 19q^{-27} - 12q^{-26} + 
    4q^{-25} + 28q^{-24} \\
  &\qquad + 5q^{-23} - 3q^{-21} + 3q^{-20} - q^{-19} + 
    6q^{-18} - 2q^{-17} + 5q^{-16} - 26q^{-15} \\
  &\qquad + 7q^{-14} + 14q^{-13} + 
    q^{-11} - 4q^{-10} + q^{-9} + 4q^{-8} + 12q^{-7} + 19q^{-6} -
    6q^{-5} \\
  &\qquad +
    3q^{-4} + 28q^{-3} - 6q^{-2} - 5q^{-1} + \cdots.
\end{split}
\end{equation*}

Let $k,m,n$
be the integers, $0\le k\le 4$, such that $P_d^{\times
  h_d}-h_dP_{-20}\sim
k(P_{-19}-P_{-20})+m(P_{-4}-P_{-20})+n(P_{-11}-P_{-20})$.
We have the following data.
$$ \extrarowheight3pt
\begin{array}{c|cccccc} \hline\hline
 d & -4 & -8 & -11 & -19 & -20 & -36 \\ \hline
h_d&  1 &  1 &  1  &  1  &  1  &  1  \\
k,m,n&0,1,0&3,1,2&0,0,1&1,0,0&0,0,0&2,0,2
\\ \hline\hline
 d & -40 & -43 & -59 & -68 & -72 & -84 \\ \hline
h_d&  1  &  1  &  3  &  2  &  1  &  2  \\
k,m,n&2,0,1&1,1,3&1,0,1&4,1,3&1,-1,-1&0,1,2
\\ \hline\hline
 d & -88 & -91 & -99 & -100 & -116 & -123 \\ \hline
h_d&  1  &  2  &  2  &   1  &   3  &   2  \\
k,m,n&1,0,2&4,0,-1&3,1,0&4,-1,0&3,1,3&1,2,5
\\ \hline\hline
 d & -136 & -139 & -148 & -152 & -155 & -168 \\ \hline
h_d&   2  &   3  &   1  &   3  &   4  &   2  \\
k,m,n&3,0,-1&4,1,5&2,-1,-3&3,0,1&0,1,4&3,1,1
\\ \hline\hline
 d & -171 & -180 & -187 & -195 & \cdots & -1208 \\ \hline
h_d&   4  &   2  &   2  &   4  & \cdots &   6  \\ 
k,m,n&3,1,5&4,1,5&0,-1,1&2,0,3&\cdots&4,0,4
\\ \hline\hline
\end{array}
$$
From the data, we find
\begin{equation*}
\begin{split}
  T_3(P_{-19}-P_{-20})&=P_{-171}^{\times 4}
    -P_{-180}^{\times 2}-2P_{-20}^{\times 1}\sim-(P_{-19}-P_{-20}), \\
  T_3(P_{-4}-P_{-20})&=4P_{-36}^{\times 1}
    -P_{-180}^{\times 2}-2P_{-20}^{\times 1} \\
  &\sim 4(P_{-19}-P_{20})-(P_{-4}-P_{-20})+3(P_{-11}-P_{-20}), \\
  T_3(P_{-11}-P_{-20})&=P_{-99}^{\times 2}
    +2P_{-11}^{\times 1}-P_{-180}^{\times 2}-2P_{-20}^{\times 1} \\
  &\sim 4(P_{-19}-P_{-20})-3(P_{-11}-P_{-20}).
\end{split}
\end{equation*}
Let $f_A$ and $f_C$ be the eigenforms on $X$ corresponding
to the newforms on $X_0(302)$ associated to the elliptic curves
$E_{\mathrm{302A1}}$ and $E_{\mathrm{302C1}}$ through the
Jacquet-Langlands correspondence. Let $a_{f_A}(p)$ and $a_{f_C}(p)$
denote the eigenvalues of $T_p$ for $f_A$ and $f_C$, respectively.
According to Cremona's table, we have $a_{f_A}(3)=-1$ and
$a_{f_C}(3)=-3$. Let $J_{f_A,\CM}$ and
$J_{f_C,\CM}$ be the maximal subgroups of $J_\CM(X)$ annihilated by
all $T_p-a_{f_A}(p)$ and all $T_p-a_{f_C}(p)$, respectively. We find
$J_{f_A,\CM}$ is generated by the divisor classes of
$$
  D_{A}=2P_{-4}+3P_{-11}-5P_{-20}, \qquad
  D_{A,0}=P_{-19}-P_{-20},
$$
where $D_{A,0}$ is a $5$-torsion, and $J_{f_C,\CM}$ is spanned by
$$
  D_C=P_{-11}-2P_{-19}+P_{-20}.
$$

We now determine the canonical divisor in Zhang's formula. In Example
\ref{example: 302}, we find that the unique genus-zero subfield of
degree $2$ in the function field of $X$ is generated by a function $x$
with divisor $\div x=P_{-43}+P_{-72}-P_{-19}-P_{-88}$. If we take a
standard Weierstrass equation $y^2=f(x)$ for $X$, then $\div
x/dy=P_{-19}+P_{-88}$. Thus, a representative for the class
$[\Omega_X^1]$ is $P_{-19}+P_{-88}$. The curve $X$ has one elliptic
point $\tau_{-4}$ of order $4$, and $7$ elliptic points $\tau_{-8}$ and
$\tau_{-1208}^{(1)},\ldots,\tau_{-1208}^{(6)}$ of order $2$. The volume is
$$
  \frac{302}{24}\left(1-\frac12\right)\left(1-\frac1{151}\right)=\frac{25}4.
$$
Therefore, the canonical divisor $\xi$ in Zhang's formula is
$$
  \xi=\frac{1}{25}\left(4P_{-19}+4P_{-88}+3P_{-4}+2P_{-8}
    +2P_{-1208}^{\times 6}\right).
$$
Up to a $25$-torsion in $J_\CM(X)\otimes\Q$, we have
$$
  \xi-P_{-20}\sim\frac1{25}\left(5P_{-4}+20P_{-11}+2P_{-19}-27P_{-20}\right)
 =\frac1{10}D_A+\frac{27}{25}D_{A,0}+\frac12D_C.
$$
Now for a discriminant $d$, let $k,m,n$ be the rational numbers such
that
$$
  P_d^{\times h_d}-h_dP_{-20}-h_d\left(\frac1{10}D_A+\frac{27}{25}D_{A,0}
  +\frac12D_C\right)=kD_{A,0}+\frac m{10}D_A+\frac n2D_C.
$$
We have
$$ \small\extrarowheight3pt
\begin{array}{c|cccccccccc} \hline\hline
 d & -11 & -19 & -43 & -59 & -91 & -123 & -139 & -155 & -187 & -195 \\
 \hline
m,n&-1,1 &-1,-1& 4,2 &-3,-1&-2,-4&  8,2 &  2,4 &  1,1 & -7,3 & -4,2
 \\ \hline\hline
\end{array}
$$
Approximating $L(E_{\mathrm{302A1}},d,1)$ and
$L(E_{\mathrm{302C1}},d,1)$ numerically, we find that for the
discriminants above, we have
$$
  \sqrt{|d|}L(E_{\mathrm{302A1}},d,1)\approx m^2\times
  1.225637922269744563\ldots,
$$
and
$$
  \sqrt{|d|}L(E_{\mathrm{302C1}},d,1)\approx n^2\times
  4.02890102461114010\ldots.
$$
\end{subsection}

\begin{subsection}{Example $X_0^{334}(1)/W_{334,1}$}
Consider $X=X_0^{334}(1)/W_{334,1}$, which has genus $2$. Let $g_1$ and $g_2$ be
eigenforms on $X$. The eigenforms on $\Gamma_0(334)$ corresponding to
$g_1$ and $g_2$ under the Jacquet-Langlands correspondence are
$$
  \widetilde g_1=q+q^2+\frac{-3+\sqrt5}2q^3+q^4+(-2-\sqrt5)q^5
     +\frac{-3+\sqrt5}2q^6-3q^7+q^8+\cdots
$$
and its Galois conjugate. We will determine the equation of $X$ and
compute the heights of several CM-divisors.

We can find modular forms $f_1,\ldots,f_4$ in $M^!(668)$ such that
their associated Borcherds forms $\psi_i$ have divisors
\begin{equation*}
\begin{split}
  \div\psi_{1}&=P_{-19}+P_{-36}-P_{-8}-P_{-27}, \\
  \div\psi_{2}&=P_{-88}+P_{-100}-P_{-8}-P_{-27}, \\
  \div\psi_{3}&=P_{-19}+P_{-27}+P_{-232}-P_{-11}-P_{-72}-P_{-100}, \\
  \div\psi_{4}&=P_{-484}^{\times 3}-P_{-11}-P_{-72}-P_{-100}.
\end{split}
\end{equation*}
Let $x$ and $y$ be the modular functions on $X$ such that
$\div x=\div\psi_{1}$, $x(\tau_{-88})=1$, $\div y=\div\psi_{3}$, and
$y(\tau_{-3})=3$. We find
$$ \small\extrarowheight3pt
\begin{array}{c|cccccccccccc} \hline\hline
 d & -3 & -4 & -8 & -11& -19 & -24& -27& -36& -72& -88&-100&-232 \\ \hline\hline
 x & 2/3&  6&\infty& 2 &  0  & 2 &\infty& 0 & 2/3&  1 & 1  &  6 \\
 y & 3 &-1/10& 2/3&\infty& 0 & -2 & 0  & 2 &\infty&10/3&\infty& 0 \\
\hline\hline
\end{array}
$$
and deduce that the relation between $x$ and $y$ is
$$
  (x-2)(3x-2)(x-1)y^2+(-2x^3+17x^2-26x+8)y-2x(x-6)=0.
$$
Setting
$$
  X=x, \qquad Y=2y(3x^3-11x^2+12x-4)-2x^3+17x^2-26x+8,
$$
we get the Weierstrass equation
$$
  Y^2=4X^6-44X^5+161X^4-292X^3+340X^2-224X+64.
$$

Furthermore, we find that $J_\CM(X)$ is a free abelian group of rank
$2$ generated by the divisor classes of
$$
  D_1=P_{-8}-P_{-3}, \qquad D_2=P_{-11}-P_{-3}.
$$
For a CM-divisor $P_d^{\times h_d}$, let $m$ and $n$ be the integers
such that $P_d^{\times h_d}-h_dP_{-3}\sim mD_1+nD_2$. We have
$$ \extrarowheight3pt
\begin{array}{c|ccccccccc} \hline\hline
 d & -3 & -4 & -8 & -11 & -19 & -24 & -27 & -36 & -56 \\ \hline
h_d& 1  &  1 &  1 &  1  &  1  &  1  &  1  &  1  &  2  \\
m,n& 0,0&15,-4&1,0& 0,1 &5,-1 & 3,0 & 2,1 &-2,2 &10,-1\\
\hline\hline
 d &-72 &-75 &-84 &-88 & -99 &-100 &-107 &-115 &-116 \\ \hline
h_d& 1  & 2 &  2 &  1 &  2  &  1  &  3  &  2  &  3   \\
m,n&3,1 &15,-2&9,-1&11,-2&17,-4&-8,3&0,2 &14,-3&10,0 \\ \hline\hline  
 d &-132&-147&-152&-168&-171&-179&-195&-196&-200 \\ \hline
h_d&  2 &  2 &  3 &  2 &  4 &  5 &  4 &  2 &  3  \\
m,n&3,1 &11,0&7,0 &-4,3&5,3 &10,2&5,2 &-3,3&18,-2\\ \hline\hline
 d &-203&-211&-216&-228&-232&-243&-244&\cdots&-1336 \\ \hline
h_d&  4 &  3 &  3 &  2 &  1 &  3 &  3 &\cdots&  6  \\
m,n&8,1 &2,2 &9,0 &12,-1&-12,5&21,-4&8,1&\cdots&12,4 \\ \hline\hline
\end{array}
$$
The action of the Hecke operator $T_3$ is
\begin{equation*}
\begin{split}
  T_3D_1&=2P_{-72}+2P_{-8}-3P_{-27}-P_{-3}\sim 2D_1-D_2, \\
  T_3D_2&=P_{-99}^{\times 2}+2P_{-11}-3P_{-27}-P_{-3}\sim
    11D_1-5D_2.
\end{split}
\end{equation*}
Thus, the $g_1$-isotypical and $g_2$-isotypical components of
$J_\CM(X)\otimes\R$ are spanned by
$$
  \frac{7+\sqrt 5}2D_1-D_2, \qquad \frac{7-\sqrt 5}2D_1-D_2,
$$
respectively.

We now verify Zhang's formula for the case of $X$.
The canonical divisor class in Zhang's formula is the class of
$$
  \xi=\frac{12}{83}\left(P_{-8}+P_{-27}+\frac23P_{-3}+\frac34P_{-4}
  +\frac12P_{-8}+\frac12P_{-1336}^{\times 6}\right).
$$
We find that, up to a $83$-torsion in $J_\CM(X)\otimes\Q$,
$$
  \xi-P_{-3}\sim 3P_{-8}-3P_{-3}=3D_1.
$$
Let
$$
  E_1=\frac1{\sqrt5}\left(\frac{7+\sqrt5}2D_1-D_2\right), \qquad
  E_2=-\frac1{\sqrt5}\left(\frac{7-\sqrt5}2D_1-D_2\right),
$$
and $s$ and $t$ be the real numbers in $\Q(\sqrt5)$ such that
$$
  P_d^{\times h_d}-h_dP_{-3}-3h_dD_1\sim sE_1+tE_2.
$$
We find
$$ \extrarowheight3pt
\begin{array}{c|cccccc} \hline\hline
 d & -3 & -11 & -19 & -107 & -115 \\ \hline
 s & -3 & (1-\sqrt5)/2 & (-3+\sqrt5)/2 & -2-\sqrt 5 & (-5+3\sqrt5)/2 
\\ \hline\hline
 d & -179 & -195 & -203 & -211 & -251 \\ \hline
 s & 2-\sqrt5 & -\sqrt5 & -(1+\sqrt5)/2 & -\sqrt5 & 2 \\ \hline\hline
\end{array}
$$
Approximating $L(g_1,d,1)$ numerically, we find that
$$
  e_d^2\sqrt{|d|}L(g_1,d,1)\approx s^2\times 0.2909633434\ldots.
$$
\end{subsection}
\end{section}

\begin{appendices}
\section{Tables for equations of hyperelliptic Shimura curves }
We list defining equations of hyperelliptic Shimura curves below. 
\renewcommand{\arraystretch}{1.5}
\tabcolsep=6pt                         
\begin{table}[htb]
\newcommand{\tabincell}[2]{\begin{tabular}{@{}#1@{}}#2\end{tabular}}
  \centering
\begin{tabular}{|c|c|c|}
 \hline
 \multirow{2}{*}{$X^{26}_0(1)$} &
 
 \multicolumn{2}{c|}{\multirow{1}{*}{\text}{ \tabincell{c}{ $y^2=-2x^6 + 19x^4 - 24x^2 - 169$ } }} \\
 \cline{2-3}
    & \multicolumn{2}{c|}{\tabincell{c}{$w_2(x,y)=(-x,-y)$, \\$w_{26}(x,y)=(x,-y)$.}}

\\\hline
 \multirow{2}{*}{$X^{35}_0(1)$} &
 
 \multicolumn{2}{c|}{\multirow{1}{*}{\text}{ \tabincell{c}{ $y^2=-(x^2+7)(7x^6+51x^4+197x^2+1)$ } }} \\
 \cline{2-3}
    & \multicolumn{2}{c|}{\tabincell{c}{$w_5(x,y)=(-x,-y)$, \\$w_{35}(x,y)=(x,-y)$.}}

\\\hline
 \multirow{2}{*}{$X^{38}_0(1)$} &
 
 \multicolumn{2}{c|}{\multirow{1}{*}{\text}{ \tabincell{c}{ $y^2=-16x^6 - 59x^4 - 82x^2 - 19$ } }} \\
 \cline{2-3}
    & \multicolumn{2}{c|}{\tabincell{c}{$w_2(x,y)=(-x,-y)$, \\$w_{38}(x,y)=(x,-y)$.}} 
\\\hline
\multirow{2}{*}{$X^{39}_0(1)$} &
 
 \multicolumn{2}{c|}{\multirow{1}{*}{\text}{ \tabincell{c}{ $y^2=-(x^4-x^3-x^2+x+1)(7x^4-23x^3+5x^2+23x+7)$ } }} \\
 \cline{2-3}
    & \multicolumn{2}{c|}{\tabincell{c}{$w_{13}(x,y)=(-\frac{1}{x},\frac{y}{x^4})$, \\$w_{39}(x,y)=(x,-y)$.}}

\\\hline
 \multirow{2}{*}{$X^{51}_0(1)$} &
 
 \multicolumn{2}{c|}{\multirow{1}{*}{\text}{ \tabincell{c}{ $y^2=-(x^2+3)(243x^6+235x^4-31x^2+1)$ } }} \\
 \cline{2-3}
    & \multicolumn{2}{c|}{\tabincell{c}{$w_3(x,y)=(-x,y)$, \\$w_{51}(x,y)=(x,-y)$.}} 

\\\hline
 \multirow{2}{*}{$X^{55}_0(1)$} &
 
 \multicolumn{2}{c|}{\multirow{1}{*}{\text}{ \tabincell{c}{ $y^2=-(x^4-x^3+x^2+x+1)(3x^4+x^3-5x^2-x+3)$ } }} \\
 \cline{2-3}
    & \multicolumn{2}{c|}{\tabincell{c}{$w_5(x,y)=(-\frac{1}{x},\frac{y}{x^4})$, \\$w_{55}(x,y)=(x,-y)$.}} 

\\\hline
\multirow{2}{*}{$X^{57}_0(1)$} &
 
 \multicolumn{2}{c|}{\multirow{1}{*}{\text}{ \tabincell{c}{ $y^2=(3s+1)(3s^3+11s^2+17s+1)$\\
 $x^2=-4s^2+2s-1$ } }} \\
 \cline{2-3}
    & \multicolumn{2}{c|}{\tabincell{c}{$w_{19}(s,x,y)=(s,x,-y)$, \\$w_{57}(s,x,y)=(s,-x,y)$.}} 

\\\hline

 \multirow{2}{*}{$X^{58}_0(1)$} &
 
 \multicolumn{2}{c|}{\multirow{1}{*}{\text}{ \tabincell{c}{ $y^2=-2x^6 - 78x^4 - 862x^2 - 1682$  } }} \\
 \cline{2-3}
    & \multicolumn{2}{c|}{\tabincell{c}{$w_2(x,y)=(-x,-y)$, \\$w_{29}(x,y)=(x,-y)$.}}

\\\hline

\end{tabular}
\bigskip
\caption{Equations of level one}
\end{table}

\renewcommand{\arraystretch}{1.5}
\tabcolsep=6pt                         
\begin{table}[htb]
\newcommand{\tabincell}[2]{\begin{tabular}{@{}#1@{}}#2\end{tabular}}
  \centering
\begin{tabular}{|c|c|c|}
 \hline
 \multirow{2}{*}{$X^{62}_0(1)$} &
 
 \multicolumn{2}{c|}{\multirow{1}{*}{\text}{ \tabincell{c}{ $y^2=-64x^8 - 99x^6 - 90x^4 - 43x^2 - 8$  } }} \\
 \cline{2-3}
    & \multicolumn{2}{c|}{\tabincell{c}{$w_2(x,y)=(-x,y)$, \\ $w_{62}(x,y)=(x,-y)$.}} 

\\\hline
 \multirow{2}{*}{$X^{69}_0(1)$} &
 
 \multicolumn{2}{c|}{\multirow{1}{*}{\text}{ \tabincell{c}{ $y^2=-243x^8+1268x^6-666x^4-2268x^2-2187$ } }} \\
 \cline{2-3}
    & \multicolumn{2}{c|}{\tabincell{c}{$w_3(x,y)=(-x,y)$, \\$w_{69}(x,y)=(x,-y)$.}} 
\\\hline
 \multirow{2}{*}{$X^{74}_0(1)$} &
 
 \multicolumn{2}{c|}{\multirow{1}{*}{\text}{ \tabincell{c}{ $y^2=-2x^{10} + 47x^8 - 328x^6 + 946x^4 - 4158x^2$\\ $- 1369$   } }} \\
 \cline{2-3}
    & \multicolumn{2}{c|}{\tabincell{c}{$w_2(x,y)=(-x,-y)$, \\$w_{74}(x,y)=(x,-y)$.}} 

\\\hline

\multirow{2}{*}{$X^{82}_0(1)$} &
 
 \multicolumn{2}{c|}{\multirow{1}{*}{\text}{ \tabincell{c}{ $y^2=4s^4 + 4s^3 + s^2 - 2s + 1$\\ $x^2=-19s^2 + 18s - 11$ } }} \\
 \cline{2-3}
    & \multicolumn{2}{c|}{\tabincell{c}{$w_2(x,y)=(-x,-y)$, \\$w_{41}(x,y)=(x,-y)$.}}

\\\hline
 \multirow{2}{*}{$X^{86}_0(1)$} &
 
 \multicolumn{2}{c|}{\multirow{1}{*}{\text}{ \tabincell{c}{ $y^2=-16x^{10} + 245x^8 - 756x^6 - 1506x^4 - 740x^2 - 43$   } }} \\
 \cline{2-3}
    & \multicolumn{2}{c|}{\tabincell{c}{$w_2(x,y)=(-x,-y)$, \\$w_{86}(x,y)=(x,-y)$.}} 

\\\hline
 \multirow{2}{*}{$X^{87}_0(1)$} &
 
 \multicolumn{2}{c|}{\multirow{1}{*}{\text}{ \tabincell{c}{ $y^2=-(x^6-7x^4+43x^2+27)(243x^6+523x^4+369x^2+81)$ } }} \\
 \cline{2-3}
    & \multicolumn{2}{c|}{\tabincell{c}{$w_3(x,y)=(-x,y)$, \\$w_{87}(x,y)=(x,-y)$.}} 

\\\hline
 \multirow{2}{*}{$X^{93}_0(1)$} &
 
 \multicolumn{2}{c|}{\multirow{1}{*}{\text}{ \tabincell{c}{ $y^2=(3s^3-7s^2-3t-1)(3s^3+s^2-3s-9)$\\ $x^2=-4s^2-6s-9$ } }} \\
 \cline{2-3}
    & \multicolumn{2}{c|}{\tabincell{c}{$w_3(s,x,y)=(s,-x,-y)$, \\$w_{31}(s,x,y)=(s,x,-y)$.}}

\\\hline

\multirow{2}{*}{$X^{94}_0(1)$} &
 
 \multicolumn{2}{c|}{\multirow{1}{*}{\text}{ \tabincell{c}{ $y^2=-8x^8 + 69x^6 - 234x^4 + 381x^2 - 256$  } }} \\
 \cline{2-3}
    & \multicolumn{2}{c|}{\tabincell{c}{$w_2(x,y)=(-x,y)$, \\$w_{94}(x,y)=(x,-y)$.}} 

\\\hline

\end{tabular}
\bigskip
\caption{Equations of level one}
\end{table}

\renewcommand{\arraystretch}{1.5}
\tabcolsep=6pt                         
\begin{table}[htb]
\newcommand{\tabincell}[2]{\begin{tabular}{@{}#1@{}}#2\end{tabular}}
  \centering
\begin{tabular}{|c|c|c|}
 \hline
 \multirow{2}{*}{$X^{95}_0(1)$} &
 
 \multicolumn{2}{c|}{\multirow{1}{*}{\text}{ \tabincell{c}{ $y^2=-(x^8+x^7-x^6-4x^5+x^4+4x^3-x^2-x+1)$\\$\times(7x^8+19x^7+21x^6-13x^4+21x^2-19x+7)$ } }} \\
 \cline{2-3}
    & \multicolumn{2}{c|}{\tabincell{c}{$w_5(x,y)=(-\frac{1}{x},\frac{y}{x^6})$, \\$w_{95}(x,y)=(x,-y)$.}} 

\\\hline

\multirow{2}{*}{$X^{111}_0(1)$} &
 
 \multicolumn{2}{c|}{\multirow{1}{*}{\text}{ \tabincell{c}{ $y^2=-(19x^8-44x^7-16x^6+55x^5+37x^4-55x^3-16x^2+44x+19)$\\ $\times (x^8-3x^5-x^4+3x^3+1)$ } }} \\
 \cline{2-3}
    & \multicolumn{2}{c|}{\tabincell{c}{$w_{37}(x,y)=(-\frac{1}{x},\frac{y}{x^8})$, \\$w_{111}(x,y)=(x,-y)$.}} 

\\\hline
\multirow{2}{*}{$X^{119}_0(1)$} &
 
 \multicolumn{2}{c|}{\multirow{1}{*}{\text}{ \tabincell{c}{$y^2=-(7x^{10}-171x^8+758x^6+3418x^4+4851x^2+2401)$\\ $\times (x^{10}+3x^{8}+26x^6+278x^4+373x^2+343)$ } }} \\
 \cline{2-3}
    & \multicolumn{2}{c|}{\tabincell{c}{$w_{7}(x,y)=(-x,y)$, \\$w_{119}(x,y)=(x,-y)$.}} 

\\\hline

 \multirow{2}{*}{$X^{134}_0(1)$} &
 
 \multicolumn{2}{c|}{\multirow{1}{*}{\text}{ \tabincell{c}{ $y^2=
-19x^{14} - 8x^{13} + 178x^{12} - 138x^{11} - 625x^{10} + 940x^9 + 383x^8   $\\$- 1486x^7+ 383x^6 + 940x^5 - 625x^4 - 138x^3 + 178x^2$\\$ - 8x - 19$  } }} \\
 \cline{2-3}
    & \multicolumn{2}{c|}{\tabincell{c}{$w_2(x,y)=(\frac{1}{x},\frac{y}{x^7})$, \\$w_{134}(x,y)=(x,-y)$.}}

 \\\hline
 \multirow{2}{*}{$X^{146}_0(1)$} &
 
 \multicolumn{2}{c|}{\multirow{1}{*}{\text}{ \tabincell{c}{ $y^2=-11x^{16} + 82x^{15} - 221x^{14} + 214x^{13} + 133x^{12} - 360x^{11} - 170x^{10}$\\$ + 
        676x^9 - 150x^8 - 676x^7 - 170x^6 + 360x^5+ 133x^4$\\$  - 214x^3 - 221x^2 - 82x - 11$ } }} \\
 \cline{2-3}
    & \multicolumn{2}{c|}{\tabincell{c}{$w_{73}(x,y)=(-\frac{1}{x},\frac{y}{x^8})$, \\$w_{146}(x,y)=(x,-y)$.}} 
 \\\hline

\multirow{2}{*}{$X^{159}_0(1)$} &
 
 \multicolumn{2}{c|}{\multirow{1}{*}{\text}{ \tabincell{c}{ $y^2=-(81x^{10}+207x^8+874x^6-130x^4-11x^2+3)$
 \\$\times(2187x^{10}+8389x^8+8878x^6+42x^4-41x^2+1)$ } }} \\
 \cline{2-3}
    & \multicolumn{2}{c|}{\tabincell{c}{$w_{3}(x,y)=(-x,y)$, \\$w_{159}(x,y)=(x,-y)$.}} 

\\\hline

\end{tabular}
\bigskip
\caption{Equations of level one}
\end{table}

\newpage
\renewcommand{\arraystretch}{1.5}
\tabcolsep=6pt                         
\begin{table}[htb]
\newcommand{\tabincell}[2]{\begin{tabular}{@{}#1@{}}#2\end{tabular}}
  \centering
\begin{tabular}{|c|c|c|}
 \hline

\multirow{2}{*}{$X^{194}_0(1)$} &
 
 \multicolumn{2}{c|}{\multirow{1}{*}{\text}{ \tabincell{c}{ $y^2=-19x^{20} - 92x^{19} - 286x^{18} - 592x^{17} - 921x^{16} - 1016x^{15} - 872x^{14}$\\$ + 
    460x^{13} + 1545x^{12} + 1752x^{11} + 34x^{10} - 1752x^9 + 1545x^8$\\$ - 460x^7 -
    872x^6 + 1016x^5 - 921x^4 + 592x^3 - 286x^2$\\$ + 92x - 19$} }} \\
 \cline{2-3}
    & \multicolumn{2}{c|}{\tabincell{c}{$w_{97}(x,y)=\left(-\frac{1}{x},-\frac{y}{x^{10}}\right)$,\\$w_{194}(x,y)=\left(x,-y\right)$.}} \\

 \hline
 \multirow{2}{*}{$X^{206}_0(1)$} &
 
 \multicolumn{2}{c|}{\multirow{1}{*}{\text}{\tabincell{c}{ $y^2= -8x^{20} + 13x^{18} + 42x^{16} + 331x^{14} + 220x^{12}- 733x^{10}    $\\$ 
 - 6646x^8 - 
        19883x^6 - 28840x^4 - 18224x^2 - 4096  $}}} \\
 \cline{2-3}
    & \multicolumn{2}{c|}{\tabincell{c}{$w_2(x,y)=(-x,y)$, \\$w_{206}(x,y)=(x,-y)$.}} \\
 
 \hline
 
 \end{tabular} 
 \bigskip
 \caption{Equations of level one}
\end{table}

\renewcommand{\arraystretch}{1.5}
\tabcolsep=6pt                         
\begin{table}[htb]
\newcommand{\tabincell}[2]{\begin{tabular}{@{}#1@{}}#2\end{tabular}}
  \centering

\begin{tabular}{|c|c|c|}
 \hline
 \multirow{2}{*}{$X^{6}_0(11)$} &
 
 \multicolumn{2}{c|}{\multirow{1}{*}{\text}{ \tabincell{c}{ $y^2=-19x^8 - 166x^7 - 439x^6 - 166x^5 + 612x^4$\\ $+ 166x^3 - 439x^2 + 166x -
    19$  } }} \\
 \cline{2-3}
    & \multicolumn{2}{c|}{\tabincell{l}{$w_{2}(x,y)=\left(\frac{x+1}{x-1},-\frac{4y}{(x-1)^4}\right)$, \\$w_3(x,y)=(-\frac{1}{x},-\frac{y}{x^4})$,\\ $w_{66}(x,y)=\left(x,-y\right).$ }}

 \\\hline
 \multirow{2}{*}{$X^6_0(17)$} &
 
 \multicolumn{2}{c|}{\multirow{1}{*}{\text}{ \tabincell{c}{ $z^2=-3x^2-16$\\ $y^2=17x^4-10x^2+9$ } }} \\
 \cline{2-3}
    & \multicolumn{2}{c|}{\tabincell{l}{$w_2(x,y,z)=\left(-x,y,z\right),$ \\$w_{3}(x,y,z)=\left(x,-y,-z\right),$\\ $w_{34}(x,y,z)=\left(x,-y,z\right).$ }}

\\\hline
\multirow{2}{*}{$X^{6}_0(19)$} &
 
 \multicolumn{2}{c|}{\multirow{1}{*}{\text}{ \tabincell{c}{ $y^2=-19x^8 + 210x^6 - 625x^4 + 210x^2 - 19$  } }} \\
 \cline{2-3}
    & \multicolumn{2}{c|}{\tabincell{l}{$w_2(x,y)=\left(-\frac{1}{x},-\frac{y}{x^4}\right)$, \\$w_{3}(x,y)=\left(\frac{1}{x},\frac{y}{x^4}\right)$,\\ $w_{114}(x,y)=\left(x,-y\right).$ }}

 \\\hline
 \multirow{2}{*}{$X^6_0(29)$} &
 
 \multicolumn{2}{c|}{\multirow{1}{*}{\text}{ \tabincell{c}{ $y^2=-64x^{12} + 813x^{10} - 3066x^8 + 4597x^6 - 12264x^4  $ \\ $+ 13008x^2 - 4096  $ } }} \\
 \cline{2-3}
    & \multicolumn{2}{c|}{\tabincell{l}{$w_2(x,y)=\left(-x,y\right)$, \\$w_{3}(x,y)=\left(-\frac{2}{x},\frac{8y}{x^6}\right),$ \\ $w_{174}(x,y)=\left(x,-y\right)$ }}

\\\hline
 \multirow{2}{*}{$X^{6}_0(31)$} &
 
 \multicolumn{2}{c|}{\multirow{1}{*}{\text}{ \tabincell{c}{ $y^2=-243x^{12} + 11882x^{10} - 177701x^8 + 803948x^6  $ \\ $ - 1599309x^4 + 962442x^2 - 
    177147  $ } }} \\
 \cline{2-3}
    & \multicolumn{2}{c|}{\tabincell{l}{$w_2(x,y)=\left(\frac{3}{x},-\frac{27y}{x^6}\right)$, \\$w_{3}(x,y)=\left(-x,y\right),$ \\ $w_{186}(x,y)=\left(x,-y\right).$ }}
 
\\\hline

\multirow{2}{*}{$X^{6}_0(37)$} &
 
 \multicolumn{2}{c|}{\multirow{1}{*}{\text}{ \tabincell{c}{ $y^2=-4096x^{12} - 18480x^{10} - 40200x^8 - 51595x^6 $ \\ $ - 40200x^4 - 18480x^2 - 4096 $} }} \\
 \cline{2-3}
    & \multicolumn{2}{c|}{\tabincell{l}{$w_2(x,y)=\left(-x,y\right)$, \\$w_{3}(x,y)=\left(\frac{1}{x},\frac{y}{x^6}\right),$\\ $w_{222}(x,y)=\left(x,-y\right).$\  }}

  \\\hline

 \end{tabular} 
 \bigskip
 \caption{Equations of level greater than one}
\end{table}
\renewcommand{\arraystretch}{1.5}
\tabcolsep=6pt                         
\begin{table}[htb]
\newcommand{\tabincell}[2]{\begin{tabular}{@{}#1@{}}#2\end{tabular}}
  \centering

\begin{tabular}{|c|c|c|}
 \hline
  \multirow{2}{*}{$X^{10}_0(11)$} &
 
 \multicolumn{2}{c|}{\multirow{1}{*}{\text}{ \tabincell{c}{ $y^2=-8x^{12} - 35x^{10} + 30x^8 + 277x^6 + 120x^4   $ \\ $ - 560x^2 - 512$ } }} \\
 \cline{2-3}
    & \multicolumn{2}{c|}{\tabincell{l}{$w_{10}(x,y)=\left(-\frac{2}{x},-\frac{8y}{x^6}\right)$, \\$w_{22}(x,y)=\left(\frac{2}{x},\frac{8y}{x^6}\right)$,\\ $w_{110}(x,y)=\left(x,-y\right).$\  }} 
    \\\hline

  \multirow{2}{*}{$X^{10}_0(13)$} &
 
 \multicolumn{2}{c|}{\multirow{1}{*}{\text}{ \tabincell{c}{ $z^2=-2x^2-25$ \\ $y^2=5x^4-74x^2
 +325$ } }} \\
 \cline{2-3}
    & \multicolumn{2}{c|}{\tabincell{l}{$w_2(x,y,z)=\left(x,-y,-z\right)$, \\$w_{5}(x,y,z)=\left(-x,-y,-z\right),$\\ $w_{65}(x,y,z)=\left(x,-y,z\right)$. }}
 
\\\hline
 \multirow{2}{*}{$X^{10}_0(19)$} &
 
 \multicolumn{2}{c|}{\multirow{1}{*}{\text}{ \tabincell{c}{ $z^2=5x^2-32 $ \\$y^2=-8x^6+57x^4-40x^2+16$} }} \\
 \cline{2-3}
    & \multicolumn{2}{c|}{\tabincell{l}{$w_2(x,y,z)=\left(-x,y,z\right)$, \\$w_{5}(x,y,z)=\left(x,-y,-z\right),$\\ $w_{38}(x,y,z)=\left(x,-y,z\right).$\  }} 
 
 \\\hline
  
  \multirow{2}{*}{$X^{10}_0(23)$} &
 
 \multicolumn{2}{c|}{\multirow{1}{*}{\text}{ \tabincell{c}{ $y^2=-43x^{20} + 318x^{19} - 1071x^{18} + 3014x^{17} - 10540x^{16}  $\\ $+ 28266x^{15} - 
    72217x^{14} + 81478x^{13} - 62765x^{12} - 68732x^{11} $\\ $+ 18840x^{10} + 68732x^9 -
    62765x^8 - 81478x^7 - 72217x^6 $\\ $- 28266x^5 - 10540x^4 - 3014x^3 - 
    1071x^2 - 318x - 43$  } }} \\
 \cline{2-3}
    & \multicolumn{2}{c|}{\tabincell{l}{$w_2(x,y)=\left(\frac{2x+1}{x-2},-\frac{5^5y}{(x-2)^{10}}\right)$, \\$w_{5}(x,y)=\left(-\frac{1}{x},-\frac{y}{x^{10}}\right)$,\\ $w_{230}(x,y)=\left(x,-y\right).$ }}
\\\hline

 \multirow{2}{*}{$X^{14}_0(3)$} &
 
 \multicolumn{2}{c|}{\multirow{1}{*}{\text}{ \tabincell{c}{ $z^2=-9x^2-2$ \\ $ y^2=-7x^4+22x^2+1  $ } }} \\
 \cline{2-3}
    & \multicolumn{2}{c|}{\tabincell{l}{$w_2(x,y,z)=\left(-x,y,z\right)$, \\$w_{3}(x,y,z)=\left(x,-y,-z\right),$\\ $w_{14}(x,y,z)=\left(x,-y,z\right).$\  }} 
 
  \\\hline
 \multirow{2}{*}{$X^{14}_0(5)$} &
 
 \multicolumn{2}{c|}{\multirow{1}{*}{\text}{ \tabincell{c}{ $y^2=-23x^8 - 180x^7 - 358x^6 - 168x^5 - 677x^4    $ \\ $+ 168x^3 - 358x^2 + 180x - 23 $ } }} \\
 \cline{2-3}
    & \multicolumn{2}{c|}{\tabincell{l}{$w_2(x,y)=\left(-\frac{1}{x},\frac{y}{x^4}\right)$, \\$w_{14}(x,y)=\left(x,-y\right),$ \\ $w_{35}(x,y)=\left(\frac{x+2}{2x-1},-\frac{25y}{(2x-1)^4}\right).$ }}

\\\hline

 \end{tabular} 
 \bigskip
 \caption{Equations of level greater than one}
\end{table}

\newpage
\renewcommand{\arraystretch}{1.5}
\tabcolsep=6pt                         
\begin{table}[htb]
\newcommand{\tabincell}[2]{\begin{tabular}{@{}#1@{}}#2\end{tabular}}
  \centering

\begin{tabular}{|c|c|c|}
 \hline

 \multirow{2}{*}{$X^{15}_0(2)$} &
 
 \multicolumn{2}{c|}{\multirow{1}{*}{\text}{ \tabincell{c}{ $y^2=-(x^2+3)(3x^2+4)(x^4-x^2+4)$ } }} \\
 \cline{2-3}
    &
      \multicolumn{2}{c|}{\tabincell{l}{$w_2(x,y)=\left(\frac2x,-\frac{4y}{x^4}\right)$, \\
      $w_{3}(x,y)=\left(-x,y\right),$\\ $w_{5}(x,y)=\left(-x,-y\right).$\  }} 
 
  \\\hline
 
 \multirow{2}{*}{$X^{15}_0(4)$} &
 
 \multicolumn{2}{c|}{\multirow{1}{*}{\text}{ \tabincell{c}{
  $z^2=-3x^2-1$ \\
  $y^2=-(4x^2-x+1)(4x^2+x+1)(5x^2+3)$ } }} \\
 \cline{2-3}
    &
      \multicolumn{2}{c|}{\tabincell{l}{
   $w_4(x,y,z)=\left(-x,-y,-z\right)$, \\
   $w_{3}(x,y,z)=\left(x,y,-z\right),$\\
   $w_{5}(x,y)=\left(x,-y,-z\right).$  }} 
 
  \\\hline

 \multirow{2}{*}{$X^{21}_0(2)$} &
 
 \multicolumn{2}{c|}{\multirow{1}{*}{\text}{ \tabincell{c}{
  $z^2=-x^2-3$ \\
  $y^2=-(3x-1)(3x+1)(x^2+7)(x^2+3)$ } }} \\
 \cline{2-3}
    &
      \multicolumn{2}{c|}{\tabincell{l}{
   $w_2(x,y,z)=\left(-x,-y,-z\right)$, \\
   $w_{3}(x,y,z)=\left(x,y,-z\right),$\\
   $w_{7}(x,y)=\left(x,-y,z\right).$  }} 
 
  \\\hline
 \multirow{2}{*}{$X^{22}_0(3)$} &
 
 \multicolumn{2}{c|}{\multirow{1}{*}{\text}{ \tabincell{c}{ $y^2=-27x^8 - 308x^6 - 2146x^4 - 308x^2 - 27  $ } }} \\
 \cline{2-3}
    & \multicolumn{2}{c|}{\tabincell{l}{$w_2(x,y)=\left(-\frac{1}{x},-\frac{y}{x^4}\right)$, \\$w_{3}(x,y)=\left(-x,y\right),$ \\ $w_{66}(x,y)=\left(x,-y\right).$ }}
 
\\\hline

 \multirow{2}{*}{$X^{22}_0(5)$} &
 
 \multicolumn{2}{c|}{\multirow{1}{*}{\text}{ \tabincell{c}{ $y^2=-11x^{12} - 80x^{10} - 240x^8 - 362x^6 - 240x^4 - 80x^2 - 11 $} }} \\
 \cline{2-3}
    & \multicolumn{2}{c|}{\tabincell{l}{$w_2(x,y)=\left(\frac{1}{x},\frac{y}{x^6}\right)$, \\$w_{5}(x,y)=\left(-\frac{1}{x},-\frac{y}{x^6}\right),$\\ $w_{110}(x,y)=\left(x,-y\right).$\  }}

 \\\hline
\multirow{2}{*}{$X^{26}_0(3)$} &
 
 \multicolumn{2}{c|}{\multirow{1}{*}{\text}{ \tabincell{c}{ $z^2=-8x^2-3$\\ $ y^2=x^6-2x^4+9x^2+8$ } }} \\
 \cline{2-3}
    & \multicolumn{2}{c|}{\tabincell{l}{$w_2(x,y,z)=\left(-x,-y,-z\right)$, \\$w_{3}(x,y,z)=\left(x,-y,-z\right),$\\ $w_{26}(x,y,z)=\left(x,-y,z\right).$\  }} \\\hline

 \multirow{2}{*}{$X^{39}_0(2)$} &
 
 \multicolumn{2}{c|}{\multirow{1}{*}{\text}{ \tabincell{c}{
  $y^2=-(x^8+11x^7+52x^6+140x^5+243x^4+280x^3+208x^2+88x+16)$ \\
  $\qquad(7x^4+24x^3+32x^2+24x+16)(x^4+3x^3+8x^2+12x+7)$ } }} \\
 \cline{2-3}
    &
      \multicolumn{2}{c|}{\tabincell{l}{
   $w_4(x,y,z)=\left(-x,-y,-z\right)$, \\
   $w_{3}(x,y,z)=\left(x,y,-z\right),$\\
   $w_{5}(x,y)=\left(x,-y,-z\right).$  }} 
 
  \\\hline

\end{tabular} 
\bigskip
\caption{Equations of level greater than one}
\end{table}

\section{Tables of coordinates of CM-points on Shimura curves}
In this appendix, we list coordinates of rational CM-points on Shimura
curves $X_0^D(N)/W_{D,N}$ and also the $x$-coordinates of CM-points in
the equations of $X_0^D(N)$. (However, we do not claim that the list
of rational CM-points on $X_0^D(N)/W_{D,N}$ is complete.) In addition,
for some larger $D$, we also give several CM-points of degree $2$ that
are used in our determination of the equations of Shimura curves.

\renewcommand{\arraystretch}{1.5}
\tabcolsep=10pt                         
\begin{table}[H]
\begin{tabular}{|c|c|c|}
\hline
CM-point& ${X_{0}^{26}(1)/W_{26,1}}$ & ${X_{0}^{26}(1)/w_{26}}$ \\
         discriminant & $s$ & $x\ (s=x^2)$ \\\hline
${-8}$ & $\infty$ & $\infty$  \\ 
${-11}$ & $1$ & $\pm 1$  \\ 
${-19}$ & $9$ & $\pm 3$  \\ 
${-20}$ & $5$ & $\pm \sqrt{5}$ \\
${-24}$ & $-3$ & $\pm \sqrt{-3}$  \\ 
${-52}$ & $0$ & $0$  \\
${-67}$ & $81/25$ & $\pm {9}/{5}$  \\ 
${-91}$ & $-9/7$ & $\pm {3/\sqrt{-7}}$  \\

${-148}$ & $333/25$ & $\pm {3\sqrt{37}}/{5}$  \\ 
${-163}$ & $4761/1225$ & $\pm {69}/{35}$  \\
${-232}$ & $-225/98$ & $\pm {15\sqrt{-2}}/{14}$  \\ 
${-312}$ & $-75/2$ & $\pm {5\sqrt{-6}}/{2}$ \\ 
${-403}$ & $-1521/775$ & $\pm {39\sqrt{-31}}/{155}$  \\ 
${-520}$ & $1521/245$ & $\pm39\sqrt{5}/35$\\ \hline
\end{tabular}
\caption{CM-values of $X^{26}_0(1)$}
\bigskip

\begin{tabular}{|c|c|c|}
\hline
CM-point& ${X_{0}^{35}(1)/W_{35,1}}$ & ${X_{0}^{35}(1)/w_{35}}$  \\ 
         discriminant & $s$ & $x\ (s=x^2)$  \\ \hline
${-7}$ & $\infty$ & $\infty$  \\ 
${-8}$ & $1$ & $\pm 1$   \\ 
${-15}$ & $-3$ & $\pm \sqrt{-3}$  \\ 
${-28}$ & $0$ & $0$ \\ 
${-35}$ & $-7$ & $\pm \sqrt{-7}$  \\ 
${-43}$ & $9$ & $\pm 3$  \\ 
${-60}$ & $-{1}/{3}$ & $\pm {\sqrt{-3}}/{3}$ \\ 
${-67}$ & ${1}/{9}$ & $ \pm{1}/{3}$  \\ 

${-163}$ & ${81}/{25}$ & $ \pm{9}/{5}$  \\
${-235}$ & $-{47}/{9}$ & $ \pm{\sqrt{-47}}/{3}$  \\

${-280}$ & $-{9}/{7}$ & $ \pm{3\sqrt{-7}}/{7}$ \\ 
${-315}$ & $-{1}/{7}$ & $ \pm {\sqrt{-7}}/{7}$ \\ 
${-427}$ & $-{175}/{9}$ & $ \pm{{5\sqrt{
-7}}}/{3}$  \\ 
${-595}$ & ${17}/{9}$ & $ \pm{\sqrt{17}}/{3}$  \\
${-1435}$ & $-{369}/{7}$ & $ \pm{3\sqrt{-41}}/{7}$ \\ \hline
\end{tabular}
\caption{CM-values of $X^{35}_0(1)$}
\end{table}

\begin{table}
\begin{tabular}{|c|c|c|}
\hline
CM-point& ${X_{0}^{38}(1)/W_{38,1}}$ & ${X_{0}^{38}(1)/w_{38}}$  \\ 
         discriminant & $s$ & $x\ (s=x^2)$  \\ \hline
${-4}$ & $\infty$ & $\infty$  \\ 
${-11}$ & $1$ & $\pm 1$  \\ 
${-19}$ & $0$ & $0$ \\ 
${-20}$ & $-1$ & $\pm \sqrt{-1}$ \\ 
${-24}$ & $-3$ & $\pm \sqrt{-3}$ \\ 
${-43}$ & $9$ & $\pm 3$  \\ 
${-163}$ & ${4}/{81}$ & $\pm {2}/{9}$  \\ 
${-228}$ & $-{19}/{9}$ & $ \pm{\sqrt{-19}}/{3}$   \\

${-232}$ & ${29}/{81}$ & $ \pm{\sqrt{29}}/{9}$   \\
${-532}$ & $-{9}/{49}$ & $ \pm{3\sqrt{-1}}/{7}$   \\ 
${-760}$ & $-{171}/{16}$ & $ \pm{3\sqrt{-19}}/{4}$   \\ \hline
\end{tabular}
\bigskip
\caption{CM-values of $X^{38}_0(1)$}

\begin{tabular}{|c|c|c|}
\hline
CM-point& ${X_{0}^{39}(1)/W_{39,1}}$ & ${X_{0}^{39}(1)/w_{39}}$ \\ 
        discriminant  & $s$ & $x\ (s=(2x^2-3x-2)/{x})$ \\ \hline
${-7}$ & $\infty$ & $0,\infty$   \\ 
${-15}$ & $-1$ & $(1\pm\sqrt{5})/{2}$  \\ 
${-19}$ & $-3$ & $\pm 1$  \\ 
${-24}$ & $1$ & $1\pm\sqrt{2}$ \\ 
${-28}$ & $0$ & $-{1}/{2},2$  \\ 
${-60}$ & $-5$ & $(-1\pm\sqrt{5})/2$  \\
${-67}$ & ${7}/{3}$ & $-{1}/{3},3$  \\ 
${-91}$ & $-{1}/{3}$ & $ (2\pm\sqrt{13})/{3}$  \\ 
${-123}$ & ${1}/{5}$ & $ (4\pm\sqrt{11})/{5}$ \\ 
${-163}$ & $-{57}/{35}$ & $ -{5}/{7},{7}/{5}$  \\ 

${-195}$ & $5$ & $ 2\pm \sqrt{5}$ \\ 
${-267}$ & $-{47}/{5}$ & $ (-8\pm\sqrt{89})/{5}$  \\ 
${-312}$ & $-{17}/{7}$ & $(1\pm 5\sqrt{2})/{7}$   \\ 
$-{403}$ & $-{75}/{17}$ & $ (-6\pm5\sqrt{13})/{17}$  \\ \hline

\end{tabular}
\bigskip
\caption{CM-values of $X^{39}_0(1)$}

\end{table}

\renewcommand{\arraystretch}{1.5}
\tabcolsep=10pt                         
\begin{table}
\begin{tabular}{|c|c|c|}
\hline
CM-point& ${X_{0}^{51}(1)/W_{51,1}}$ & ${X_{0}^{51}(1)/w_{51}}$  \\ 
       discriminant   & $s$ & $x\ (s=x^2)$   \\ \hline
${-3}$ & $\infty$ & $\infty$   \\ 
${-7}$ & $1$ & $\pm 1$  \\ 
${-12}$ & $0$ & $ 0$  \\ 
${-24}$ & $-{1}/{3}$ & $\pm{1}/{\sqrt{-3}}$ \\ 
${-28}$ & ${1}/{9}$ & $\pm{1}/{3}$ \\ 
${-51}$ & $-3$ & $\pm\sqrt{-3}$  \\ 
${-163}$ & ${1}/{16}$ & $ \pm{1}/{4}$   \\
${-187}$ & $-{11}/{9}$ & $\pm{\sqrt{-11}}/{3}$  \\
${-267}$ & $-{25}/{3}$ & $\pm{5}/{\sqrt{-3}}$  \\ 
${-408}$ & $-{1}/{6}$ & $ \pm{1}/{\sqrt{-6}}$ \\ \hline
\end{tabular}
\bigskip
\caption{CM-values of $X^{51}_0(1)$}

\begin{tabular}{|c|c|c|}
\hline
CM-point& ${X_{0}^{55}(1)/W_{55,1}}$ & ${X_{0}^{55}(1)/w_{55}}$ \\ 
        discriminant  & $s$ & $x\ (s=(x^2-1)/{x})$   \\ \hline
${-3}$ & $-{3}/{2}$ & ${1}/{2},-2$  \\ 
${-12}$ & $\infty$ & $0,\infty$ \\ 
${-15}$ & $1$ & $(1\pm\sqrt{5})/{2}$  \\ 
${-27}$ & $0$ & $ \pm1$  \\ 
${-60}$ & $-1$ & $(-1\pm\sqrt{5})/{2}$  \\ 
${-67}$ & ${8}/{3}$ & $-{1}/{3},3$ \\ 
${-88}$ & $-2$ & $-1\pm\sqrt{2}$ \\ 
${-115}$ & $-4$ & $ -2\pm\sqrt{5}$ \\ 
${-163}$ & $-{5}/{6}$ & $-{3}/{2},{2}/{3}$ \\
${-187}$ & ${1}/{2}$ & $(1\pm\sqrt{17})/{4}$  \\ 

${-235}$ & $-{4}/{11}$ & $(-2\pm5\sqrt{5})/{11}$  \\ 
${-715}$ & $-{4}/{3}$ & $(-2\pm\sqrt{13})/{3}$  \\ \hline

\end{tabular}
\bigskip
\caption{CM-values of $X^{55}_0(1)$}

\end{table}

\renewcommand{\arraystretch}{1.5}
\tabcolsep=10pt                         
\begin{table}
\renewcommand{\arraystretch}{1.5}
\tabcolsep=10pt                         
\begin{tabular}{|c|c|c|}
\hline
CM-point& ${X_{0}^{57}(1)/W_{57,1}}$ & ${X_{0}^{57}(1)/W_{57,1}}$  \\ 
         discriminant & $s$ & $x$\\ \hline
${-4}$ & $\infty$ & $\infty$  \\ 
${-7}$ & $-1$ & $\pm\sqrt{-7}$   \\ 
${-16}$ & $0$ & $\pm\sqrt{-1}$ \\ 
${-19}$ & $-{1}/{3}$ & $ \pm\sqrt{-19}/3$  \\
${-24}$ & $1$ & $\pm\sqrt{-3}$ \\ 
${-28}$ & ${1}/{3}$ & $\pm\sqrt{-7}/3$ \\
${-43}$ & $-3$ & $\pm\sqrt{-43}$  \\ 
${-123}$ & $-{1}/{2}$ & $ \pm\sqrt{
-3}$  \\ 
${-163}$ & $-{11}/{6}$ & $\pm\sqrt{-163}/3$  \\ 
${-267}$ & $-11$ & $\pm13\sqrt{-3}$ \\ \hline

\end{tabular}
\bigskip
\caption{CM-values of $X^{57}_0(1)$}

\begin{tabular}{|c|c|c|}
\hline
CM-point& ${X_{0}^{58}(1)/W_{58,1}}$ & ${X_{0}^{58}(1)/w_{29}}$  \\ 
        discriminant  & $s$ & $x\ (s={-x^2}/{19})$   \\ \hline

${-3}$ & ${27}/{19}$ & $\pm 3\sqrt{-3}$  \\ 
${-8}$ & $\infty$ & $\infty$  \\ 
${-11}$ & ${11}/{19}$ & $\pm \sqrt{-11}$   \\ 
${-19}$ & $1$ & $\pm \sqrt{-19}$    \\
${-27}$ & ${3}/{19}$ & $\pm \sqrt{-3}$   \\ 
${-40}$ & $-{5}/{19}$ & $\pm \sqrt{5}$   \\ 
${-43}$ & ${43}/{19}$ & $\pm \sqrt{-43}$ \\ 
${-148}$ & ${25}/{19}$ & $\pm 5\sqrt{-1}$  \\ 
${-163}$ & ${163}/{475}$ & $ \pm{\sqrt{-163}}/{5}$  \\ 
${-232}$ & $0$ & $ 0$ \\ \hline
\end{tabular}
\bigskip
\caption{CM-values of $X^{58}_0(1)$}

\end{table}

\renewcommand{\arraystretch}{1.5}
\tabcolsep=10pt                         
\begin{table}
\renewcommand{\arraystretch}{1.5}
\tabcolsep=10pt                         
\begin{tabular}{|c|c|c|}
\hline
CM-point& ${X_{0}^{62}(1)/W_{62,1}}$ & ${X_{0}^{62}(1)/w_{62}}$  \\ 
       discriminant   & $s$ & $x\ (s=x^2)$   \\ \hline
${-4}$ & $\infty$ & $\infty$  \\ 
${-8}$ & $0$ & $0$  \\ 
${-19}$ & $1$ & $\pm 1$  \\
${-20}$ & $-1$ & $\pm \sqrt{-1}$  \\ 
${-40}$ & $-{1}/{2}$ & $\pm 1/\sqrt{-2}$  \\ 
${-67}$ & ${1}/{4}$ & $\pm {1}/{2}$  \\ 
${-163}$ & $25$ & $\pm 5$   \\
${-372}$ & $-{4}/{9}$ & $\pm {2\sqrt{-1}}/{3}$  \\ 
${-403}$ & ${13}/{4}$ & $ \pm{\sqrt{13}}/{2}$  \\ \hline
\end{tabular}
\bigskip
\caption{CM-values of $X^{62}_0(1)$}

\renewcommand{\arraystretch}{1.5}
\tabcolsep=10pt                         

\begin{tabular}{|c|c|c|}
\hline
CM-point& ${X_{0}^{69}(1)/W_{69,1}}$ & ${X_{0}^{69}(1)/w_{69}}$ \\ 
        discriminant  & $s$ & $x\ (s=x^2)$   \\ \hline
${-3}$ & $\infty$ & $\infty$  \\ 
${-4}$ & $1$ & $\pm1$ \\ 

${-12}$ & $0$ & $0$  \\
${-24}$ & $-3$ & $ \pm\sqrt{-3}$  \\ 
${-75}$ & ${9}/{5}$ & $\pm{3/\sqrt{5}}$ \\ 
${-115}$ & $5$ & $\pm\sqrt{5}$ \\ 
${-123}$ & $-{1}/{3}$ & $\pm 1/\sqrt{-3}$ \\ 
${-147}$ & $-{9}/{7}$ & $\pm{3}/{\sqrt{-7}}$   \\ 
${-163}$ & ${25}/{9}$ & $\pm{{5}}/{3}$  \\ 
${-483}$ & $-27$ & $\pm3\sqrt{-3}$  \\ \hline
\end{tabular}
\bigskip
\caption{CM-values of $X^{69}_0(1)$}

\end{table}

\renewcommand{\arraystretch}{1.5}
\tabcolsep=10pt                         
\begin{table}
\begin{tabular}{|c|c|c|}
\hline
CM-point& ${X_{0}^{74}(1)/W_{74,1}}$ & ${X_{0}^{74}(1)/w_{74}}$ \\ 
       discriminant   & $s$ & $x\ (s=(x^2-9)/{4})$  \\ \hline
${-8}$ & $\infty$ & $\infty$  \\ 
${-19}$ & $-2$ & $\pm 1$  \\
${-20}$ & $-1$ & $\pm \sqrt{5}$   \\
${-24}$ & $-3$ & $\pm \sqrt{-3}$  \\ 
${-35}$ & $\pm\sqrt{5}$ & $\pm 2\pm \sqrt{5}$\\
${-43}$ & $0$ & $\pm 3$  \\
${-51}$ & $-2\pm \sqrt{-3}$ & $\pm 2\pm \sqrt{-3}$   \\ 
${-52}$ & $1$ & $\pm \sqrt{13}$  \\
${-88}$ & $-5$ & $\pm \sqrt{-11}$ \\ 
${-91}$ & $2\pm \sqrt{13}$ & $\pm 2\pm \sqrt{13}$  \\
${-148}$ & $-{9}/{4}$ & $0$ \\ 
${-163}$ & $10$ & $\pm 7$   \\ \hline
\end{tabular}
\bigskip
\caption{CM-values of $X^{74}_0(1)$}

\renewcommand{\arraystretch}{1.5}
\tabcolsep=10pt                         
\begin{tabular}{|c|c|c|}
\hline
CM-point& ${X_{0}^{82}(1)/W_{82,1}}$ & ${X_{0}^{82}(1)/w_{41}}$  \\ 
        discriminant  & $s$ & $x$  \\ \hline
${-3}$ & ${1}/{2}$ & $ \pm{3\sqrt{-3}}/{2}$ \\ 
${-11}$ & $0$ & $ \pm\sqrt{-11}$ \\ 
${-19}$ & $\infty $ & $\infty $ \\ 
${-24}$ & $1$ & $\pm2\sqrt{-3}$  \\ 
${-27}$ & $-1$ &$\pm4\sqrt{-3}$ \\
${-52}$ & ${1}/{3}$ & $ \pm8\sqrt{-1}/3$
\\ 
${-67}$ & ${2}/{3}$ & $\pm\sqrt{-67}/3$ \\ 
${-88}$ & $-{1}/{2}$ & $ \pm3\sqrt{-11}/2$
\\ 
${-123}$ & ${1}/{4}$ & $ \pm\sqrt{-123}/4$  \\
${-232}$ & ${7}/{13}$ & $ \pm24\sqrt{-2}/13$ \\ \hline
\end{tabular}
\bigskip
\caption{CM-values of $X^{82}_0(1)$}
\end{table}

\renewcommand{\arraystretch}{1.5}
\tabcolsep=10pt                         
\begin{table}
\begin{tabular}{|c|c|c|c|}
\hline
CM-point& ${X_{0}^{86}(1)/W_{86,1}}$ & ${X_{0}^{86}(1)/w_{86}}$ \\ 
       discriminant   & $s$ & $x\ (s=(-x^2+9)/{4})$   \\ \hline
${-4}$ & $\infty$ & $\infty$   \\ 
${-11}$ & $2$ & $\pm 1$  \\ 
${-24}$ & $3$ & $ \pm{\sqrt{-3}}$  \\ 
${-35}$ & $\pm \sqrt{5}$ & $\pm 2\pm \sqrt{5}$ \\ 
${-40}$ & $1$ & $\pm \sqrt{5}$  \\ 
${-43}$ & ${9}/{4} $ & $0 $  \\ 
${-52}$ & ${5}/{2}$ & $\pm \sqrt{-1}$  \\ 
${-56}$ & $1\pm \sqrt{2}$ & $\pm \sqrt{5\pm 4\sqrt{2}}$ \\ 
${-67}$ & $0$ & $\pm 3$  \\ 
${-68}$ & $(5\pm\sqrt{17})/{4}$ & $\pm \sqrt{4\pm \sqrt{17}}$   \\ 
${-228}$ & $10\pm \sqrt{57}$ & $\pm\sqrt{19}\pm 2\sqrt{-3}$   \\ 
${-232}$ & $-5$ & $ \pm\sqrt{29}$  \\ \hline

\end{tabular}
\bigskip
\caption{CM-values of $X^{86}_0(1)$}

\renewcommand{\arraystretch}{1.5}
\tabcolsep=10pt                         
\begin{tabular}{|c|c|c|}
\hline
CM-point& ${X_{0}^{87}(1)/W_{87,1}}$ & ${X_{0}^{87}(1)/w_{87}}$ \\ 
        discriminant  & $s$ & $x\ (s=x^2)$  \\ \hline
${-3}$ & $\infty $ & $\infty $  \\ 
${-12}$ & $0$ & $0$  \\ 
${-15}$ & $-3$ & $\pm \sqrt{-3}$   \\ 
${-19}$ & $1$ & $ \pm1$  \\ 
${-43}$ & $9$ & $\pm 3$  \\
${-48}$ & $-1$ & $\pm \sqrt{-1}$  \\ 

${-60}$ & $-{1}/{3}$ & $\pm {\sqrt{-3}}/{3}$  \\ 
${-147}$ & $-{9}/{7}$ & $\pm{3/\sqrt{-7}}$  \\ 
${-435}$ & $-{5}/{3}$ & $\pm {\sqrt{-15}}/{3}$ \\ \hline

\end{tabular}
\bigskip
\caption{CM-values of $X^{87}_0(1)$}
\end{table}

\renewcommand{\arraystretch}{1.5}
\tabcolsep=10pt                         
\begin{table}
\begin{tabular}{|c|c|c|}
\hline
CM-point& ${X_{0}^{93}(1)/W_{93,1}}$ & ${X_{0}^{93}(1)/w_{31}}$ \\ 
        discriminant  & $s$ & $x$  \\ \hline
${-4}$ & $\infty $ & $\infty $ \\ 
${-7}$ & $-1$ & $\pm\sqrt{-7}$   \\ 
${-16}$ & $0$ & $\pm3\sqrt{-1}$  \\ 
${-19}$ & $1$ & $\pm\sqrt{-19}$ \\ 
${-28}$ & $3$ & $\pm3\sqrt{-7}$ \\ 
${-51}$ & $-3$ & $\pm3\sqrt{-3}$  \\ 
${-67}$ & $-{1}/{3}$ & $\pm\sqrt{-67}/3$ \\ 

${-163}$ & $-7$ & $\pm\sqrt{-163}$ \\ 
${-267}$ & ${3}/{2}$ & $\pm3\sqrt{-3}$ \\ 
${-403}$ & ${7}/{3}$ & $\pm\sqrt{-403}/3$ \\ \hline

\end{tabular}
\bigskip
\caption{CM-values of $X^{93}_0(1)$}

\renewcommand{\arraystretch}{1.5}
\tabcolsep=10pt                         
\begin{tabular}{|c|c|c|}
\hline
CM-point& ${X_{0}^{94}(1)/W_{94,1}}$ & ${X_{0}^{94}(1)/w_{94}}$ \\ 
         discriminant & $s$ & $x\ (s=2x^2-9)$   \\ \hline
${-3}$ & $-1$ & $\pm 2$   \\ 
${-4}$ & $-9$ & $0$  \\
${-8}$ & $\infty $ & $\infty $ \\ 
${-24}$ & $-5$ & $\pm \sqrt{2}$   \\ 
${-27}$ & $-7$ & $\pm 1$   \\ 
${-51}$ & $\pm\sqrt{17}$ & $(\pm 1\pm \sqrt{17})/{2}$ 
\\ 
${-56}$ & $-4\pm\sqrt{-7}$ & ${\pm \sqrt{10\pm 2\sqrt{-7}}}/{%
2}$  \\ 
${-84}$ & $-6\pm\sqrt{-7}$ & $(\pm\sqrt{-1}\pm\sqrt{7})/{2}$   \\ 
${-115}$ & $3\pm 4\sqrt{5}$ & $\pm 1\pm \sqrt{5}$   \\ 
${-148}$ & $-17$ & $\pm 2\sqrt{-1}$  \\ 
${-168}$ & $-3\pm2\sqrt{-7}$ & $\pm \sqrt{3\pm \sqrt{-7}}$  \\ 

${-235}$ & $-{13}/{5}$ & $\pm {4/\sqrt{5}}$  \\ \hline
\end{tabular}
\bigskip
\caption{CM-values of $X^{94}_0(1)$}
\end{table}

\clearpage

\renewcommand{\arraystretch}{1.5}
\tabcolsep=10pt                         
\begin{table}

\begin{tabular}{|c|c|c|}
\hline
CM-point& ${X_{0}^{95}(1)/W_{95,1}}$ & ${X_{0}^{95}(1)/w_{95}}$ \\ 
        discriminant  & $s$ & $x\ (s=(x^2-1)/2x)$  \\ \hline
${-7}$ & $\infty $ & $0, \infty $  \\ 

${-20}$ & $\pm \sqrt{-1}$ & $\pm \sqrt{-1}$  \\ 
${-28}$ & $-{3}/{4}$ & $-2,{1}/{2}$   \\ 
${-35}$ & ${1}/{2}$ & $(1\pm\sqrt{5})/{2}$  \\ 

${-43}$ & $0$ & $\pm 1$  \\ 
${-115}$ & $-{1}/{2}$ & $(-1\pm\sqrt{5})/{2}$   \\ 
${-163}$ & ${4}/{3}$ & $-{1}/{3},3$ 
\\ 
${-235}$ & $-2$ & $-2\pm\sqrt{5}$ \\ 
${-760}$ & ${1}/{7}$ & $(1\pm5\sqrt{2})/{7}$  \\ \hline

\end{tabular}
\bigskip
\caption{CM-values of $X^{95}_0(1)$}

\renewcommand{\arraystretch}{1.5}
\tabcolsep=10pt                         
\begin{tabular}{|c|c|c|}
\hline
CM-point& ${X_{0}^{111}(1)/W_{111,1}}$ & ${X_{0}^{111}(1)/w_{111}}$  \\
         discriminant & $s$ & $x\ (s=(x^2-2x-1)/(x^2+x-1))$ \\ \hline
${-15}$ & $\infty $ & $(-1\pm\sqrt{5})/{2} $  \\ 

${-19}$ & $1$ & $0$   \\ 
${-24}$ & $0$ & $1\pm\sqrt{2}$   \\ 
${-43}$ & $-2$ & $\pm1$  \\ 

${-51}$ & $-1$ & $(1\pm\sqrt{17})/{4}$   \\ 
${-52}$ & $\pm\sqrt{-1}$ & $(-1\pm \sqrt{-1})/{2},1\pm \sqrt{-1}$   \\ 
${-60}$ & $-{1}/{2}$ & $(1\pm\sqrt{5})/{2}$   \\ 
${-148}$ & $(2\pm6\sqrt{-1})/{5}$ & $\pm \sqrt{-1}$  \\ 
${-163}$ & $-{1}/{5}$ & $-{1}/{2},2$ \\ 
${-267}$ & $-{1}/{3}$ & $(5\pm\sqrt{89})/{8}$  \\ 
${-555}$ & $2$ & $-2\pm\sqrt{5}$ 
\\ \hline

\end{tabular}
\bigskip
\caption{CM-values of $X^{111}_0(1)$}

\end{table}

\renewcommand{\arraystretch}{1.5}
\tabcolsep=10pt                         
\begin{table}
\begin{tabular}{|c|c|c|}
\hline
CM-point& ${X_{0}^{119}(1)/W_{119,1}}$ & ${X_{0}^{119}(1)/w_{119}}$ \\ 
         discriminant & $s$ & $x\ (s=(x^2-5)/{4})$   \\ \hline
${-7}$ & $\infty $ & $\infty $ \\ 

${-11}$ & $-1$ & $\pm1$   \\ 
${-28}$ & $-{5}/{4}$ & $0$   \\ 
${-51}$ & $-2$ & $\pm\sqrt{-3}$  \\ 
${-56}$ & $\pm\sqrt{2}$ & $\pm\sqrt{5\pm4\sqrt{2}}$ \\ 
${-63}$ & $(-9\pm\sqrt{-3})/{6}$ & ${\pm\sqrt{-9\pm6\sqrt{-3}}}/{3}$   \\ 
${-91}$ & $-3$ & $\pm\sqrt{-7}$  \\ 
${-99}$ & $-1\pm\sqrt{-3}$ & $\pm\sqrt{1\pm4\sqrt{-3}}$ 
\\ 

${-112}$ & $-{3}/{2}$ & $\pm \sqrt{-1}$   \\ 
${-163}$ & $1$ & $\pm3$   \\ 

${-232}$ & $(-13\pm\sqrt{-2})/{9}$ & $(\pm 1\pm 2\sqrt{-2})/{3}$ 
\\ 
${-595}$ & $0$ & $\pm \sqrt{5}$   \\ \hline

\end{tabular}
\bigskip
\caption{CM-values of $X^{119}_0(1)$}

\renewcommand{\arraystretch}{1.5}
\tabcolsep=10pt                         
\begin{tabular}{|c|c|c|}
\hline
CM-point& ${X_{0}^{134}(1)/W_{134,1}}$ & ${X_{0}^{134}(1)/w_{134}}$ \\ 
discriminant & $s$ & $x\ (s=(2x^2-5x+2)/(x^2-2x+1))$  \\ \hline
${-4}$ & $\infty $ & $1 $ \\ 

${-19}$ & $2$ & $0,\infty$   \\ 
${-24}$ & $3$ & $(1\pm\sqrt{-3})/{2}$   \\ 
${-35}$ & $\pm \sqrt{5}$ & $(\pm 1\pm\sqrt{5})/{2}$  \\ 
${-40}$ & $1$ & $(3\pm\sqrt{5})/{2}$   \\ 
${-67}$ & ${9}/{4}$ & $-1$  \\
${-88}$ & $5$ & $(5\pm\sqrt{-11})/{6}$ \\ 
${-91}$ & $-5\pm 2\sqrt{13}$ & $(- 1\pm\sqrt{13})/{2},({1\pm\sqrt{13}})/{6}$  \\ 
${-148}$ & ${5}/{2}$ & $\pm \sqrt{-1}$  \\ 
${-163}$ & $0$ & ${1}/{2},2$ \\ \hline
\end{tabular}

\bigskip
\caption{CM-values of $X^{134}_0(1)$}
\end{table}

\renewcommand{\arraystretch}{1.5}
\tabcolsep=10pt                         
\begin{table}
\begin{tabular}{|c|c|c|}
\hline
CM-point& ${X_{0}^{146}(1)/W_{146,1}}$ & ${X_{0}^{146}(1)/w_{146}}$  \\ 
discriminant & $s$ & $x\ (s=(x^2-1)/{x})$   \\ \hline
${-11}$ & $\infty $ & $0,\infty$  \\ 

${-20}$ & $1$ & $(1\pm\sqrt{5})/{2}$\\ 
${-40}$ & $-1$ & $(-1\pm\sqrt{5})/{2}$   \\ 
${-43}$ & $0$ & $\pm 1$   \\ 
${-51}$ & $\pm \sqrt{-3}$ & $(\pm 1\pm\sqrt{-3})/{2}$ \\ 
${-52}$ & $3$ & $(3\pm \sqrt{13})/{2}$  \\ 

${-88}$ & $2$ & $1\pm\sqrt{2}$  \\ 

${-132}$ & $\pm \sqrt{-1}$ & $(\pm\sqrt{-1}\pm\sqrt{3})/{2}$   \\ 

${-232}$ & $5$ & $(5\pm \sqrt{29})/{2}$  \\ 
${-292}$ & $\pm 2\sqrt{-1}$ & $\pm \sqrt{-1}$   \\ \hline
\end{tabular}%
\bigskip
\caption{CM-values of $X^{146}_0(1)$}

\renewcommand{\arraystretch}{1.5}
\tabcolsep=10pt                         
\begin{tabular}{|c|c|c|}
\hline
CM-point& ${X_{0}^{159}(1)/W_{159,1}}$ & ${X_{0}^{159}(1)/w_{159}}$  \\ 
discriminant & $s$ & $x\ (s=(3x^2+7)/{2})$   \\ \hline
${-3}$ & $\infty $ & $\infty $  \\ 

${-12}$ & ${7}/{2}$ & $0$ \\
${-19}$ & $5$ & $\pm1$   \\ 
${-39}$ & $\pm\sqrt{13}$ & ${\pm\sqrt{-21\pm6\sqrt{13}}}/{3}$  \\ 
${-48}$ & $2$ & $\pm \sqrt{-1}$   \\ 
${-57}$ & $3$ & $\pm1/\sqrt{-3}$  \\ 
${-67}$ & ${11}/{3}$ & $\pm{1}/{3}$ \\ 
${-75}$ & ${19}/{5}$ & $\pm1/\sqrt{5}$   \\ 
${-84}$ & $(5\pm2\sqrt{7})/{3}$ & $(\pm 2\sqrt{-1}\pm\sqrt{-7})/{3}$   \\ 
${-120}$ & $(17\pm2\sqrt{10})/{3}$ & $(\pm2\sqrt{2}\pm\sqrt{5})/{3}$  \\ 
${-132}$ & $-7\pm6\sqrt{3}$ & $\pm2\sqrt{-1}\pm\sqrt{-3}$   \\ 
${-156}$ & $(34\pm\sqrt{13})/{9}$ & $\pm\sqrt{15\pm6\sqrt{13}}/{9} $  \\ 
${-232}$ & $(91\pm2\sqrt{-2})/{27}$ & $\pm\sqrt{-7\pm4\sqrt{-2}}/{9} $   \\

${-267}$ & $-1$ & $\pm\sqrt{-3}$     \\ 
${-795}$ & $11$ & $\pm\sqrt{5}$   \\ \hline

\end{tabular}%
\bigskip
\caption{CM-values of $X^{159}_0(1)$}
\end{table}

\renewcommand{\arraystretch}{1.5}
\tabcolsep=10pt                         
\begin{table}[htb]
\newcommand{\tabincell}[2]{\begin{tabular}{@{}#1@{}}#2\end{tabular}}
  \centering
\begin{tabular}{|c|c|c|}
\hline
CM-point& ${X_{0}^{194}(1)/W_{194,1}}$ & ${X_{0}^{194}(1)/w_{194}}$  \\ 
discriminant & $s$ & $x\ (s=(x^2-1)/{x})$  \\ \hline
${-19}$ & $\infty $ & $0,\infty $  \\ 

${-20}$ & $1$ & $(1\pm\sqrt{5})/{2}$ \\ 
${-40}$ & $-1$ & $(-1\pm\sqrt{5})/{2}$  \\
${-51}$ & $\pm \sqrt{-3}$ & $(\pm 1\pm\sqrt{-3})/{2}$ 
\\ 
${-52}$ & $-3$ & $(-3\pm\sqrt{13})/{2}$ \\
${-67}$ & $0$ & $\pm 1$  \\ 
${-123}$ & $(-3\pm 2\sqrt{-3})/{3}$ & $(-3\pm\sqrt{-3})/{2},(3\pm\sqrt{-3})/{6}$ 
\\ 
${-148}$ & $-{1}/{3}$ & $(-1\pm\sqrt{37})/{6}$ 
\\

${-232}$ & $5$ & $(5\pm \sqrt{29})/{2}$  \\ 
${-235}$ & $\pm \sqrt{5}$ & $(\pm 3\pm\sqrt{5})/{2}$ 
\\

${-388}$ & $\pm 2\sqrt{-1}$ & $\pm \sqrt{-1}$  \\ \hline
\end{tabular}
\bigskip
\caption{CM-values of $X^{194}_0(1)$}

\renewcommand{\arraystretch}{1.5}
\tabcolsep=10pt                         
  \centering
\begin{tabular}{|c|c|c|}
\hline
CM-point& ${X_{0}^{206}(1)/W_{206,1}}$ & ${X_{0}^{206}(1)/w_{206}}$  \\ 
discriminant & $s$ & $x\ (s=2x^2+1)$  \\ \hline
${-4}$ & $1$ & $0$ \\ 
${-8}$ & $\infty$ & $\infty$  \\ 
${-19}$ & $3$ & $\pm 1$   \\ 
${-52}$ & $-1$ & $\pm \sqrt{-1}$  \\ 
${-56}$ & $\pm\sqrt{-7}$ & $\pm \sqrt{-2\pm 2\sqrt{-7}}/{2}$   \\ 

${-91}$ & $-2\pm\sqrt{-7}$ & $\pm \sqrt{-6\pm 2\sqrt{-7}}/{2}$   \\ 

${-120}$ & $2\pm\sqrt{-15}$ & ${\pm \sqrt{2\pm 2\sqrt{-15}}%
}/{2}$ \\ 
${-132}$ & $8\pm\sqrt{33}$ & $(\pm\sqrt{3}\pm\sqrt{11})/{2}$  \\ 
${-163}$ & $9$ & $\pm 2$  \\ 
${-184}$ & $-2\pm\sqrt{-23}$ & $\pm \sqrt{-6\pm 2\sqrt{-23}%
}/{2}$ \\ 
${-232}$ & $-3$ & $\pm \sqrt{-2}$    \\ 
${-235}$ & $19\pm 8\sqrt{5}$ & $\pm 2\pm \sqrt{5}$  \\ 
${-267}$ & $\pm\sqrt{-3}$ & ${\pm \sqrt{-2\pm 2\sqrt{-3}}%
}/{2}$ \\ 
${-328}$ & $-2\pm\sqrt{41}$ & ${\pm \sqrt{-6\pm 2\sqrt{41}}%
}/{2}$  \\ 
${-372}$ & $-13\pm 8\sqrt{3}$ & $\pm 2\sqrt{-1}\pm\sqrt{-3}$  \\ \hline
\end{tabular}
\bigskip
\caption{CM-values of $X^{206}_0(1)$}
\end{table}

\renewcommand{\arraystretch}{1.5}
\tabcolsep=10pt                         
\begin{table}
\begin{tabular}{|c|c|c|c|c|}
\hline
CM-point& ${X_{0}^{6}(11)/W_{6,11}}$ & ${X_{0}^{6}(11)/\langle w_{6},w_{11}\rangle}$ & ${X_{0}^{6}(11)}/w_{66}$ \\ 
discriminant & $s$ & $t\ (s=(-t^2+2)/{2})$ & $x\ (t=(2x^2+2)/(x^2-2x-1))$  \\ \hline
${-19}$ & $-1$ & $\pm 2$ & $-1,0,1,\infty$  \\ 

${-24}$ & $0$ & $\pm \sqrt{2}$ & $- 1\pm \sqrt{2}$  \\ 

${-40}$ & $-{1}/{9}$ & $\pm 2\sqrt{5}/{3}$ & $-2\pm \sqrt{5},%
(-1\pm \sqrt{5})/{2}$  \\ 
${-43}$ & $-{1}/{49}$ & $\pm 10/{7}$ & $-3,-2,{1}/{3},%
{1}/{2}$  \\ 

${-51}$ & $-{1}/{17}$ & $\pm6/\sqrt{17}$ & $(-3\pm 
\sqrt{17})/{4},(-3\pm \sqrt{17})/{2}$ 
\\ 
${-52}$ & $-{1}/{25}$ & $\pm2 \sqrt{13}/{5}$ & $(-3\pm 
\sqrt{13})/{2},(-2\pm \sqrt{13})/{3}$  \\ 
${-84}$ & ${1}/{7}$ & $\pm 2\sqrt{21}/{7}$ & $(\sqrt{21}\pm 
\sqrt{-7})/(\sqrt{21}\pm 7)$  \\ 
${-88}$ & $\infty $ & $\infty $ & $1\pm\sqrt{2} $ \\ 
${-120}$ & ${3}/{5}$ & $\pm{ 2/\sqrt{5}}$ & $(\sqrt{5}%
\pm \sqrt{-15}/({5\pm \sqrt{5}})$ \\ 
${-123}$ & $-{9}/{41}$ & $\pm 10/\sqrt{41}$ & $
(-5\pm \sqrt{41})/{2},(-5\pm \sqrt{41})/{8}$  \\ 
${-132}$ & $1$ & $0$ & $\pm \sqrt{-1}$  \\ \hline
\end{tabular}
\bigskip
\caption{CM-values of $X^{6}_0(11)$}

\renewcommand{\arraystretch}{1.5}
\tabcolsep=10pt                         
\begin{tabular}{|c|c|c|}
\hline
CM-point& $X_{0}^{6}\left( 17\right) /W_{6,17}$ & $X_{0}^{6}\left( 17\right)
/<w_{3},w_{34}>$ \\ 
discriminant & $s$ & $x\ (s=x^2)$ \\ \hline
${-4}$ & $0$ & $0$ \\ 
${-19}$ & $1$ & $\pm 1$ \\ 
${-43}$ & $9$ & $\pm 3$ \\ 
${-51}$ & $\infty$ & $\infty$ \\ 
${-52}$ & $-1$ & $\pm \sqrt{-1}$ \\ 
${-67}$ & ${1}/{4}$ & $\pm {1}/{2}$ \\
${-84}$ & $-3$ & $\pm \sqrt{-3}$ \\ 
${-120}$ & $-{1}/{3}$ & $\pm 1/\sqrt{-3}$ \\ 
${-123}$ & ${1}/{9}$ & $\pm {1}/{3}$ \\ 
${-132}$ & $3$ & $\pm\sqrt{3}$ \\ 
${-408}$ & $-{16}/{3}$ & $\pm{4/\sqrt{-3}}$ \\ \hline
\end{tabular}
\bigskip
\caption{CM-values of $X^{6}_0(17)$}
\end{table}

\renewcommand{\arraystretch}{1.5}
\tabcolsep=10pt                         
\begin{table}
\begin{tabular}{|c|c|c|c|}
\hline
CM-point& $X_{0}^{6}\left( 19\right) /W_{6,19}$ & $X_{0}^{6}\left( 19\right)
/\langle w_2,w_{57} \rangle$ & $X_{0}^{6}\left( 19\right) /w_{_{114}}$ \\ 
discriminant & $s$ & $t\ (s=-t^2/{4})$ & $x\ (t=(x^2-1)/{x})$ \\ \hline
${-3}$ & $0$ & $0$ & $\pm 1$ \\ 
${-19}$ & $\infty $ & $\infty $ & $0,\infty $ \\
${-40}$ & $-{1}/{4}$ & $\pm1$ & $(\pm 1\pm \sqrt{5})
/{2}$ \\ 
${-51}$ & ${3}/{4}$ & $\pm\sqrt{-3}$ & $(\pm 1\pm 
\sqrt{-3})/{2}$ \\ 
${-52}$ & $-{9}/{4}$ & $\pm3$ & $(\pm 3\pm 
\sqrt{13})/{2}$ \\ 
${-67}$ & $-{9}/{16}$ & $\pm{ 3}/{2}$ & $\pm 2,\pm{ 1}/{%
2}$ \\
${-84}$ & $-{3}/{4}$ & $\pm\sqrt{3}$ & $(\pm \sqrt{%
3}\pm \sqrt{7})/{2}$ \\
${-88}$ & $-1$ & $\pm 2$ & $\pm 1\pm \sqrt{2}$ \\ 
${-132}$ & ${1}/{4}$ & $\pm\sqrt{-1}$ & $(\pm \sqrt{-1}\pm \sqrt{3})/{2}$ \\
${-148}$ & $-{1}/{36}$ & $\pm{ 1}/{3}$ & $%
(\pm 1\pm \sqrt{37})/{6}$ \\ 
${-228}$ & $1$ & $\pm 2\sqrt{-1}$ & $\pm \sqrt{-1}$ \\ \hline
\end{tabular}
\bigskip
\caption{CM-values of $X^{6}_0(19)$}

\renewcommand{\arraystretch}{1.5}
\tabcolsep=10pt                         
\begin{tabular}{|c|c|c|c|}
\hline
CM-point& $X_{0}^{6}\left( 29\right) /W_{6,29}$ & $X_{0}^{6}\left( 29\right) /\langle w_3,w_{58}\rangle$ & $%
X_{0}^{6}\left( 29\right) /w_{174}$ \\ 
discriminant & $s$ & $t\ (s={18}/(t^2+18))$ & $x\ (t={-12x}/(x^2-2))$ \\ \hline
${-4}$ & $1$ & $0$ & $0,\infty$ \\ 
${-24}$ & $0$ & $\infty $ & $\pm \sqrt{2}$ \\ 
${-51}$ & ${9}/{17}$ & $\pm4$ & $(\pm 3\pm \sqrt{17})/{2}$
\\ 
${-52}$ & $9$ & $\pm4\sqrt{-1}$ & $\pm \sqrt{-1},\pm 2\sqrt{-1}$ \\ 
${-67}$ & ${1}/{9}$ & $\pm 12$ & $\pm 1,\pm 2$ \\ 
${-88}$ & ${9}/{25}$ & $\pm4\sqrt{2}$ & $\pm{ \sqrt{2}%
}/{2},\pm 2\sqrt{2}$ \\
${-120}$ & $-{3}/{5}$ & $\pm4\sqrt{-3}$ & $(\pm\sqrt{-3}\pm\sqrt{5})/{2}$ \\ 
${-123}$ & ${9}/{41}$ & $\pm8$ & $(\pm 3\pm \sqrt{41})/{%
4}$ \\ 
${-132}$ & ${3}/{11}$ & $\pm4\sqrt{3}$ & $(\pm\sqrt{3}\pm\sqrt{11})/{2}$ \\
${-168}$ & $-{9}/{7}$ & $\pm4\sqrt{-2}$ & $(\pm 3\sqrt{-2}\pm\sqrt{14})/{4}$ \\ 
${-228}$ & ${27}/{19}$ & $\pm{ 4/\sqrt{-3}}$ & $(\pm 3\sqrt{-3}\pm\sqrt{-19})/{2}$ \\ 
${-232}$ & $\infty $ & $\pm 3\sqrt{-2}$ & $\pm \sqrt{-2}$ \\
${-267}$ & ${81}/{89}$ & $\pm{4}/{3}$ & $(\pm 9\pm \sqrt{89})/{2}$ \\ \hline
\end{tabular}
\bigskip
\caption{CM-values of $X^{6}_0(29)$}
\end{table}

\renewcommand{\arraystretch}{1.5}
\tabcolsep=10pt                         
\begin{table}
\begin{tabular}{|c|c|c|c|}
\hline
CM-point& $X_{0}^{6}\left( 31\right) /W_{6,31}$ & $X_{0}^{6}\left( 31\right)/\langle w_6,w_{31} \rangle $ & $%
X_{0}^{6}\left( 31\right) /w_{186}$ \\
discriminant & $s$ & $t\ (s=(t^2+3)/{3})$ & $x\ (t=(x^2-3)/{2x})$ \\ \hline
${-3}$ & $\infty $ & $\infty $ & $0,\infty $ \\ 
${-24}$ & $0$ & $\pm \sqrt{-3}$ & $\pm \sqrt{-3}$ \\ 
${-43}$ & ${4}/{3}$ & $\pm 1$ & $\pm 3,\pm 1$ \\ 
${-52}$ & ${16}/{3}$ & $\pm \sqrt{13}$ & $\pm 4\pm \sqrt{13}$ \\ 
${-84}$ & ${16}/{9}$ & $\pm \sqrt{21}/{3}$ & $(\pm 4\sqrt{3}%
\pm \sqrt{21})/{3}$ \\ 
${-88}$ & ${16}/{27}$ & $\pm{ \sqrt{-11}}/{3}$ & $(\pm 4\pm 
\sqrt{-11})/{3}$ \\ 
${-120}$ & ${8}/{3}$ & $\pm \sqrt{5}$ & $\pm 2\sqrt{2}\pm \sqrt{5}$ \\ 

${-123}$ & $-{16}/{9}$ & $\pm {5/\sqrt{-3}}$ & $\pm 3\sqrt{-3}, \pm {\sqrt{-3}}/{3}$ \\ 
${-148}$ & ${64}/{27}$ & $\pm{ \sqrt{37}}/{3}$ & $(\pm 8\pm 
\sqrt{37})/{3}$ \\ 
${-168}$ & ${8}/{9}$ & $\pm{ \sqrt{-3}}/{3}$ & $(\pm 2\sqrt{6}%
\pm \sqrt{-3})/{3}$ \\ 
${-228}$ & $-{16}/{3}$ & $\pm \sqrt{-19}$ & $\pm 4\sqrt{-1}\pm \sqrt{-19}$ \\ 

${-232}$ & ${1}/{3}$ & $\pm \sqrt{-2}$ & $\pm 1\pm \sqrt{-2}$ \\ 
${-372}$ & $1$ & $0$ & $  \pm\sqrt{3} $ \\ 
${-403}$ & ${52}/{27}$ & $\pm{ 5}/{3}$ & $(\pm 5\pm 2\sqrt{13}%
)/{3}$ \\ \hline
\end{tabular}
\bigskip
\caption{CM-values of $X^{6}_0(31)$}

\renewcommand{\arraystretch}{1.5}
\tabcolsep=10pt                         
\begin{tabular}{|c|c|c|c|}
\hline
CM-point& $X_{0}^{6}\left( 37\right) /W_{6,37}$ & $X_{0}^{6}\left( 37\right)
/{\langle w_6,w_{37}\rangle}$ & $X_{0}^{6}\left( 37\right) /w_{222}$ \\
discriminant & $s$ & $t\ (s={t^2}/(t^2+4))$ & $x\ (t=(x^2-1)/{x})$ \\ \hline
${-3}$ & $0$ & $0$ & $\pm 1$ \\ 
${-4}$ & $1$ & $\infty $ & $0,\infty$ \\ 
${-40}$ & $9$ & $\pm{3/\sqrt{-2}}$ & $\pm{ \sqrt{-2}}/{2},\pm 
\sqrt{-2}$ \\ 
${-67}$ & ${9}/{25}$ & $\pm{ 3}/{2}$ & $\pm{ 1}/{2},\pm 2$ \\ 

${-84}$ & $-{9}/{7}$ & $\pm{ 3\sqrt{-1}}/{2}$ & $(\pm \sqrt{7}\pm 3\sqrt{-1})/{%
4}$ \\ 
${-120}$ & $-{27}/{5}$ & $\pm{ 3\sqrt{-6}}/{4}$ & $(\pm \sqrt{%
10}\pm 3\sqrt{-6})/{8}$ \\
${-123}$ & $-3$ & $\pm\sqrt{-3}$ & $(\pm 1\pm \sqrt{-3})/{2}
$ \\ 
${-132}$ & ${27}/{11}$ & $\pm{ 3\sqrt{-3}}/{2}$ & $(\pm 3%
\sqrt{-3}\pm \sqrt{-11})/{4}$ \\ 

${-148}$ & $\infty$ & $\pm2\sqrt{-1}$ & $\pm\sqrt{-1}$ \\
${-232}$ & ${81}/{49}$ & $\pm{ 9\sqrt{-2}}/{4}$ & $\pm{ \sqrt{%
-2}}/{4},\pm 2\sqrt{-2}$ \\ 
${-312}$ & $-{3}/{13}$ & $\pm{ \sqrt{-3}}/{2}$ & $(\pm \sqrt{%
13}\pm \sqrt{-3})/{4}$ \\ 
${-408}$ & ${9}/{17}$ & $\pm{ 3\sqrt{2}}/{2}$ & $(\pm 3\sqrt{2%
}\pm \sqrt{34})/{4}$ \\ 
${-555}$ & $-{27}/{37}$ & $\pm{ 3\sqrt{-3}}/{4}$ & $(\pm 
\sqrt{37}\pm 3\sqrt{-3})/{8}$ \\ \hline
\end{tabular}

\bigskip
\caption{CM-values of $X^{6}_0(37)$}
\end{table}

\renewcommand{\arraystretch}{1.5}
\tabcolsep=10pt                         
\begin{table}
\begin{tabular}{|c|c|c|c|}
\hline
CM-point& $X_{0}^{10}\left( 11\right) /W_{10,11}$ & $X_{0}^{10}\left( 11\right)
/\langle w_{10},w_{11}\rangle$ & $X_{0}^{10}\left( 11\right) /w_{110}$ \\
discriminant & $s$ & $t\ (s={-t^2}/{2})$ & $x\ (t=(x^2-2)/{2x})$ \\ \hline
${-8}$ & $\infty $ & $\infty $ & $0,\infty $ \\
${-35}$ & ${7}/{8}$ & $\pm{ \sqrt{-7}}/{2}$ & $(\pm 1\pm 
\sqrt{-7})/{2}$ \\ 
${-40}$ & $1$ & $\pm \sqrt{-2}$ & $\pm \sqrt{-2}$ \\ 
${-43}$ & $-{1}/{8}$ & $\pm{ 1}/{2}$ & $\pm 2,\pm 1$ \\ 
${-52}$ & ${9}/{8}$ & $\pm{ 3\sqrt{-1}}/{2}$ & $\pm 2\sqrt{-1},\pm \sqrt{-1}$ \\ 

${-88}$ & $0$ & $0$ & $\pm \sqrt{2}$ \\
${-120}$ & ${3}/{8}$ & $\pm{ \sqrt{-3}}/{2}$ & $(\pm 
\sqrt{5}\pm \sqrt{-3})/{2}$ \\ 
${-132}$ & ${11}/{8}$ & $\pm{ \sqrt{-11}}/{2}$ & $(\pm 
\sqrt{-3}\pm \sqrt{-11})/{2}$ \\ 
${-187}$ & $-{9}/{8}$ & $\pm{ 3}/{2}$ & $(\pm 3\pm 
\sqrt{17})/{2}$ \\ 
${-340}$ & ${17}/{8}$ & $\pm{ \sqrt{-17}}/{2}$ & $(\pm
3\sqrt{-1}\pm \sqrt{-17})/{2}$ \\ 
${-660}$ & ${33}/{32}$ & $\pm{ \sqrt{-33}}/{4}$ & $(\pm \sqrt{-1}\pm \sqrt{-33})/{4}$ \\ 
${-715}$ & ${99}/{104}$ & $\pm{ 3\sqrt{-143}%
}/{26}$ & $(\pm 3\sqrt{-143}\pm \sqrt{65})/{26}$ \\ \hline
\end{tabular}
\bigskip
\caption{CM-values of $X^{10}_0(11)$}

\renewcommand{\arraystretch}{1.5}
\tabcolsep=10pt                         
\begin{tabular}{|c|c|c|}
\hline
CM-point& $X_{0}^{10}\left( 13\right) /W_{10,13}$ & $X_{0}^{10}\left( 13\right)/\langle w_{2}, w_{65} \rangle $ \\ 

discriminant & $s$ & $x\ (s=x^2)$ \\ \hline
${-3}$ & $1$ & $\pm 1$ \\
${-35}$ & $5$ & $\pm\sqrt{5}$ \\ 
${-40}$ & $\infty $ & $\infty$ \\ 
${-43}$ & $9$ & $\pm 3$ \\ 
${-52}$ & $0$ & $0$ \\ 
${-88}$ & ${25}/{9}$ & $\pm{5}/{3}$ \\ 
${-120}$ & $-15$ & $\pm \sqrt{-15}$ \\ 
${-195}$ & $-{5}/{3}$ & $\pm{ \sqrt{-15}}/{3}$ \\ 
${-235}$ & ${5}/{9}$ & $\pm{ \sqrt{5}}/{3}$ \\ 
${-312}$ & $-{13}/{3}$ & $\pm{ \sqrt{-39}}/{3}$ \\ \hline
\end{tabular}
\bigskip
\caption{CM-values of $X^{10}_0(13)$}
\end{table}

\renewcommand{\arraystretch}{1.5}
\tabcolsep=10pt                         
\begin{table}
\begin{tabular}{|c|c|c|}
\hline
CM-point& $X_{0}^{10}\left( 19\right) /W_{10,19}$ & $X_{0}^{10}\left( 19\right)/\langle w_{5}, w_{38} \rangle $ \\ 

discriminant & $s$ & $x\ (s=x^2)$ \\ \hline
${-3}$ & $1$ & $\pm 1$ \\ 
${-8}$ & $0$ & $0 $ \\ 
${-52}$ & $\infty$ & $\infty$ \\ 
${-52}$ & $-4$ & $\pm 2\sqrt{-1}$ \\ 
${-67}$ & ${4}/{27}$ & $\pm {2}/{3}$ \\ 
${-88}$ & $2$ & $\pm \sqrt{2}$ \\
${-148}$ & $-1$ & $\pm \sqrt{-1}$ \\ 
${-228}$ & ${4}/{3}$ & $\pm {2/\sqrt{3}}$ \\ 
${-280}$ & ${4}/{5}$ & $\pm {2/\sqrt{5}}$ \\ 
${-532}$ & $-{1}/{4}$ & $\pm {\sqrt{-1}}/{2}$ \\ 
${-760}$ & ${32}/{5}$ & $\pm {4\sqrt{10}}/{5}$ \\ \hline
\end{tabular}
\bigskip
\caption{CM-values of $X^{10}_0(19)$}

\renewcommand{\arraystretch}{1.5}
\tabcolsep=10pt                         
\begin{tabular}{|c|c|c|c|}
\hline
CM-point& $X_{0}^{10}\left( 23\right) /W_{10,23}$ &$X_{0}^{10}\left( 23\right)
/\langle w_{10},w_{23}\rangle$  &$X_{0}^{10}\left( 23\right)
/w_{230}$  \\ 
discriminant & $s$ & $t\ (s={t^2}/(t^2-5))$ & $x\ (t=(5x^2+5)/(x^2-4x-1))$ \\ \hline
${-20}$ & $0$ & $0$ & $\pm \sqrt{-1}$ \\ 
${-40}$ & $\infty $ & $\pm \sqrt{5}$ & $(-1\pm \sqrt{5})/{2}$ \\ 
${-43}$ & ${5}/{4}$ & $\pm 5$ & $-{1}/{2},0,2,\infty$ \\ 
${-67}$ & $5$ & $\pm{ 5}/{2}$ & $-3,-1,{1}/{3},1$ \\ 
${-88}$ & ${5}/{16}$ & $\pm{ 5/\sqrt{-11}}$ & $(-1\pm 
\sqrt{-11})/{6},(1\pm \sqrt{-11})/{2}$ \\ 
${-115}$ & $1$ & $\infty $ & $2\pm \sqrt{5}$ \\ 
${-120}$ & ${5}/{8}$ & $\pm{ 5/\sqrt{-3}}$ & $(2\sqrt{-1}%
\pm 2\sqrt{-2})/({\sqrt{-1}\pm \sqrt{3}})$ \\ 
${-148}$ & ${5}/{9}$ & $\pm{ 5\sqrt{-1}}/{2}$ & $(-1\pm 2\sqrt{-1})/{5},1\pm 2\sqrt{-1}
$ \\ 
${-235}$ & ${25}/{16}$ & $\pm{ 5\sqrt{5}}/{3}$ & $(-11\pm 5%
\sqrt{5})/{2},(1\pm \sqrt{5})/{2}$ \\ 
${-520}$ & $-{5}/{8}$ & $\pm {5/\sqrt{13}}$ & ${2(1\pm\sqrt{-2})}/({1\pm\sqrt{13}})$ \\ \hline
\end{tabular}
\bigskip
\caption{CM-values of $X^{10}_0(23)$}
\end{table}

\renewcommand{\arraystretch}{1.5}
\tabcolsep=10pt                         
\begin{table}
\begin{tabular}{|c|c|c|}
\hline
CM-point& $X_{0}^{14}\left( 3\right) /W_{14,3}$ & $X_{0}^{14}\left( 3\right)/\langle w_3,w_{14}\rangle $ \\ 

 discriminant & $s$ & $ x\ (s=x^2)$ \\ \hline
${-8}$ & $0$ & $0 $ \\ 
${-11}$ & $1$ & $\pm 1$ \\ 
${-35}$ & $-{1}/{7}$ & $\pm{1/ \sqrt{-7}}$ \\ 
${-51}$ & ${1}/{9}$ & $\pm {1}/{3}$ \\
${-84}$ & $\infty $ & $\infty$ \\ 
${-120}$ & $-{1}/{27}$ & $\pm {\sqrt{-3}}/{9}$ \\ 
${-123}$ & ${25}/{9}$ & $\pm{ 5}/{3}$ \\
${-168}$ & $-{2}/{9}$ & $\pm{\sqrt{-2}}/{3}$ \\
${-228}$ & $-{25}/{27}$ & $\pm{ 5/\sqrt{-3}}$ \\ 
${-267}$ & ${25}/{1521}$ & $\pm{ 5}/{39}$ \\ 
${-312}$ & ${49}/{117}$ & $\pm{ 7\sqrt{13}}/{39}$ \\ \hline
\end{tabular}
\bigskip
\caption{CM-values of $X^{14}_0(3)$}

\renewcommand{\arraystretch}{1.5}
\tabcolsep=10pt                         
\begin{tabular}{|c|c|c|c|}
\hline
CM-point& $X_{0}^{14}\left( 5\right)/ W_{14,5}$ & $X_{0}^{14}\left( 5\right)/\langle w_{5},w_{14}\rangle $ & $%
X_{0}^{14}\left( 5\right) /w_{14}$ \\
discriminant & $s$ & $t\ (s=t^2)$ & $x\ (t=(x^2-x-1)/(2x^2+2))$ \\ \hline
${-4}$ & $\infty$ & $\infty $ & $\pm\sqrt{-1} $ \\ 
${-11}$ & $1$ & $\pm 1$ & $(-1\pm \sqrt{-11})/{2},%
(1\pm \sqrt{-11})/{6}$ \\ 
${-35}$ & $0$ & $0$ & $(1\pm \sqrt{5})/{2}$ \\
${-84}$ & $-1$ & $\pm \sqrt{-1}$ & ${(1\pm 2\sqrt{-1})(1\pm\sqrt{21})}/{10}$ \\ 
${-91}$ & $-7$ & $\pm \sqrt{-7}$ & $({1\pm 3\sqrt{13}})/{(2\pm 4\sqrt{-7})}$ \\
${-120}$ & $5$ & $\pm \sqrt{5}$ & $(1\pm 5\sqrt{-3})/{(2\pm 4\sqrt{5})}$ \\ 
${-235}$ & ${25}/{81}$ & $\pm {5}/{9}$ & $(-9\pm\sqrt{5})/{2},(9\pm\sqrt{5})/{3}$ \\ 
${-280}$ & ${5}/{16} $ & $\pm{\sqrt{5}}/{4}$ & $-2\pm \sqrt{5}$ \\ 
${-340}$ & $-25$ & $\pm 5\sqrt{-1}$ & ${(1\pm 10\sqrt{-1})(1\pm 9\sqrt{5})}/{202}$ \\ 
${-420}$ & ${5}/{9}$ & $\pm {\sqrt{5}}/{3}$ & $(3\pm\sqrt{-35})/{(6\pm4\sqrt{5})}$ \\
${-520}$ & ${5}/{81}$ & $\pm {\sqrt{5}}/{9}$ & $(9\pm 5\sqrt{13})/{(18\pm4\sqrt{5})}$ \\
${-840}$ & $-{35}/{9}$ & $\pm {\sqrt{-35}}/{3}$ & $%
(3\pm11\sqrt{5})/{(6\pm4\sqrt{-35})}$ \\ \hline
\end{tabular}
\bigskip
\caption{CM-values of $X^{14}_0(5)$}
\end{table}

\renewcommand{\arraystretch}{1.5}
\tabcolsep=10pt                         
\begin{table}
\begin{tabular}{|c|c|c|c|}
\hline
CM-point& $X_{0}^{15}(2) /W_{15,2}$ & $X_{0}^{15}(2)/\gen{w_2,w_{15}}$
  & $X_0^{15}(2)/\gen{w_{15}}$ \\ 
discriminant & $s$ & $t\ (s=t^2)$ & $x\ (t=(x^2+2)/2x)$ \\ \hline
$-7$ & $1/4$ & $\pm 1/2$ & $(\pm 1\pm\sqrt{-7})/2$ \\
$-12$ & $\infty$ & $\infty$ & $\infty,0$ \\
$-15$ & $5/4$ & $\pm\sqrt5/2$ & $(\pm\sqrt5\pm\sqrt{-3})/2$ \\
$-28$ & $9/4$ & $\pm3/2$ & $\pm1,\pm2$ \\
$-40$ & $0$ & $0$ & $\pm\sqrt{-2}$ \\
$-48$ & $-1/4$ & $\pm\sqrt{-1}/2$ & $\pm\sqrt{-1},\pm2\sqrt{-1}$ \\
$-52$ & $1$ & $\pm 1$ & $\pm1\pm\sqrt{-1}$ \\
$-60$ & $-1/12$ & $\pm1/2\sqrt{-3}$ & $\pm\sqrt{-3},\pm2/\sqrt{-3}$ \\
$-88$ & $4$ & $\pm2$ & $\pm2\pm\sqrt2$ \\
$-120$ & $2$ & $\pm\sqrt2$ & $\pm\sqrt2$ \\
$-132$ & $-1$ & $\pm\sqrt{-1}$ & $\pm\sqrt{-1}\pm\sqrt{-3}$ \\
$-148$ & $1/25$ & $\pm1/5$ & $(\pm1\pm7\sqrt{-1})/5$ \\
$-168$ & $2/3$ & $\pm\sqrt{6}/3$ & $(\pm\sqrt6\pm2\sqrt{-3})/3$ \\
$-228$ & $-1/9$ & $\pm\sqrt{-1}/3$ & $(\pm\sqrt{-1}\pm\sqrt{-19})/3$
  \\
$-232$ & $144/121$ & $\pm 12/11$ & $(\pm12\pm7\sqrt{-2})/11$ \\
$-240$ & $-25/12$ & $\pm5/2\sqrt{-3}$ & $\pm2\sqrt{-3},\pm\sqrt{-3}/3$
  \\
$-280$ & $10$ & $\pm\sqrt{10}$ & $\pm2\sqrt2\pm\sqrt{10}$ \\
$-312$ & $2/25$ & $\pm\sqrt2/5$ & $(\pm\sqrt2\pm4\sqrt{-3})/5$ \\
$-340$ & $9/17$ & $\pm3/\sqrt{17}$ & $(\pm3\pm5\sqrt{-1})/\sqrt{17}$
  \\
$-372$ & $-31/9$ & $\pm\sqrt{-31}/3$ &
         $(\pm7\sqrt{-1}\pm\sqrt{-31})/3$ \\
$-408$ & $68/25$ & $\pm2\sqrt{17}/5$ & $(\pm3\sqrt2\pm2\sqrt{17})/5$
  \\
$-420$ & $5/3$ & $\pm\sqrt{15}/3$ & $(\pm\sqrt{15}\pm\sqrt{-3})/3$ \\
$-520$ & $-8/121$ & $\pm2\sqrt{-2}/11$
       & $(\pm2\sqrt{-2}\pm5\sqrt{-10})/11$ \\
$-660$ & $-5/11$ & $\pm\sqrt{-55}/11$
       & $(\pm\sqrt{-55}\pm3\sqrt{-33})/11$ \\
$-708$ & $-841/121$ & $\pm 29\sqrt{-1}/11$
       & $(\pm29\sqrt{-1}\pm19\sqrt{-3})/11$ \\
$-760$ & $450/529$ & $\pm15\sqrt2/23$
       & $(\pm15\sqrt2\pm4\sqrt{-38})/23$ \\
$-840$ & $40/27$ & $\pm2\sqrt{30}/9$
       & $(\pm2\sqrt{30}\pm\sqrt{-42})/9$ \\
\hline
\end{tabular}
\bigskip
\caption{CM-values of $X^{15}_0(2)$}
\end{table}
\clearpage

\renewcommand{\arraystretch}{1.5}
\tabcolsep=10pt                         
\begin{table}
\begin{tabular}{|c|c|c|}
\hline
CM-point& $X_{0}^{21}(2) /W_{21,2}$ & $X_{0}^{21}(2)/\gen{w_3,w_7}$ \\
discriminant & $s$ & $x\ (s=x^2)$ \\ \hline
$-4$ & $\infty$ & $\infty$ \\
$-7$ & $-7$ & $\pm\sqrt{-7}$ \\
$-15$ & $-5/3$ & $\pm\sqrt{-15}/3$ \\
$-16$ & $1$ & $\pm 1$ \\
$-28$ & $1/9$ & $\pm1/3$ \\
$-60$ & $9$ & $\pm3$ \\
$-84$ & $-3$ & $\pm\sqrt{-3}$ \\
$-100$ & $1/5$ & $\pm1/\sqrt5$ \\
$-112$ & $25$ & $\pm5$ \\
$-120$ & $-1/3$ & $\pm1/\sqrt{-3}$ \\
$-148$ & $37/9$ & $\pm\sqrt{37}/3$ \\
$-168$ & $0$ & $0$ \\
$-228$ & $-25/3$ & $\pm5/\sqrt{-3}$ \\
$-232$ & $-32$ & $\pm4\sqrt{-2}$ \\
$-280$ & $-35/9$ & $\pm\sqrt{-35}/3$ \\
$-312$ & $-8/3$ & $\pm2\sqrt{-6}/3$ \\
$-372$ & $-3/4$ & $\pm\sqrt{-3}/2$ \\
$-408$ & $-75$ & $\pm5\sqrt{-3}$ \\
$-420$ & $21$ & $\pm\sqrt{21}$ \\
$-532$ & $-19/4$ & $\pm\sqrt{-19}/2$ \\
$-708$ & $25/48$ & $\pm5/4\sqrt{-3}$ \\
$-840$ & $-16/3$ & $\pm4/\sqrt{-3}$ \\
\hline
\end{tabular}
\bigskip
\caption{CM-values of $X^{21}_0(2)$}
\end{table}
\clearpage

\renewcommand{\arraystretch}{1.5}
\tabcolsep=10pt                         
\begin{table}
\begin{tabular}{|c|c|c|c|}
\hline
CM-point& $X_{0}^{22}\left( 3\right) /W_{22,3}$ & $X_{0}^{22}\left( 3\right)
/\langle w_2,w_{33}\rangle$ & $X_{0}^{22}\left( 3\right) /w_{66}$ \\
discriminant & $s$ & $t\ (s=(t^2+1)/{t^2})$ & $x\ (t=(x^2-1)/2x)$ \\ \hline
${-3}$ & $1$ & $\infty $ & $0,\infty$ \\ 
${-11}$ & $\infty $ & $0$ & $\pm 1$ \\ 
${-20}$ & ${5}/{4}$ & $\pm 2$ & $\pm 2\pm \sqrt{5}$ \\
${-132}$ & $0$ & $\pm \sqrt{-1}$ & $\pm \sqrt{-1}$ \\
${-168}$ & ${27}/{28}$ & $\pm 2\sqrt{-7}$ & $\pm 2\sqrt{-7}\pm 3\sqrt{-3%
}$ \\
${-267}$ & ${169}/{196}$ & $\pm{ 14\sqrt{-3}}/{9}$ & $\pm{ 
\sqrt{-3}}/{9},\pm 3\sqrt{-3}$ \\ 
${-312}$ & ${25}/{52}$ & $\pm{ 2\sqrt{-39}}/{9}$ & $(\pm 2%
\sqrt{-39}\pm 5\sqrt{-3})/{9}$ \\ 
${-372}$ & ${31}/{4}$ & $\pm{ 2\sqrt{3}}/{9}$ & $(\pm 2\sqrt{3%
}\pm \sqrt{93})/{9}$ \\ 
${-408}$ & ${18}/{17}$ & $\pm \sqrt{17}$ & $\pm \sqrt{17}\pm 3\sqrt{2}$
\\
${-627}$ & $-{11}/{16}$ & $\pm{ 4\sqrt{-3}}/{9}$ & $(\pm 4%
\sqrt{-3}\pm \sqrt{33})/{9}$ \\
${-660}$ & ${45}/{44}$ & $\pm 2\sqrt{11}$ & $\pm 2\sqrt{11}\pm 3\sqrt{5%
}$ \\
${-708}$ & ${675}/{676}$ & $\pm 26\sqrt{-1}$ & $\pm
26\sqrt{-1}\pm 15\sqrt{-3}$ \\ \hline
\end{tabular}
\bigskip
\caption{CM-values of $X^{22}_0(3)$}

\renewcommand{\arraystretch}{1.5}
\tabcolsep=10pt                         
\begin{tabular}{|c|c|c|c|}
\hline
CM-point& $X_{0}^{22}\left( 5\right) /W_{22,5}$ & $X_{0}^{22}\left( 5\right)
/\langle w_{5},w_{22}\rangle$ & $X_{0}^{22}\left( 5\right) /w_{110}$ \\ 
discriminant & $s$ & $t\ (s=(t^2+1)/{t^2})$ & $x\ (t={-2x}/(x^2-1))$ \\ \hline
${-4}$ & $1$ & $\infty $ & $\pm 1$ \\ 
${-11}$ & $\infty $ & $0$ & $0,\infty$ \\ 
${-20}$ & $0$ & $\pm \sqrt{-1}$ & $\pm \sqrt{-1}$ \\
${-115}$ & ${5}/{4}$ & $\pm 2$ & $(\pm 1\pm \sqrt{5})/{2}$ \\ 

${-235}$ & $5$ & $\pm{ 1}/{2}$ & $\pm 2\pm \sqrt{5}$ \\
${-280}$ & $-{1}/{7}$ & $\pm{ \sqrt{-14}}/{4}$ & $(\pm 2\sqrt{%
-14}\pm \sqrt{-7})/{7}$ \\
${-520}$ & ${5}/{13}$ & $\pm \sqrt{-26}/{4}$ & $(\pm 2\sqrt{%
-26}\pm \sqrt{65})/{13}$ \\ 
${-660}$ & $-{5}/{4}$ & $\pm{ 2\sqrt{-1}}/{3}$ & $(\pm 3\sqrt{-1}\pm \sqrt{-5}%
)/{2}$ \\ 
${-715}$ & $-{5}/{11}$ & $\pm \sqrt{-11}/{4}$ & $(\pm 4%
\sqrt{-11}\pm \sqrt{-55})/{11}$ \\ 
${-760}$ & $-{5}/{76}$ & $\pm{ 2\sqrt{-19}}/{9}$ & $(\pm 9%
\sqrt{-19}\pm \sqrt{-95})/{38}$ \\ \hline
\end{tabular}
\bigskip
\caption{CM-values of $X^{22}_0(5)$}
\end{table}
\clearpage

\renewcommand{\arraystretch}{1.5}
\tabcolsep=10pt                         
\begin{table}
\begin{tabular}{|c|c|c|}
\hline
CM-point& $X_{0}^{26}\left( 3\right) /W_{26,3}$ & $X_{0}^{26}\left( 3\right)/\langle w_{3},w_{26}\rangle $ \\

discriminant& $s$ & $x\ (s=x^2)$ \\ \hline
${-8}$ & $\infty $ & $\infty $ \\ 
${-11}$ & $1$ & $\pm 1$ \\ 
${-20}$ & $-1$ & $\pm \sqrt{-1}$ \\
${-24}$ & $0$ & $0$ \\ 
${-84}$ & $-3$ & $\pm \sqrt{-3}$ \\ 
${-123}$ & $9$ & $\pm 3$ \\ 
${-132}$ & $3$ & $\pm \sqrt{3}$ \\ 
${-195}$ & $-{3}/{5}$ & $\pm{ \sqrt{-15}}/{5}$ \\ 
${-228}$ & $-{3}/{4}$ & $\pm {\sqrt{-3}}/{2}$ \\ 
${-267}$ & $-{9}/{25}$ & $\pm{3 \sqrt{-1}}/{5}$ \\ 
${-312}$ & $-{3}/{8}$ & $\pm{\sqrt{-6}}/{4}$ \\ 
${-372}$ & $-12$ & $\pm 2\sqrt{-3}$ \\ 
${-408}$ & $-{3}/{25}$ & $\pm{ \sqrt{-3}}/{5}$ \\
${-708}$ & $-{27}/{49}$ & $\pm{ 3\sqrt{-3}}/{7}$ \\ \hline
\end{tabular}
\bigskip
\caption{CM-values of $X^{26}_0(3)$}

\end{table}
\clearpage

\renewcommand{\arraystretch}{1.5}
\tabcolsep=10pt                         
\begin{table}
\begin{tabular}{|c|c|c|c|}
\hline
CM-point& $X_{0}^{39}(2) /W_{39,2}$ & $X_{0}^{39}(2)/\gen{w_3,w_{13}}$
  & $X_0^{39}(2)/\gen{w_{39}}$ \\ 
discriminant & $s$ & $t\ (s=t^2)$ & $x\ (t=3(x^2+4x+2)/(x^2-2))$ \\
  \hline
$-7$ & $-7$ & $\pm\sqrt{-7}$ & $(-3\pm\sqrt{-7})/2,(-3\pm\sqrt{-7})/4$ \\
$-15$ & $-15$ & $\pm\sqrt{-15}$ & $(-3\pm\sqrt{-3})(3\pm\sqrt{-15})/12$ \\
$-24$ & $\infty$ & $\infty$ & $\pm\sqrt2$ \\
$-28$ & $9$ & $\pm3$ & $\infty,0,-1,-2$ \\
$-52$ & $-9$ & $\pm3\sqrt{-1}$ & $-1\pm\sqrt{-1}$ \\
$-60$ & $1$ & $\pm 1$ & $-3\pm\sqrt5,(-3\pm\sqrt5)/2$ \\
$-84$ & $-3$ & $\pm\sqrt{-3}$ & $(3\pm\sqrt3)(-3\pm\sqrt{-3})/6$ \\
$-132$ & $-11$ & $\pm\sqrt{-11}$
      & $(3\pm\sqrt{-1})(-3\pm\sqrt{-11})/10$ \\
$-148$ & $-1$ & $\pm\sqrt{-1}$ & $-3\pm\sqrt{-1},(-3\pm\sqrt{-1})/5$
  \\
$-228$ & $-27$ & $\pm3\sqrt{-3}$ & $(1\pm\sqrt{-1})(-1\pm\sqrt{-3})/2$
  \\
$-232$ & $-9/2$ & $\pm3/\sqrt{-2}$ & $-2\pm\sqrt{-2},(-2\pm\sqrt{-2})/3$
  \\
$-312$ & $0$ & $0$ & $-2\pm\sqrt2$ \\
$-372$ & $-25/3$ & $\pm 5/\sqrt{-3}$
       & $(9\pm\sqrt{-3})(-9\pm5\sqrt{-3})/78$ \\
$-408$ & $-12$ & $\pm2\sqrt{-3}$
       & $(6\pm\sqrt{-6})(-3\pm2\sqrt{-3})/21$ \\
$-520$ & $-10$ & $\pm\sqrt{-10}$
       & $(6\pm\sqrt{-2})(-3\pm\sqrt{-10})/19$ \\
$-708$ & $-59$ & $\pm\sqrt{-59}$
       & $(3\pm5\sqrt{-1})(-3\pm\sqrt{-59})/34$ \\
$-1092$ & $-1/3$ & $\pm1/\sqrt{-3}$
       & $(9\pm\sqrt{39})(-9\pm\sqrt{-3})/42$ \\
\hline
\end{tabular}
\bigskip
\caption{CM-values of $X^{39}_0(2)$}
\end{table}
\clearpage

\end{appendices}

\def\cprime{$'$}


\begin{thebibliography}{99}

\bibitem{Bayer}
Montserrat Alsina and Pilar Bayer.
\newblock {\em Quaternion orders, quadratic forms, and {S}himura curves},
  volume~22 of {\em CRM Monograph Series}.
\newblock American Mathematical Society, Providence, RI, 2004.

\bibitem{Bar}
Alexander~Graham Barnard.
\newblock {\em The singular theta correspondence, {L}orentzian lattices and
  {B}orcherds-{K}ac-{M}oody algebras}.
\newblock ProQuest LLC, Ann Arbor, MI, 2003.
\newblock Thesis (Ph.D.)--University of California, Berkeley.

\bibitem{Borcherds-Invent}
Richard~E. Borcherds.
\newblock Automorphic forms with singularities on {G}rassmannians.
\newblock {\em Invent. Math.}, 132(3):491--562, 1998.

\bibitem{Borcherds-Duke}
Richard~E. Borcherds.
\newblock Reflection groups of {L}orentzian lattices.
\newblock {\em Duke Math. J.}, 104(2):319--366, 2000.

\bibitem{Magma}
Wieb Bosma, John Cannon, and Catherine Playoust.
\newblock The {M}agma algebra system. {I}. {T}he user language.
\newblock {\em J. Symbolic Comput.}, 24(3-4):235--265, 1997.
\newblock Computational algebra and number theory (London, 1993).

\bibitem{Boutot-Carayol}
Jean-Fran\c{c}ois Boutot and Henri Carayol.
\newblock Uniformisation {$p$}-adique des courbes de {S}himura: les
  th\'eor\`emes de \v {C}erednik et de {D}rinfel\cprime d.
\newblock {\em Ast\'erisque}, (196-197):7, 45--158 (1992), 1991.
\newblock Courbes modulaires et courbes de Shimura (Orsay, 1987/1988).

\bibitem{Bruinier-Ono}
Jan Bruinier and Ken Ono.
\newblock Heegner divisors, {$L$}-functions and harmonic weak {M}aass forms.
\newblock {\em Ann. of Math. (2)}, 172(3):2135--2181, 2010.

\bibitem{Bruinier}
Jan~Hendrik Bruinier.
\newblock {\em Borcherds products on {O}(2, {$l$}) and {C}hern classes of
  {H}eegner divisors}, volume 1780 of {\em Lecture Notes in Mathematics}.
\newblock Springer-Verlag, Berlin, 2002.

\bibitem{Bruinier-converse}
Jan~Hendrik Bruinier.
\newblock On the converse theorem for {B}orcherds products.
\newblock {\em J. Algebra}, 397(1):315--342, 2014.

\bibitem{Cremona}
John~E. Cremona.
\newblock {\em Algorithms for modular elliptic curves}.
\newblock Cambridge University Press, Cambridge, second edition, 1997.

\bibitem{Elkies-computation}
Noam~D. Elkies.
\newblock Shimura curve computations.
\newblock In {\em Algorithmic number theory ({P}ortland, {OR}, 1998)}, volume
  1423 of {\em Lecture Notes in Comput. Sci.}, pages 1--47. Springer, Berlin,
  1998.

\bibitem{Elkies-K3}
Noam~D. Elkies.
\newblock Shimura curve computations via {$K3$} surfaces of {N}\'eron-{S}everi
  rank at least 19.
\newblock In {\em Algorithmic number theory}, volume 5011 of {\em Lecture Notes
  in Comput. Sci.}, pages 196--211. Springer, Berlin, 2008.

\bibitem{Errthum}
Eric Errthum.
\newblock Singular moduli of {S}himura curves.
\newblock {\em Canad. J. Math.}, 63(4):826--861, 2011.

\bibitem{Galbraith}
Steven~D. Galbraith.
\newblock {\em Equations For Modular Curves}.
\newblock 1996.
\newblock Thesis (Ph.D.)--University of Oxford.

\bibitem{Molina-isogeny}
Josep Gonz{\'a}lez and Santiago Molina.
\newblock The kernel of {R}ibet’s isogeny for genus three {S}himura curves.
\newblock {\em J. Func. Anal.}, 264(2):508--550, 2013.

\bibitem{Victor-genus-two}
Josep Gonz{\'a}lez and Victor Rotger.
\newblock Equations of {S}himura curves of genus two.
\newblock {\em Int. Math. Res. Not.}, (14):661--674, 2004.

\bibitem{Victor-genus-one}
Josep Gonz{\'a}lez and Victor Rotger.
\newblock Non-elliptic {S}himura curves of genus one.
\newblock {\em J. Math. Soc. Japan}, 58(4):927--948, 2006.

\bibitem{Gross-Zagier}
Benedict~H. Gross and Don~B. Zagier.
\newblock On singular moduli.
\newblock {\em J. Reine Angew. Math.}, 355:191--220, 1985.

\bibitem{Heim-Murase}
Bernhard Heim and Atsushi Murase.
\newblock A characterization of holomorphic {B}orcherds lifts by symmetries.
\newblock {\em Int. Math. Res. Not.}, to appear.

\bibitem{Ihara}
Yasutaka Ihara.
\newblock Congruence relations and {S}him\=ura curves.
\newblock In {\em Automorphic forms, representations and {$L$}-functions
  ({P}roc. {S}ympos. {P}ure {M}ath., {O}regon {S}tate {U}niv., {C}orvallis,
  {O}re., 1977), {P}art 2}, Proc. Sympos. Pure Math., XXXIII, pages 291--311.
  Amer. Math. Soc., Providence, R.I., 1979.

\bibitem{Jordan}
Bruce~Winchester Jordan.
\newblock {\em ON THE DIOPHANTINE ARITHMETIC OF {S}HIMURA CURVES}.
\newblock ProQuest LLC, Ann Arbor, MI, 1981.
\newblock Thesis (Ph.D.)--Harvard University.

\bibitem{Katok}
Svetlana Katok.
\newblock {\em Fuchsian groups}.
\newblock Chicago Lectures in Mathematics. University of Chicago Press,
  Chicago, IL, 1992.

\bibitem{Kudla-Integral}
Stephen~S. Kudla.
\newblock Integrals of {B}orcherds forms.
\newblock {\em Compositio Math.}, 137(3):293--349, 2003.

\bibitem{Kurihara}
Akira Kurihara.
\newblock On some examples of equations defining {S}himura curves and the
  {M}umford uniformization.
\newblock {\em J. Fac. Sci. Univ. Tokyo Sect. IA Math.}, 25(3):277--300, 1979.

\bibitem{Molina-hyperelliptic}
Santiago Molina.
\newblock Equations of hyperelliptic {S}himura curves.
\newblock {\em Proc. Lond. Math. Soc. (3)}, 105(5):891--920, 2012.

\bibitem{Ogg}
Andrew~P. Ogg.
\newblock Real points on {S}himura curves.
\newblock In {\em Arithmetic and geometry, {V}ol. {I}}, volume~35 of {\em
  Progr. Math.}, pages 277--307. Birkh\"auser Boston, Boston, MA, 1983.

\bibitem{Schofer}
Jarad Schofer.
\newblock Borcherds forms and generalizations of singular moduli.
\newblock {\em J. Reine Angew. Math.}, 629:1--36, 2009.

\bibitem{Serre-Stark}
Jean-Pierre Serre and Harold~M. Stark.
\newblock Modular forms of weight {$1/2$}.
\newblock In {\em Modular functions of one variable, {VI} ({P}roc. {S}econd
  {I}nternat. {C}onf., {U}niv. {B}onn, {B}onn, 1976)}, pages 27--67. Lecture
  Notes in Math., Vol. 627. Springer, Berlin, 1977.

\bibitem{Tu}
Fang-Ting Tu.
\newblock Schwarzian differential equations associated to {S}himura curves of
  genus zero.
\newblock {\em Pacific J. Math.}, 269(2):453--489, 2014.

\bibitem{Yang-AIM}
Yifan Yang.
\newblock Defining equations of modular curves.
\newblock {\em Adv. Math.}, 204:481--508, 2006.

\bibitem{Yang-CM}
Yifan Yang.
\newblock Special values of hypergeometric functions and periods of cm elliptic
  curves.
\newblock {\em \tt{http://arxiv.org/abs/1503.07971}}, 2015.

\bibitem{Zhang}
Shouwu Zhang.
\newblock Heights of {H}eegner points on {S}himura curves.
\newblock {\em Ann. of Math. (2)}, 153(1):27--147, 2001.

\end{thebibliography}
\end{document}